\documentclass[12pt,reqno]{amsart}

 \newif\ifdraft
 \drafttrue            

\ifdraft
\usepackage[article]{MetaPackage}
\else
\usepackage{MetaPackage}
\fi

\newcommand{\Cexp}{C_{\mrm{exp}}}
\newcommand{\Csha}{C_{\mrm{sha}}}

\DeclareFunSpace{QCB}{\mathrm{QCB}}
\DeclareFunSpace{QCBt}{\lift{\mathrm{QCB}}}

\newcommand{\varLd}{\varphi_{L^{\dagger}}}

\newcommand{\varL}{\varphi_L}
\newcommand{\Sigmad}{\Sigma^{\dagger}}

\newcommand{\Jacu}[1]{\mrm{J}^u_{#1}}

\title[Bounded Cohomology and Statistics]{Thermodynamic Formalism for Quasimorphisms: Bounded Cohomology and Statistics}

\author{Pablo D. Carrasco}
\address{ICEx-UFMG, Avda. Presidente Antonio Carlos 6627, Belo Horizonte-MG, BR31270-90}
\email{pdcarrasco@ufmg.br}

\author{Federico Rodr\'{i}guez-Hertz}
\address{Penn State, 227 McAllister Building, University Park, State College, PA16802}
\email{hertz@math.psu.edu}

 \date{\today}


\begin{document}

\begin{abstract}
   For a compact negatively curved space, we develop a thermodynamic formalism framework to study the space of quasimorphisms of its fundamental group modulo bounded functions. We prove that this space is Banach isomorphic to the space of Bowen functions on the associated Gromov geodesic flow, modulo a weak form of Livšic cohomology.

  We also show that each unbounded quasimorphism is associated with a unique invariant measure for the flow, which uniquely determines the cohomology class. As a consequence, we establish the Central Limit Theorem and the invariance principle for any unbounded quasimorphism with respect to Markov measures, and we prove that the associated equilibrium state has the Bernoulli property.
    \end{abstract}

\maketitle
\setcounter{tocdepth}{1}
\tableofcontents

\section{Introduction and main results}
\label{sec:IntroandMain}

In this article we study the bounded cohomology of compact spaces of negative curvature, and its counterpart on the corresponding fundamental group. The theory of bounded cohomology was introduced by Johnson in the context of Banach algebras \cite{JohnsonBanach}, but gained maturity with the foundational paper of Gromov \cite{GromovBounded}, where it was extended to topological spaces. For a topological space $M$, one considers the usual singular cochain complex $\{C^n(M;\R),\partial_n\}$ and observes that the boundary map $\partial$ sends bounded cochains in $C^n(M;\R)$ to bounded cochains in $C^{n+1}(M;\R)$. The bounded cohomology of $M$ is the cohomology of the cochain complex $\{C_b^n(M;\R),\partial_n\}$ and is denoted by $H_{\mrm b}^*(M;\R)$. For a discrete group $\Gamma$, its bounded cohomology can be defined as $H^{*}_{\mrm{b}}(\Gamma;\R):=H^{*}_{\mrm{b}}(K(\Gamma,1);\R)$. In addition to its algebraic structure, the vector space $H^*_{\mrm{b}}(M,\R)$ carries a natural seminorm, $\norm[\ell^{\oo}]{[f]}\defeq \inf\set{\norm[\ell^{\oo}]{g}:[g]=[f]}$.

 \paragraph{\textbf{Banach Structure of} $\bm{\Ker c_2}$}Of particular interest is the second bounded cohomology group $H^2_{\mrm{b}}(M;\R)$, and the kernel of the map $c_2: H^2_{\mrm{b}}(M;\R)\to H^2(M;\R)$ induced by the inclusion $C^2_{\mrm{b}}(M;\R)\into C^2(M,\R)$. This map is called the comparison map. If $\pi_1(M,*)$ is (word) hyperbolic then $c_2$ is surjective, and in fact it is a result of Mineyev \cite{Mineyev2002} that hyperbolicity of a group is equivalent to surjectivity of the comparison map replacing $\R$ by arbitrary bounded $G$-modules. Moreover, by the work of Matsumoto and Morita \cite{Matsumoto1985}, and independently Ivanov \cite{Ivanov2017}, it follows that the natural seminorm on $H^2_{\mrm{b}}(M;\R)$ is positive definite, hence $H^2_{\mrm{b}}(M;\R)$ is a Banach space. In this setting however, not much is known about the Banach structure besides the fact that $\Ker c_2$ is infinite dimensional. This was originally established by Brooks and Series \cite{BrooksSeries84} for surface groups (see also Mitsumatsu \cite{Mitsumatsu1984}, Barge and Ghys \cite{Barge1988}), extended for non-elementary word hyperbolic groups by Epstein and Fujiwara \cite{Epstein1997}, and then pushed further for non-elementary countable subgroups of the isometry group of a Gromov hyperbolic metric space, assuming some properness condition of the action (Fujiwara \cite{Fujiwara1998}, then Bestvina and Fujiwara \cite{Bestvina2002}, culminating with Hamenstäedt \cite{Hamenstaedt2008}). We point out that these constructions fall into two (essentially equivalent) types, either they directly construct an infinite linearly independent subset in $\Ker c_2$, or they inject linearly some infinite dimensional vector space. In both, no information is given about the norm. Understanding the Banach structure of $H^2_{\mrm{b}}(M;\R)$ and obtaining a concrete characterization of it is one of the main motivations of this paper.

 It is well known that elements of $\Ker c_2$ can be represented by quasimorphisms of $\Gamma=\pi_1(M,*)$, that is, functions $L:\Gamma\to\R$ satisfying $\norm{\delta L}\defeq\sup_{g,h}|L(g\cdot h)-L(g)-L(h)|<\oo$: see \Cref{sec:QuasiMorphismGroups} for details. For the present discussion it suffices to say that the space $\QM{\Gamma}$ of quasimorphisms of $\Gamma$ is a vector space containing both the bounded functions on $\Gamma$ and the real-valued morphisms of $\Gamma$ (denoted by $\ell^{\oo}(\Gamma), \Hom{\Gamma,\R}$, respectively). Moreover, the seminorm $L\mapsto \norm{\delta L}$ induces a norm on 
\[
 \frac{\QM{\Gamma}}{\Hom{\Gamma,\R}\oplus\ell^\oo(\Gamma)}, 
\]
 and with the latter this space is isometric to $\Ker(c_2)$. 

\paragraph{\textbf{Unbounded Cohomology}}It will be convenient (and more natural) to ignore the purely algebraic part $\Hom{\Gamma,\R}$ and work with 
 \[
  \QMt{\Gamma}\defeq\frac{\QM{\Gamma}}{\ell^\oo(\Gamma)}.
 \]

\begin{definition}\label{def:unboundedqm}
	The space $\QMt{\Gamma}$ will be referred to as the space of unbounded quasimorphisms on $\Gamma$.
\end{definition}

The seminorm on quasimorphisms does not induce a norm on $\QMt{\Gamma}$, hence some care is needed. We proceed as follows: assume that $\Gamma$ is finitely generated (which is the case of interest for us), fix some finite generating set $S$, and define 
\[
	\norm{L} = \norm[\ell^{\oo}]{L|S}+\norm{\delta L}.
\]
This defines a complete norm on $\QM{\Gamma}$, which in turn induces a Banach structure on $\QMt{\Gamma}$, and moreover induces the same norm on \(\Ker c_2=\frac{\QMt{\Gamma}}{\Hom{\Gamma,\,\R}}\) as before. This does not seem to be readily available in the literature, so we include a discussion in \Cref{sec:QuasiMorphismGroups}.

\paragraph{\textbf{Unbounded classes as probabilities}}In this article (and its sequel) we propose a different way to understand classes of (un)bounded cohomology. We will show that for some natural classes of groups $\Gamma$, each $\alpha\in\QMt{\Gamma}$ is uniquely determined by a probability measure which is related to a natural dynamical system associated to the group. In this article we focus on the hyperbolic case.

Our results apply to compact negatively curved spaces and some of them even apply to non-elementary hyperbolic groups, but since the formulation of the theorems in these contexts requires the introduction of additional dynamical concepts, we restrict ourselves in the introduction to the following case. Let $M$ be a $d-$dimensional closed hyperbolic manifold, and let $\lie{g}=(g_t)_{t\in\R}:X=T^1M\to X$ be the geodesic flow associated with the metric.

The aforementioned probability measures in this case turn out to be equilibrium measures for some non-regular potentials. To define these potentials we fix some reference measure on $X$ which is invariant under $\lie{g}$. For concreteness, we work with $\mu_X$, the Liouville probability measure on $X$.

 \begin{definition}\label{def:BowenfunctionT1X}
  A Borel function $\phi \in \Lp{\oo}{\mu_X}$ is a weak Bowen function if there exist $C,\eps>0$ so that $(\aee{\mu_X}\ x,y\in X,\forall T> 0):$
 \[
 	 \sup_{t\in [0,T]}\dis[X]{g_tx}{g_ty}\leq \eps\Rightarrow \abs*{\int_0^T \phi(g_tx)\dd t-\int_0^T\phi(g_ty)\dd t }\leq C. 
 \]
 The Bowen constant of $\phi$ is $\norm[B]{\phi} = \inf C$.
 \end{definition}  

 Functions satisfying this condition were introduced by Bowen in his studies of the thermodynamic formalism of hyperbolic systems \cite{Bowen1974}, and received a lot of attention since then. See e.g. \cite{EquSta}, \cite{Climenhaga2012} and references therein. Bowen's definition, however, deals with (typically continuous) functions defined everywhere, and not almost everywhere with respect to some fully supported measure. This difference will be key in what follows.

We denote by $\Bow[][\mu_X]{X}$ the vector space of weak Bowen functions for some a priori fixed $\eps>0$. Note that since $\lie{g}$ is an Anosov flow \cite{AnosovThesis}, it follows without much trouble that $\Bow[][\mu_X]{X}$ contains the space of Hölder continuous functions on $X$.

\paragraph{\textbf{Livšic cohomology}}It will be shown that each $\varphi\in \Bow[][\mu_X]{X}$ has an associated $\lie{g}-$invariant probability measure $\mu_\varphi$, and furthermore $\mu_\varphi=\mu_{\varphi'}$ if and only if $\varphi,\varphi'$ satisfy a cohomology relation that we define below. Ultimately, classes in $\QMt{\Gamma}$ will correspond to cohomology classes of weak Bowen functions.

 \begin{definition}\label{def:LivsiccohomologousLiou}
 We say that $\phi,\psi\in \Bow[][\mu_X]{X}$ are Livšic cohomologous ($\phi\sim\psi$) if $\exists u\in \Lp{\oo}{X,\mu_X}$ which is differentiable along orbits of $\lie{g}$, such that $\phi-\psi = \frac{\dd u\circ g_t}{\dd t}\vert_{t=0}$.  
 \end{definition}

The above definition is well known in ergodic theory, and its importance owes much to the work of Livšic \cite{Livshits1971}, especially for regular (Hölder) functions. An in-depth analysis of this equivalence relation for the general case that we are discussing is carried out in \Cref{sec:quasimorphismsSFT}. In particular, one has that $\Bow[][\mu_X]{X}/\thicksim$ is naturally a Banach space with respect to the norm

 \begin{align}
 \tnorm[B]{[\phi]} = \inf\{\norm[L^{\oo}]{\psi}+\norm[B]{\psi}:\psi\sim\phi\}.
 \end{align}

 There is a natural involution $I:X\to X$ which reverses the direction of geodesics. We say that $\phi \in \Bow[][\mu_X]{X}$ is antisymmetric if $\phi\circ I$ is Livšic cohomologous to $-\phi$, and denote by
\[
	\Bow[a][\mu_X]{X}
\]
the set of antisymmetric weak Bowen functions. 

Antisymmetry arises naturally in our description of quasimorphisms. In contrast with \cite{Hamenstaedt2008}, where a related anti-flip condition appears as part of a specific construction, it serves here to characterize the image of the correspondence. This is formalized in the following theorem.

\begin{maintheorem}\label{thm:A}
	Let $M$ be a closed hyperbolic manifold with fundamental group $\Gamma=\pi_1(M,*)$. Then there is a Banach isomorphism 
\[
	\Psi:\QMt{\Gamma}\to \Bow[a][\mu_X]{X}/\thicksim.
\]
\end{maintheorem}

\begin{remark}\label{rem:othermeasures}
	We say that a probability measure $\mu \in \ProbM[X]$ is invariant under the flow $\lie{g}=(g_t)_{t\in \R}$ if $\forall t\in\R, \mu=\mu\circ g_{t}$.  The set of invariant measures under $\lie{g}$ is denoted by $\PTM{\lie{g}}{X}$.

	The above theorem remains true if $\mu_X$ is replaced by any other invariant measure $\mu\in\PTM{\lie{g}}{X}$ of full support.
\end{remark}

\paragraph{\textbf{Concrete representation of} $\bm{\Ker c_2}$}

This implies almost directly that $\Ker c_2$ has uncountable dimension (a fact already known from the previously cited works), but more importantly, it elucidates its Banach space structure. Let us recall some alternative characterization due to Monod and Shalom \cite{Monod2004}, based on the seminal work of Burger and Monod \cite{Burger2002}. 

A standard Borel $\Gamma$-space $(B,\nu)$ (where $\nu$ is quasi-invariant) is a strong boundary of $\Gamma$ if
\begin{itemize}
	\item $\Gamma\acts B$ is amenable in the Zimmer sense \cite{semisimpleZimmer},
	\item if $E$ is a separable coefficient $\Gamma-$module, then every measurable $\Gamma-$invariant map $\Gamma\times B\to E$ is $\nu-$essentially constant.
\end{itemize}

\begin{theorem}[Corollary $2.6$ in \cite{Monod2004}]
	If $\Gamma$ is a countable discrete group and $(B,\nu)$ is a strong boundary for it, then there exists a canonical isomorphism $H^2_{\mrm{b}}(\Gamma)\cong Z\Lp{\oo}_{\mrm{a}}\paren{B^3}^{\Gamma}$, where the right-hand side is the space of $\Gamma-$invariant, alternating $\nu-$essentially bounded real-valued $2-$cocycles on $B$.  
\end{theorem}

For $\Gamma=\pi_1(M,*)$ with $M$ compact hyperbolic manifold, one can choose $B=\partial\lift{M}$, identified as the Poisson boundary associated to some non-degenerate random walk on $\Gamma$, say for a driving measure with finite first moment \cite{Kaimanovich2000}.

The equivalence $H^2_{\mrm{b}}(\Gamma)\cong Z\Lp{\oo}_{\mrm{a}}\paren{\partial M^3}^{\Gamma}$, although conceptually important, replaces $H^2_{\mrm{b}}(\Gamma)$ by another complicated Banach space, which is difficult to deal with. In contrast, the equivalence given in \Cref{thm:A} replaces $2-$cocycles with equivalence classes of functions, which we will show to be  more manageable objects.  The principal tool for this is a measurable version of Livšic theorem which we prove in \Cref{sec:LivsicSFT}, stating that Livšic cohomology classes of elements in $\Bow[][\mu_X]{X}$ are completely determined by their values on closed geodesics. To be more precise, each closed geodesic $\ga$ in $X$ defines a linear functional $P_{\ga}:\Bow[][\mu_X]{X}\to \R$, and $\varphi\sim\phi$ if and only if for every $\ga$ one has $P_{\ga}(\varphi)=P_{\ga}(\phi)$. For a continuous weak Bowen function $\varphi$, $P_\ga(\varphi)$ coincides with the average of $\varphi$ along $\ga$. However, for general weak Bowen functions the theorem is much more delicate, since in principle $\varphi$ is only an essentially bounded function on $X$, and $\mu_X(\ga)=0$, so one cannot evaluate $\varphi$ along the curve. 

 We point out that the classical version of Livsic cohomology (for H\"older potentials) is insufficient for our purposes, as there are examples of continuous functions that are not continuously cohomologous (meaning, with continuous transfer function $u$) to any Hölder function. See \Cref{rem:onesidedvstwo}.

\paragraph{\textbf{All (un)bounded cohomology classes}} Some authors have previously obtained related  characterizations of $\Ker c_2$ using dynamical systems. The first was probably Picaud in the setting where $M$ is a hyperbolic surface \cite{Picaud1997}, who proved that one can embed $\Ker c_2$ into the space of $\Gamma$-invariant measures on $\partial M$, or alternatively, into the space of geodesic currents in the sense of Bonahon and Sullivan. Then, using symbolic dynamics for the action of $\Gamma$ on $\partial M$ \cite{Bowen1979}, he characterized some regular quasimorphisms\footnote{Essentially, homogeneous quasimorphisms that vanish on the set of generators of $\Gamma=\pi_1(M,*)$} as H\"older functions on a subshift of finite type. The arguments rely heavily on the two-dimensionality of $M$ (although they also work for surfaces of finite type).

However, the restriction on the class of quasimorphisms considered is more serious. On the one hand, there are no general methods available to check whether a given quasimorphism corresponds to a Hölder function, while on the other hand not every quasimorphism corresponds to such functions. 

Restricting to H\"older classes of objects also appears in the work of Hamenstäedt \cite{Hamenstaedt2008}, who showed that (antisymmetric) H\"older functions endowed with the seminorm $\norm[B]{\cdot}$ continuously embed into $\QMt{\Gamma}$, and in other places e.g. \cite{Horsham2009}.

\Cref{thm:A} makes it possible to consider all quasimorphisms. We will apply this to obtain precise statistical properties of these objects.

\paragraph{\textbf{Unbounded quasimorphisms as invariant measures}} The philosophy that drives the previous result is that in fact elements of $\QMt{\Gamma}$ can be understood as a certain type of $\lie{g}-$invariant measures, namely equilibrium states for weak Bowen functions. The perspective of treating quasimorphisms as dynamically defined measures appears to be new, at least in explicit form (see, however, \cite{Picaud1997}), and proves to be useful. In \cite{thermodynamiclattice}, we adopt this approach to provide a new proof of a beautiful result of Burger and Monod, which states that if $\Gamma$ is a uniform lattice in a real semisimple Lie group of non-compact type and rank at least $2$, then $\QMt{\Gamma}={0}$.

Each $\mu \in \PTM{\lie{g}}{X}$ defines a positive linear functional on $\Bow[][\mu_X]{X}$, which extends the usual integral for continuous functions with the Bowen property, and therefore is denoted in the same manner. Below, if $\mu\in\PTM{\lie{g}}{X}$, then $\hmu(\lie{g})=\hmu(g_1)$ denotes the metric entropy of the time-one map $g_1$ with respect to $\mu$.

\begin{maintheorem}\label{thm:B}
Let $M$ be a closed hyperbolic manifold. For each $[\phi]\in  \Bow[][\mu_X]{X}/\thicksim$ there exists a measure $\mu_{[\phi]}=\mu_{\phi} \in \PTM{\lie{g}}{X}$ which can be characterized as the unique element of this space satisfying
\[
	\hmu[\mu_{\phi}]{\lie{g}}+\int \phi\dd \mu_{\phi}=\sup_{\mu\in \PTM{\lie{g}}{X}}\paren*{\hmu[\mu]{\lie{g}}+\int \phi\dd \mu}. 
\]
For every $t\neq 0$, the stochastic process defined by $(\phi\circ g_{nt})_{n\in \N}: X\to\R$ is measure-theoretically isomorphic to a Bernoulli process with finitely many symbols.
\end{maintheorem}

We will show that $\mu_{[\phi]}$ uniquely determines $[\phi]\in \Bow[][\mu_X]{X}/\thicksim$, hence by \Cref{thm:A}, the unique associated unbounded quasimorphism (provided that $\phi$ is antisymmetric).

\begin{remark}
	Again, one can replace the Liouville measure by any other fixed invariant measure of full support. 
\end{remark}

The theory of thermodynamic formalism for non-continuous potentials was first developed by Haydn and Ruelle in \cite{Haydn1992} (see also Ruelle's article \cite{Ruelle_1992}). These works are mainly concerned with the existence of dominant eigenfunctions and eigenmeasures for the corresponding transfer operator, but finer (as well as cohomological) properties of the systems were not addressed. A plausible reason for this is that the arguments involve some arbitrary extensions of the potential, which makes further investigations difficult. Here we give a different concrete construction for weak Bowen functions, but of course we have been inspired by these previous works.

\paragraph{\textbf{Central Limit Theorem for quasimorphisms}} Using the characterization given by \Cref{thm:A} we are able to show that any non-trivial element of $\Ker(c_2)$ has a very rich dynamical behavior.

\begin{maintheorem}\label{thm:C}
Let $M$ be a closed hyperbolic manifold. Then there exists a Markov measure $\mu\in \PTM{\lie{g}}{X}$ such that the following is true. For any $0\neq [\phi]\in \Bow[][\mu]{X}/\thicksim$ there exists a positive constant $\sigma^2=\sigma^2([\phi])$ so that for $\E{[\phi]}=\E{\phi}=\int \phi \dd \mu$, it holds
\[
     \forall c\in\R,\ \mu\paren*{x\in X:\frac{\int_0^T\phi(g_tx)\dd t-T\E{\phi}}{\sigma\sqrt{T}}\leq c}
      \xrightarrow[T\to\oo]{}\frac{1}{\sqrt{2\pi}}\int_{-\oo}^c e^{-\frac{u^2}{2}}\dd u
\]
\end{maintheorem}

Central Limit Theorems for quasimorphisms were first studied by Calegari and Fujiwara in \cite{CALEGARI2009}, and have been thoroughly studied ever since. For a gentle introduction we refer to the notes of Calegari \cite{calegaryhyp}. See also Horsham and Sharp \cite{Horsham2009}, and the more recent article of Cantrell \cite{Cantrell2020}. 

Unlike \Cref{thm:B}, the cited limit theorems are given with respect to some appropriate exit measure for a random walk in the group (for example, the Patterson-Sullivan measure). But more importantly, the most general version \cite{Cantrell2020} assumes a strong form of regularity, essentially that the quasimorphism corresponds to a H\"older continuous function on the subshift associated to some Strong Markov structure of $\Gamma$.

To treat all unbounded cohomology classes of quasimorphisms we prove a general version of the Central Limit Theorem for discrete Markov chains, which to our knowledge is new even in the independent case. The novelty is that we can consider highly non-regular processes given by measurable functions on the corresponding subshift (weak Bowen functions). Observe that in this level of generality even the convergence of the variance for such functions is non-trivial. On the other hand, the method does not use any spectral gap property for the action of the Ruelle-Perron-Frobenius operator on the space of Bowen functions (which is probably not true), and we do not obtain any Barry-Esseen type bound.

\paragraph{\textbf{Random walks and quasimorphisms}} Let us state the previous theorem directly in the group context. For a closed hyperbolic manifold $M$, write as above $\Gamma=\pi_1(M,*)$, and let $\mrm{Conj}(\Gamma)=\Gamma/\Inn(\Gamma)$ denote its set of conjugacy classes. Equip $\Gamma$ with a word metric associated to some finite set of generators $S$. $\Gamma$ is word hyperbolic, therefore any $\gamma\in\Gamma$ can be represented as a word in $S$ of minimal length, and two such words differ by a cyclic permutation. For our purposes it is more convenient to work with $\mc{R}(\Gamma;S)$, the set of reduced cyclic words of $\Gamma$ in $S$, instead of $\mrm{Conj}(\Gamma)$. If $w\in \mc{R}(\Gamma;S)$ let $\abs[S]{w}$ denote its length.

We say that $L,L'\in\QM{\Gamma}$ are cohomologous if $[L]=[L']\in\QMt{\Gamma}$. Every quasimorphism $L$ is cohomologous to a class function $\cl{L}$: as a result, the asymptotic behavior of $L$ on conjugacy classes can be studied via the corresponding analysis of $\cl{L}:\mc{R}(\Gamma;S)\to\R$. Given $n\in\N$ let $B_n\defeq\{w\in \mc{R}(\Gamma;S): \abs[S]{w}\leq n\}$, which is a finite set, and consider the probability measures 
 \[
 	\nu_n^{\Gamma,S}:= \frac{1}{\# B_n}\sum_{w\in B_n}\delta_{w}.
 \]
Ideally, one would like to use a limiting distribution of $(\nu_n^{\Gamma,S})_n$ to study the statistical properties of $\bar{L}$, but since $\mc{R}(\Gamma;S)$ lacks sufficient structure, it is necessary to enlarge it in order to find this distribution.

 \begin{theorem}[Compactification of conjugacy classes]\label{thm:compactification}
 There exists a compact metric space $X_{\Gamma}$, a filtration of finite $\sigma-$algebras $(\mc{F}_n)_{n}$ of $X_{\Gamma}$, an injective map $\Phi:\mc{R}(\Gamma;S)\to X_{\Gamma}$, a constant $D\geq 0$ and a probability $\mu^{\Gamma}\in \ProbM[X_{\Gamma}]$ satisfying:
 \begin{enumerate}
 	\item $\Im(\Phi)$ is dense in $X_{\Gamma}$;
 	\item $\Phi(w)\in \mc{F}_{D+|w|_S}$;
	\item the sequence $(\mu_n^{\Gamma,S}=\nu_n^{\Gamma,S}\circ \Phi^{-1})_{n\in\N}$ converges weakly to $\mu^{\Gamma}$.
 \end{enumerate}
 \end{theorem}

We remark that $\mu_n^{\Gamma,S}$ is a probability measure on $\mc{F}_{D+n}$. If $L\in\QM{\Gamma}$, one first considers the induced function  $\bar{L}:\mc{R}(\Gamma;S)\to\R$, and then uses $\Phi$ as above to induce a sequence of continuous functions $(L^{(n)}:X_{\Gamma}\to \R)_{n}$ which verifies
\begin{enumerate}
 	\item $L^{(n)}(\Psi(w))=\cl{L}(w)$, if $\abs[S]{w}=n$;
     \item $L^{(n)}$ is $\mc{F}_n-$measurable.
 \end{enumerate}
 See \ref{subsec:representationthm} for details. The measure $\mu^{\Gamma}$ does not depend on chosen generating set $S$. On the other hand, the construction of $(L^{(n)})_n$ is not canonical, but can be achieved so that two of such sequences are at uniformly bounded distance from each other. Thus, with no loss of generality, we assume that $L$ itself is a class function.

\setcounter{corollaryL}{2}
 \begin{corollaryL}\label{cor:C}
 Let $M$ be a closed hyperbolic manifold with fundamental group $\Gamma$. If $L\in\QM{\Gamma}$ is unbounded, then there exist constants $\sigma^2=\sigma^2(L)>0, e(L)\in\R$ so that for every $c\in\R$,
 \[
 	\mu^{\Gamma}\paren*{x\in X_\Gamma:\frac{L^{(n)}(x)-ne(L)}{\sigma\sqrt{n}}\leq c}\xrightarrow[n\to\oo]{} \frac{1}{\sqrt{2\pi}}\int_{-\oo}^c e^{-\frac{u^2}{2}}\dd u
 \]
 \end{corollaryL}

In \cite{biharmonic2011} Björklund and Harnick prove a very general Central Limit Theorem for quasimorphisms on locally countable groups when the sequence of words is chosen by a random walk whose driving measure $\mu$ is spread-out. When the corresponding stationary measure on $\partial\Gamma$ is the Patterson-Sullivan measure, the results in \cite{Connell2007} imply that the CLT holds for every quasimorphism with respect to $\mu$.

However, these results concern distributions arising from random walks on the group. In contrast, the measure $\mu^{\Gamma}$ considered here, although reminiscent of the Patterson-Sullivan measure, is not associated to a random walk on $\Gamma$. Instead, it arises as a uniform measure on a natural compactification of $\mrm{Conj}(\Gamma)$. The two settings are conceptually different and consequently, the results of \cite{biharmonic2011,Connell2007} do not imply \Cref{cor:C}. The latter permits one to study the distribution of quasimorphisms on conjugacy classes, which provides a natural framework for the problem.

More generally, the results above show that quasimorphisms admit a natural dynamical 
and probabilistic interpretation, allowing their statistical properties to be studied 
using tools from thermodynamic formalism.

\paragraph{\textbf{Functional limit theorem, Law of iterated logarithm and large deviations}}

 Other limit theorems follow from our approach. We highlight the invariance principle, the law of the iterated logarithm, and concentration inequalities, giving a unified treatment of these results in full generality. 

 If $(J_n)_{n\geq 1}$ is a stationary sequence of zero mean in $(X_{\Gamma},\mu^{\Gamma})$, let $S_n=\sum_{k=1}^n J_k$ and consider the random element $\lie{L}: \N\times X_{\Gamma}\to \mc{C}([0,1])$ 
\[
 	\lie{L}_t(n,x)=\left(\lie{L}(n,x)\right)(t)=\frac{S_{([nt])}(x)+(nt-[nt])J_{[nt]+1}(x)}{\sigma \sqrt{n}}\quad t\in [0,1].
\]
Let $\eta_n=\mu^{\Gamma}\circ (\lie{L}(n,\cdot))^{-1}$ be its distribution. Define
 \[
 	\mc{K}=\set{x\in \mc{C}([0,1]): x(0)=0, x\text{ is absolutely continuous, and }\int_0^1 \dot{x}^2(t)\dd t\leq 1}.
 \]
For $e<t<+\oo$ denote $\phi(t)=\sqrt{2t\log\log t}$.

 \begin{maintheorem}\label{thm:D}
 Under the hypothesis of the previous corollary, suppose that $L$ is unbounded with $e(L)=0$. Then, there exists a stationary ergodic sequence $(J_n)_{n\geq 1}$ so that
 \begin{enumerate}
 	\item $\sup_{n\geq1, x\in X_{\Gamma}}|S_n(x)-L^{(n)}(x)|<\oo$;
 	\item the sequence of distributions $(\eta_n)_n$ converges weakly to the Wiener measure. That is, $(\lie{L}_t(n,\cdot))_n$ converges in distribution to the standard Brownian process.
 	\item For $\aee{\mu^{\Gamma}}\ x$, the sequence $\paren*{\frac{\lie{L}(n,x)}{\phi(n\sigma^2(L))}}_{n\geq \frac{e}{\sigma^2(L)}}$ is relatively compact, and each of its limit points belongs to $\mc K$.
     \item For every $0\leq\delta\leq 1$ there exists $H(\delta)\in (0,1], C>0$ such that for every $n\geq 0$, 
     \begin{align*}
     &\mu^{\Gamma}\left(\frac{L^{(n)}}{n}\geq \delta\right)\leq e^{-nH(\delta)}\\
     &\mu^{\Gamma}\left(\frac{L^{(n)}}{\sqrt{n}}\geq \delta\right)\leq e^{-C\delta^2}.
     \end{align*} 
 \end{enumerate}
 \end{maintheorem}

As in \Cref{thm:C}, this will be a consequence of the corresponding statements for weak Bowen functions. We remark that for regular (edge-combable) quasimorphisms \cite{Cantrell2020} gives a large deviation theorem and, although not explicitly stated, the results in \cite{Cantrell_2023} can be adapted to give the invariance principle and the LIL for these regular quasimorphisms.

\paragraph{\textbf{Asymptotics with respect to the Patterson-Sullivan measure}} When $\Gamma$ is a non-elementary Gromov-hyperbolic group, as in the case that we are discussing, there exists a symbolic coding of $\partial \Gamma$ (the so-called Cannon coding), which allows one to lift the Patterson-Sullivan measure to a Markov measure. It was kindly pointed out to us by S. Cantrell that, combining this with the technique developed for proving \Cref{thm:C}, one obtains remarkable corollaries about the asymptotics of any unbounded quasimorphism with respect to the Patterson-Sullivan measure, and with respect to averaging on spheres in the Cayley graph.

 For $n\in\N$ denote $S_n=\{g\in\Gamma:\abs[S]{g}=n\}$ the sphere of radius $n$. For a ray $r \subset \mrm{Cay}(\Gamma)$ we write $r_n\in S_n$ for the initial segment of size $n$. The class determined by $r$ on $\partial \Gamma$ is denoted by $[r]$. The identity of $\Gamma$ is $1_{\Gamma}$.

\begin{corollaryL}\label{cor:D}
Let $\Gamma$ be a non-elementary Gromov-hyperbolic group, and let $\nu \in \ProbM[\partial \Gamma]$ be the Patterson-Sullivan measure. If $L\in\QM{\Gamma}$ is unbounded, then there exist $\sigma^2=\sigma^2(L)>0$ such that for every $c\in\R$,
\begin{align*}
 &\nu\left([\tilde r]:\exists r\in [\tilde r]\text{ with }r_0=1_{\Gamma}\text{ and } \frac{L(r_n)}{\sigma\sqrt{n}}\leq c\right)\xrightarrow[n\mapsto\oo]{} \frac{1}{\sigma\sqrt{2\pi}}\int_{-\oo}^c e^{-\frac{u^2}{2\sigma}}\dd u\\
 &\frac1{\# S_n}\#\left\{g\in S_n:\frac{L(g)}{\sigma\sqrt{n}}\leq c\right\}\xrightarrow[n\mapsto\oo]{} \frac{1}{\sigma\sqrt{2\pi}}\int_{-\oo}^c e^{-\frac{u^2}{2}}\dd u.
 \end{align*}
\end{corollaryL}

We remark that this corollary does not follow from \cite{biharmonic2011}, which concerns limit theorems for quasimorphisms with respect to the driving measure of a random walk on the group. In contrast, \Cref{cor:D} describes the asymptotic behavior with respect to the Patterson-Sullivan measure on the boundary.

\paragraph{\textbf{On the definition of Dynamical System}}

At the end of this introduction, we would like to draw the reader's attention to some interpretations of the meaning of the previous theorems from the optics of dynamical systems. The notion of classical dynamical systems, used by Poincaré, Gibbs, Boltzmann and many others at that time, consists of (at least) two main components: the law of evolution, which is modeled for example by some flow $(g_t:M\to M)_{t\in\R}$ on the phase space, and an observable, which is represented by some function $\phi:M\to\R$. The evolution is in principle unknown, and the only possible interaction with the physical system is to take measurements at discrete times. To smooth out small fluctuations, it is more natural to record $\phi^{(n)}:M\to\R$, the cumulative sum of the measurements up to time $n$, and then take the average. In this way, what one actually gets is a sequence $\{\phi^{(n)}:M\to\R\}_{n\geq 0}$ which, assuming that measurement errors remain bounded (otherwise no relevant information is extracted), satisfies some condition of the form 
\[ 
 \sup_{n,m}|\phi^{(n+m)}-\phi^{(n)}-\phi^{(m)}\circ g_n|<C<\oo.
 \] 
 In this interpretation, quasimorphisms play the role of approximate observables whose cumulative measurements satisfy the bounded error relation above. What \Cref{thm:A,thm:B,thm:C} and \Cref{thm:D} say is that if the driving evolution is sufficiently chaotic, then the system has an error-correcting mechanism that allows the observable $\phi$ to be recovered, at least for almost every point, and in such a way that the empirical measurements remain at a bounded distance from the sums $\{\sum_{k=0}^n \phi\circ g_k\}_{n\geq 1}$. This potential is essentially determined by some variational principle for the measurements. Moreover, as one would expect from reality, the measurements are normally distributed, provided one averages with respect to some well-behaved law.

\subsection*{Acknowledgements}

The authors would like to thank S. Cantrell for several suggestions  and references, in particular for telling us about \Cref{cor:D}. We also thank M. Burger, C. Dilsavor, L. Flaminio, F. Ledrappier, E. Pujals, D. Thompson and B. Pozzetti for their constructive comments that improved this work.

\section{Geodesic flow on negatively curved spaces and its symbolic representation} 
\label{sec:geodesic_flow_on_negative_curved_spaces_and_its_symbolic_representation}

In this section, we review some dynamical concepts related to geodesic flows and explain their symbolic structure. We also state analogues of the theorems from the introduction in this more general setting.

\subsection{Geodesic flows} 
\label{sub:geodesic_flows}

 Let $M$ be a closed (connected) geodesic space which is negatively curved in the Alexandrov sense. Its universal covering $\lift{M}$ is therefore a proper $\CAT$ space and hence contractible. The group $\Gamma=\pi_1(M,*)$ acts by isometries on $\lift{M}$ by deck transformations. Thus, we can (and do) identify $M=\lift{M}/\Gamma$. By the Milnor-Švarc theorem, $\Gamma$ is finitely generated and quasi-isometric to $\lift{M}$, when equipped with any word metric associated to a finite generating set. In particular, $\Gamma$ is word hyperbolic. We adopt the following nomenclature, which appears in \cite{Constantine2020}.

\begin{definition}\label{def:locallyCAT}
$M$ is referred to as a closed locally $\CAT$ space.
\end{definition}

The limit set of $\Gamma$ on $\lift{M}$ can be naturally identified with its geometric boundary $\partial \Gamma\cong \partial \lift{M}$, and we denote 
\[
	\mc{G}:=\{c:\R\to \lift{M}: c\text{ is a geodesic}, c(-\oo), c(+\oo)\in\partial \Gamma\}.
\]
$\mc{G}$ is equipped with the topology of uniform convergence on compact subsets; this topology is metrizable. On $\mc G$ there are three natural actions, namely:
\begin{enumerate}
	\item $\Gamma$ acts cocompactly by isometries with $\gamma\cdot c=\gamma\circ c$;
	\item $\Z_2$ acts by an involution, $I:c(t)\mapsto c(-t)$;
	\item $\R$ acts by $s\cdot c(t):=c(t+s)$.
\end{enumerate}

Let $\mc{E}:=\mathcal{G}/\Gamma$. Since the $\Gamma-$ and $\R-$ actions commute, there exists an induced $\R$ action on $\mc{E}$. We denote by $\lie{g}=(g_t)_{t\in\R}:\mc E\toit$ the corresponding flow. Likewise, the $\Z_2$ action induces an action on $I:\mc{E}\toit$, with $I(g_t)=g_{-t}, \forall t\in\R$.  

\begin{definition}\label{def:geodesicflow}
$\lie{g}$ is the geodesic flow associated to $\Gamma$.
\end{definition}

The above flow was introduced by Gromov \cite{GromovHypGr} for hyperbolic groups; further details can be found in \cite{Champetier1994,Mineyev2005}. In the setting of a negatively curved  space, the flow is better behaved and can be seen as a direct analogue of the (differentiable) geodesic flow on a hyperbolic manifold. See Bourdon's article \cite{Bourdon1995}. We note that the above construction is by no means canonical, and the resulting flow is susceptible to reparametrization. However, Gromov proved that two different geodesic flows associated with $\Gamma$ are equivalent in the following sense.

\begin{definition}\label{def:equivalentflows}
If $X, Y$ are metric spaces equipped with flows $\phi^X=(\phi_t^X)_t:X\toit,\ \phi^Y=(\phi_t^Y)_t:Y\toit$ we say that a homeomorphism $h:X\to Y$ is an equivalence between the flows if it sends orbits of $\phi^X$ onto orbits of $\phi^Y$. In this case the flows are said to be equivalent.
\end{definition}
By the above, we will refer to any equivalent flow to $\lie{g}$ as a geodesic flow.

\begin{remark}\label{rem:differentiablecase}
For $M$ hyperbolic manifold the space $\mc{E}$ is homeomorphic to its unit tangent bundle, and $\lie{g}$ is flow equivalent to any geodesic flow on $T^1M$ defined by a smooth Riemannian metric.
\end{remark}

In \cite{Constantine2020}, the authors show that for locally $\CAT$ spaces one can find a geodesic flow associated to $\Gamma$ that is also a topologically mixing Smale flow (also called a metric Anosov flow). For our purposes it will suffice to note three dynamical consequences: expansitivity, shadowing and symbolic representation.

\begin{theorem}[Expansivity]\label{thm:expansivity}
There exists a geodesic flow $\lie{g}=(g_t)_{t\in\R}$ and a constant $\Cexp>0$ such that 
\[
	\sup_{t\in\R}\dis[\mc E]{g_t(x)}{g_t(y)}\leq \Cexp\Rightarrow x=g_{t_0}(y)\quad |t_0|<\Cexp. 
\]
\end{theorem}

\begin{notation}
A geodesic $\alpha:\R\to \mc{E}$ is periodic if there exists $T>0$ so that $\alpha(t+T)=\alpha(t), \forall t\in\R$. The smallest of such $T$ is called the period of $\alpha$ and is denoted by $\per(\alpha)$. 
\end{notation}

\begin{theorem}[Shadowing]\label{thm:shadowing}
The geodesic flow given in the previous theorem also satisfies: given $\delta>0$ there exists $\Csha=\Csha(\delta)>0$ so that if $\alpha_1, \alpha_2$ are periodic orbits there exists a periodic orbit $\alpha_3$ such that
\begin{align*}
& \per(\alpha_3)\leq\per(\alpha_1)+\per(\alpha_2)+2\Csha\\
& \sup_{t\in [0,\per(\alpha_1)]}\set{\dis[\mc E]{\alpha_1(t)}{\alpha_3(t)}}<\delta, \sup_{t\in [0,\per(\alpha_2)]}\set{\dis[\mc E]{\alpha_2(t)}{\alpha_3(t+\Csha)}}<\delta.
\end{align*}
\end{theorem}

The proof of both theorems for Smale flows is the same as for smooth hyperbolic flows, which is standard. See for example Chapter $18$ in \cite{KatHass}. From now on we work with a parametrization $\lie{g}$ satisfying both \Cref{thm:expansivity,thm:shadowing}. Symbolic representation is discussed below.

\begin{notation}
We fix $\delta_0>0$ and write $\Csha$ for the corresponding constant. If $\alpha_1, \alpha_2$ are periodic orbits then we write $\alpha_3=\alpha_1\star\alpha_2$ for any periodic orbit $\alpha_3$ satisfying the conclusion of the previous theorem.
\end{notation}

 \subsection{Weak Bowen Functions and their Livšic cohomology} 
 \label{sub:WeakBowenLivšic}

Recall that $\PTM{\lie{g}}{\mc E}$ denotes the invariant (probability) measures for the flow $\lie{g}$. Given $x\in \mc E$ and $T>0$ we denote
\begin{align}\label{def:Birkhofsum}
S_T\phi(x):=\int_0^T \phi(g_tx)\dd t.
\end{align}

Fix $0<\eps_0<\frac{\Cexp}{4}$. 

\begin{definition}\label{def:weakBowenfunction}
Let $\mu\in \PTM{\lie{g}}{\mc E}$ be an invariant probability measure of full support. We say that $\phi \in \Lp{\oo}{\mc E}$ is a $\mu-$weak Bowen function if there exist $C>0$ so that 
\[
	(\aee{\mu}\ x,y\in \mc E, \forall T>0): \sup_{t\in [0,T]}\dis[\mc E]{g_tx}{g_ty}\leq \eps_0\Rightarrow \abs*{S_T\phi(x)-S_T\phi(y)}\leq C. 
\]
The Bowen constant of $\phi$ is $\norm[B]{\phi}=\inf C$, and its Bowen norm is $\tnorm[B]{\phi}=\norm[L^{\oo}]{\phi}+\norm[B]{\phi}$. The space of $\mu-$weak Bowen functions is denoted by $\Bow[][\mu]{\mc E}$.
 \end{definition}

Let us discuss this norm in more detail.

\begin{notation}
	If $x,y\in \mc E, T>0$ we write
	\[
	\dis[\mc{E},T]{x}{y}=\sup_{t\in [0,T]}\dis[\mc E]{g_tx}{g_ty}
	\]
	and
	\[
	B(x,\eps_0,T)=\{y:\dis[\mc{E},T]{x}{y}\leq \eps_0\};
	\]
	these are the $T-$Bowen distance and the $(\eps_0,T)$-Bowen ball, respectively.
\end{notation}

 Thus,
 \[
 	\tnorm[B]{\phi}=\norm[L^{\oo}]{\phi}+\sup_{n\in\N_{>0}}\sup_{x, y\in B(x,\eps_0,n)}\abs{S_n\phi(x)-S_n\phi(y)}
 \]
 where in the last supremum it is understood that $x,y$ are chosen almost everywhere with respect to $\mu$. Note that this is clearly a norm, and if $(\phi_k)_k \subset \Bow[][\mu]{\mc E}$ verifies $\lim_k \tnorm[B]{\phi_k-\phi}=0$ for some function $\phi$, then $\phi\in \Bow[][\mu]{\mc E}$. 

\begin{proposition}\label{pro:BowenBanach}
    $\paren*{\Bow[][\mu]{\mc E},\tnorm[B]{\cdot}}$ is a Banach space. 
\end{proposition}

 \begin{proof}
 	For $n\in\N_{>0}$ let $U_n=\set{(x,y)\in \mc{E}\times \mc{E}:y\in B(x,\eps_0,n)}$. Note that for $\phi\in \Bow[][\mu]{\mc E}$, denoting $h_n(x,y)=S_n\phi(x)-S_n\phi(y)$, we have $\norm[B]{\phi}=\sup_{n>0}\esssup(h_n|U_n)$, where the essential supremum is taken with respect to $\mu\otimes\mu|U_n$.

	Let 
	\[
	Z=\mc{E}\times\{0\}\cup\bigcup_{n>0} U_n\times\{n\}
	\]
	be the disjoint union of the different $U_n$ and $\mc{E}$. For a given function $u$ on $Z$ we write $u(\cdot,0), u(\cdot,\cdot, n)$ for its restrictions to $\mc{E}\times\{0\}, U_n\times\{n\}$, respectively. Let $\nu$ be the ($\sigma-$finite) measure on $Z$ induced by $\mu$.

    Consider the  Banach space $A$, whose underlying space is $\Lp{\oo}{Z,\nu}$, but with norm 
\[
 	\norm[A]{u} \defeq \norm[L^{\oo}]{\restr{u}{\mc{E}\times\{0\}}}+\sup_{n>0}\norm[L^{\oo}]{\restr{u}{U_n\times\{n\}}}.
 \] 
Notice that this norm is equivalent to the standard $\Lp{\oo}$ norm on $Z$.

 We have that the functions with the Bowen property embed isometrically into $A$ by $\iot[\phi](a)=\phi(x)$ for $a=(x,0)\in \mc{E}\times\{0\}$ and $\iot[\phi](a)=S_n\phi(x)-S_n\phi(y)$ if $a=(x,y,n)\in U_n\times\{n\}$. 

 It is thus enough to show that $\iota(\Bow[][\mu]{\mc E}) \subset A$ is closed. Let $(\phi_k)_k \subset \Bow[][\mu]{\mc E}$ and assume that $u_k\defeq \iot[\phi_k]\xrightarrow[k\mapsto\oo]{} u$; we want to show that there is a function $\phi:D_\phi \subset\mc{E}\to \R$ so that $u(x,y,n)=S_n\phi(x)-S_n\phi(y)$ whenever $\dis[\mc{E},n]{x}{y}\leq \eps_0$. Note that $\phi(x)$ can only be defined as $\phi(x)=\lim_k u_k(x,0)$, which exists $\mu-$almost everywhere.

 On the other hand, we know that for every given $n$ and $k$, $S_n\phi_k(x)-S_n\phi_k(y)=u_k(x,y,n)\xrightarrow[k\mapsto\oo]{} u(x,y,n)$. Also, for $\aee{\mu}$ it holds for every $n$, $S_n\phi_k(x)\xrightarrow[k\mapsto\oo]{} S_n\phi(x)$ (pointwise convergence together with the bounded convergence theorem). Hence 
      \[
      u_k(x,y,n)=S_n\phi_k(x)-S_n\phi_k(y)\xrightarrow[k\mapsto\oo]{} S_n\phi(x)-S_n\phi(y)
      \]
      and so 
      \[
  	S_n\phi(x)-S_n\phi(y)=u(x,y,n)
      \]
      for every $n$, for almost every $x,y$ with $\dis[\mc{E},n]{x}{y}\leq \eps_0$, which means that $u=\iot[\phi]$, thus the image is closed.  
\end{proof}

\begin{definition}\label{def:Livsiccohomologous}
We say that $\phi\in \Bow[][\mu]{\mc E}$ is a Livšic coboundary if there exists $u\in \Lp{\oo}{\mc E,\mu}$ which is differentiable along orbits of $\lie{g}$ and satisfies $\phi=\dert[t=0]{u\circ g_t}$. The space of Livšic coboundaries is denoted by $\Cob[][\mu]{\mc E}$.

 Two functions $\phi, \psi\in \Bow[][\mu]{\mc E}$ are said to be Livšic cohomologous ($\phi\sim\psi$) if their difference is a Livšic coboundary. 
\end{definition}

Note that $\Cob[][\mu]{\mc E}\subset \paren*{\Bow[][\mu]{\mc E},\tnorm[B]{\cdot}}$ is a closed subspace, thus
\[
	\Bow[][\mu]{\mc E}/\thicksim \defeq \frac{\Bow[][\mu]{\mc E}}{\Cob[][\mu]{\mc E}}
\]
is a Banach space with respect to the norm
 \begin{align}\label{def:Bowennorm}
 \tnorm[B]{[\phi]}=\inf\{\tnorm[B]{\psi}:\psi\sim\phi\}.
 \end{align}

 \begin{definition}\label{def:antisymmetric}
 	We say that $\phi\in \Bow[][\mu]{\mc E}$ is antisymmetric if $\phi\circ I\sim -\phi$. We denote by $\Bow[a][\mu]{\mc E}$ the subspace of $\Bow[][\mu]{\mc E}$ consisting of antisymmetric elements.
 \end{definition}

With these definitions we have the following generalization of \Cref{thm:A,thm:B}. 

\begin{theorem}\label{thm:mainAB}
Let $M$ be a closed locally $\CAT$ space, and let $\mu\in \PTM{\lie g}{\mc E}$ be fully supported. Denote by $\Gamma=\pi_1(M,*)$.

\begin{enumerate}
	\item There exists a Banach isomorphism 
     \[
     \Psi:\paren*{\QMt{\Gamma},\norm{\cdot}}\to \paren*{\Bow[a][\mu]{\mc E},\tnorm{\cdot}}.
     \]
    \item For each $[\phi]\in  \Bow[][\mu]{\mc E}/\thicksim$ there exists a measure $\mu_{[\phi]}=\mu_{\phi} \in \PTM{\lie{g}}{\mc E}$ which can be characterized as the unique element of this space satisfying
\[
	\hmu[\mu_{\phi}]{\lie{g}}+\int \phi\dd \mu_{\phi}=\sup_{\mu\in \PTM{\lie{g}}{X}}\paren*{\hmu[\mu]{\lie{g}}+\int \phi\dd \mu}. 
\]
For every $t\neq 0$, the stochastic process defined by $(\phi\circ g_{nt})_{n\in \N}: \mc E\to\R$ is measure-theoretically isomorphic to a Bernoulli process with finitely many symbols.   
\end{enumerate}
\end{theorem}

 An important part of this article is devoted to extending the classical cohomology theory of Livšic \cite{Livshits1971} for non-regular weak Bowen functions. Consider $\phi:D_\phi\to\R\in \Bow[a][\mu]{\mc E}$: with no loss of generality we assume that $D_{\phi}$ is $\lie g$-invariant, and also dense, since $\mu$ is fully supported. Let $x\in\mc E$ be contained in a periodic orbit $\alpha$, and suppose that $y\in D_\phi$ is so that
\[
 	\sup_{t\in [0,\per(\alpha)]}\dis[\mc E]{g_tx}{g_ty}\leq \Cexp.
 \]
Then the limit 
 \[
      \lim_{T\to\+\oo} \frac{S_T\phi(y)}{T}
 \]
 exists, does not depend on the point $y$ chosen with the above rule, and in fact only depends on $\alpha$. These facts are simple to prove: existence follows almost directly from subadditivity of the sequence $(S_n\phi(x))_{n\in\N}$, while independence of the chosen point ($y$ and $x\in \alpha$) is consequence of the weak Bowen property. See \Cref{sec:LivsicSFT} where an analogous fact is proven. 

\begin{definition}\label{def:averageonperiodicorbit}
The average of $\phi$ on the periodic orbit $\alpha$ is
\[
\av_{\alpha}(\phi)\defeq \lim_{T\to +\oo}\frac{S_T\phi(y)}{T},
\]
where $x\in \alpha$ and $y\in D_{\phi}$ satisfy
\[
\sup_{t\in [0,\per(\alpha)]}\dis[\mc E]{g_tx}{g_ty}\leq \Cexp.
\]
\end{definition}

This definition is natural, since if $D(\phi)=\mathcal E$, then $\av_{\alpha}(\phi)$ is just the integral of $\phi$ with respect to the homogeneous probability measure supported on $\alpha$. We will prove:

\begin{maintheorem}\label{thm:E}
Let $\mu\in \PTM{\lie g}{\mc E}$ be fully supported. Then $\phi,\psi\in \Bow[][\mu]{\mc E}$ are Livšic cohomologous if and only if for every periodic orbit $\alpha$ of $\lie g$ it holds $\av_{\alpha}(\phi)=\av_{\alpha}(\psi)$.
\end{maintheorem}

This immediately implies the claim stated in the introduction that $\Ker c_2$ is infinite-dimensional \cite{BrooksSeries84, Mitsumatsu1984, Barge1988, Epstein1997, Fujiwara1998, Bestvina2002, Hamenstaedt2008}.

\begin{corollary}\label{cor:infinitedimensionalkerc}
For $M,\mu$ as in \Cref{thm:mainAB}, the kernel $\Ker \paren*{c_2: H^2_{\mrm{b}}(M,\R)\to H^2(M;\R)}$ is infinite dimensional.
 \end{corollary}
 \begin{proof}
  Take $(\al_n)_n$ an infinite sequence of different (oriented) periodic orbits for the flow $\lie{g}$, and choose a pairwise disjoint open family $(U_n)_n$ so that $U_n=I(U_n)$ is a neighborhood of $\al_n$ in $\mc E$. Standard methods allow one to construct for each $n$ a Lipschitz function $\phi_n:\mc E\to \R$ so that $\restr{\phi_n}{\al_n}>0$ and $\restr{\phi_n}{U_n^c}=0$. Define $\phi_n\circ I=-\phi_n$.
  \Cref{thm:E} directly implies that $\set{[\phi_n]}_n \subset \Bow[][\mu]{\mc E}/\thicksim$ is linearly independent, therefore the target space is infinite dimensional, which implies the same for $\QMt{\Gamma}, \Gamma=\pi_1(M,*)$, by \Cref{thm:mainAB}. Since $\Ker c_2$ is isomorphic to $\QMt{\Gamma}/\Hom{\Gamma,\R}$, the claim follows.
 \end{proof}

There exist analogous versions of the Limit Theorems and their corresponding Corollaries for the general setting discussed in \Cref{thm:mainAB}. The proofs will be given in this framework.

\subsection{Symbolic Representation of the Geodesic flow} 
\label{sub:symbolic_representation_of_the_geodesic_flow}

A key fact in the differentiable setting is that the geodesic flow of a negatively curved metric is an Anosov flow, and in particular it admits Markov partitions \cite{RatnerMarkov,SymbHyp}. This allows one to use the powerful machinery of symbolic dynamics to study the flow. In a recent work of Constantine, Lafont and Thompson \cite{Constantine2020} the same type of structure is established for compact locally $\CAT$ spaces, refining a previous construction due to Pollicott \cite{Pollicott1987}. We explain this below. 

Given $\ms A$ a finite set (the alphabet) and $R:\ms A\times\ms A\to\{0,1\}$ consider
\[
	\lift\Sigma=\lift\Sigma_R=\set*{\seq{x}=(x_n)_{n\in \Z}:x_n\in\ms A, R(x_n,x_{n+1})=1, \forall n\in\Z}.
\]
By identifying $\ms A=\{1,\cdots,d\}$ we get that $R=(R_{ij})_{1\leq i, j\leq d}\in\Mat_{d}\paren{\set{0,1}}$.  From now on we assume that $R$ is irreducible and aperiodic, and in particular, there exists $M>0$ such that for all $k\geq M$, $R^k$ is a positive matrix. If $M=1$ ($R_{ij}=1$ for all $i,j$) then $\lift \Sigma$ is called the full shift.

\begin{definition}\label{def:SFT}
By a subshift of finite type (SFT) we mean a space $\Sigma_R$ as described above (in particular, $R$ is irreducible and aperiodic). The number $M\in\N$ such that $R^M$ is positive will be referred to as the specification constant of the subshift.
\end{definition}

The space $\lift \Sigma$ is topologized as a subset of $\ms{A}^{\Z}$, where the latter is given the product topology induced by the discrete one on $\ms A$. It follows that $\lift\Sigma \subset \ms{A}^{\Z}$ is closed, and homeomorphic to a Cantor space (Moore-Kline theorem). The metric on $\ms{A}^{\Z}$ defined as
\[
	\dis[\lift\Sigma]{\seq{x}}{\seq{y}}=\frac{1}{2^{N(\seq{x},\seq{y})+1}},\quad N=\max\{n: x_i=y_i, i=-N,\ldots N\}
\] 
is compatible with the product topology, therefore $\ms{A}^{\Z}$ and $\lift\Sigma$ are metrizable. In the rest of the article it is assumed that these spaces are equipped the above metric. 

The shift map on $\lift\Sigma$ is the (bi-Lipschitz) homeomorphism $\tau:\lift\Sigma\toit$, $\tau(\seq{x})=(x_{n+1})_{n\in \Z}$. 

Fix a SFT $\lift\Sigma$, and let $f:\lift\Sigma\to\R_{>0}$ be a continuous bounded function. Here and in other parts of this article we write
\begin{align}
 S_n^{\tau}f=\begin{dcases}
 0 & n=0\\
 \sum_{k=0}^{n-1} f\circ\tau^k & n>0\\
 \sum_{k=n}^{-1} f\circ \tau^{k} &n<0
 \end{dcases}
 \end{align}
(Birkhoff's sums of $f$ associated to the dynamics $\tau$): then $(S_n^{\tau}f)_{n\in\Z}$ is an additive cocycle over $\Z$. On $\lift\Sigma\times \R$ consider the $\Z-$action given by $n\cdot (\seq{x},t)=(\tau^n \seq{x},t-S_n^{\tau}f(\seq{x}))$, and denote by $\lift\Sigma_f=\lift\Sigma\times \R/\Z$. The translation flow on $\lift{\Sigma}\times \R$ induces a continuous flow $(\tau_t)_{t\in\R}$ on $\lift{\Sigma}_f$ by
\begin{align*}
 \tau_t([\seq{x},s])=[\seq{x},s+t].
 \end{align*}

\begin{definition}\label{def:suspensionflow}
The flow $\lie{t}=(\tau_t)_{t\in\R}:\lift\Sigma_f\toit$ is the suspension flow of $\tau$ under the function $f$.
\end{definition}

Recall that an invariant probability for a flow is ergodic if every (Borel) set that is invariant under flow has zero or full measure.

\begin{theorem}\label{thm:symbolicrepresentation}
Let $M$ be a closed locally $\CAT$ space. Then there exist a reparametrization $\lie{g}=(g_t:\mc E\toit)_{t\in\R}$ of the associated geodesic flow, a SFT $\lift\Sigma$, a Hölder function $f:\lift\Sigma\to\R_{>0}$, and a Hölder map $h:\lift\Sigma_f\to\mc E$ satisfying:
\begin{enumerate}
	\item $h$ is surjective, and $h\circ\tau_t=g_t\circ h$ for every $t\in\R$;
	\item there exists a closed nowhere dense set $B \subset \mc E$ so that $h$ is one-to-one on $h^{-1}(\bigcup_{t\in\R} g_t(B))$;
	\item there exists $k\in\N$ so that $h$ is at most $k$-to-one;
	\item $h$ sends periodic orbits of $\lift\Sigma_f$ to periodic orbits of $\lie g$;
    \item if $\nu\in  \PTM{\lie g}{\mc E}$ has full support and is ergodic, then there exists a unique $\nu_{\lie t}\in \PTM{\lie{t}}{\lift\Sigma_f}$ so that $\nu_{\lie t}\circ h^{-1}= \nu$.
\end{enumerate}
\end{theorem}

\begin{proof}
This is classical for hyperbolic flows \cite{SymbHyp}. In the previous generality, everything is proven in \cite{Constantine2020}, except for the last point, that follows from the others. 
\end{proof}

\begin{corollary}\label{cor:BowengBowent}
Let $\nu\in  \PTM{\lie g}{\mc E}$ be an ergodic measure of full support and $\nu_{\lie t}\in \PTM{\lie{t}}{\lift\Sigma}$ as in the previous theorem. Then $h$ induces Banach isomorphisms  $h_*:\Bow[][\nu_{\lie t}]{\lift\Sigma_f}\to \Bow[][\nu]{\mc E}$, $h_*:\Bow[][\nu_{\lie t}]{\lift\Sigma_f}/\thicksim\to \Bow[][\nu]{\mc E}/\thicksim$.
\end{corollary}

The above theorem allows us to reduce the study of some dynamical properties of $\lie g$ to the corresponding ones in $\lie t$, which typically are more manageable. As an illustration, if $\nu_{\lie t}\in \PTM{\lie{t}}{\lift\Sigma}$, then there exists some $\tau$-invariant measure $\nu_{\tau}$ such that
\begin{align}\label{eq:medidasuspension}
H\in \Cr{\lift\Sigma_f} \Rightarrow \nu_{\lie t}(H)=\frac{\int_{\lift\Sigma}\ \paren*{\int_0^{f(\seq x)} H([\seq x,t])\dd t}\dd\nu_{\tau}(\seq x)}{\int f(\seq x) \dd\nu_{\tau}(\seq x)},
\end{align}
and conversely, for $\nu_{\tau}\in \PTM{\tau}{\lift \Sigma}$ the above formula defines an invariant measure for $\lie{t}$. It follows that there exists a one-to-one correspondence between invariant measures on $\PTM{\lie t}{\lift\Sigma}$ and $\PTM{\tau}{\lift\Sigma}$. Moreover, $\nu$ is ergodic if and only if $\nu_{\tau}$ is ergodic (the definition of ergodicity for $\tau$ is analogous to the case of flows). See for example \cite{zeta}, Chapter $6$.

\begin{corollary}\label{cor:ergodiccorrespondence}
In the same hypotheses as in the previous theorem, there exists a one-to-one correspondence between fully supported ergodic probability measures for $\lie g$, and fully supported ergodic measures for $\tau$. 
\end{corollary}

After discussing some of the structure of SFTs we will improve \Cref{cor:BowengBowent}, establishing that $\Bow[][\nu]{\mc E}$ is in fact isomorphic to the space $\Bow[][\nu_{\tau}]{\lift\Sigma_f}$.

\paragraph{\textbf{Other possible extensions}} To apply our arguments we rely on the existence of the symbolic model for the flow, but once one reaches this setting the rest of the theory goes through. In particular, the same applies to Axiom A flows, which is the case of the geodesic flow restricted to its non-wandering set, when $M$ is a convex compact $\CAT$ space: this fact is also proven in \cite{Constantine2020}. We also mention that the content of this article gives a coarse version of thermodynamic formalism for classical hyperbolic systems, as Anosov diffeomorphisms and flows.

The idea of studying equilibrium states for the Gromov-Mineyev flow of a locally $\CAT$ space also appears in the recent work of Dilsavor and Thompson \cite{GibbsCAT} (based on the first author Ph.D. thesis \cite{thesisDilsavor}). The focus there is on constructing equilibrium states for some families of continuous sub-additive potentials (in our terminology, continuous quasicocycles for the flow), while the arguments are more dependent on the geometry of the group, e.g. constructing some Patterson-Sullivan type density and using it to assemble the equilibrium state. In the compact case, the results here give another proof of existence and uniqueness of equilibrium states for the families that they consider, but their geometrical approach allows them to consider non-compact cases as well. It seems likely that further synergy between their results and ours can be exploited.

\section{Quasimorphisms on hyperbolic groups}
\label{sec:QuasiMorphismGroups}

Given a finitely generated group $\Gamma$, one can consider bounded real-valued cochains and proceed analogously to the case of topological spaces in order to define its bounded cohomology. It is a result of Gromov (page 257 of \cite{GromovBounded}) that if $M$ is a $\mrm{K}(\Gamma,1)$, the bounded cohomology of $M$ and $\Gamma$ are naturally isomorphic. This is useful because, as mentioned in the introduction, to study the $\Ker c_2$ one can study instead the space $\QM{\Gamma}$. We remind the reader that, by the Cartan--Hadamard theorem, if $M$ is a locally $\CAT$ space then it is a $\mrm{K}(\pi_1(M,*),1)$, 

\begin{notation} if $X$ is a set we denote $\ell^{\oo}(X)=\{f:X\to\R:(\exists C>0)\ \forall x\in X,\ |f(x)|\leq C\}$, and we write $\ell^{\oo}=\ell^{\oo}(\N)$. The $\ell^{\oo}$ norm on $\ell^{\oo}(X)$ is defined as usual.
\end{notation}

Recall that $L, L'\in \QM{\Gamma}$ are said to be cohomologous if $L-L'\in \ell^{\oo}(\Gamma)$.  Note that $L\mapsto \norm{\delta L}$ is a seminorm on $\QM{\Gamma}$, hence it induces naturally a norm on $\QMt{\Gamma}$. However, this norm is not complete, and it is more convenient to use a different one. Fix a finite generating set $S$ of $\Gamma$, and denote $S_e=S\cup \{1_{\Gamma}\}$. For $L \in\QM{\Gamma}$ and $a,b\in\Gamma$ we are writing 
\[
     \delta L(a,b)=L(ab)-L(a)-L(b).
 \]
 Define
 \begin{align}
   \norm{L}=\norm[S]{L}=\norm[\ell^\oo]{\restr{L}{S_e}}+\norm[\ell^\oo]{\delta L}.
 \end{align}  

\begin{proposition}\label{pro:qmareBanach}
    $\paren{\QM{\Gamma},\norm{\cdot}}$ is a Banach space. Moreover, if $S'$ is another finite generating set of $\Gamma$ then the norms $\norm[S]{\cdot}, \norm[S']{\cdot}$ are equivalent.
\end{proposition}

\begin{proof}
Clearly $\norm{\cdot}$ is a semi-norm, and if $\norm{L}=0$ then, on one hand $\delta L=0$ and hence $L$ is a homomorphism, and on the other hand it vanishes on a generating set, so it is the $0$ homomorphism. Thus $\norm{\cdot}$ is norm. The equivalence of the norms for finite $S, S'$  is obvious, so it remains to show completeness. We proceed similarly as in \Cref{pro:BowenBanach}.

Consider the space $\mc{S}=\ell^{\infty}(S_e)\times\ell^{\infty}(\Gamma\times \Gamma)$ with the norm 
\[
      \norm[S]{(l,u)}=\norm[\ell^{\infty}(S_e)]{l}+\norm[\ell^{\infty}(\Gamma\times \Gamma)]{u}.
\]
$\mc{S}$ is a Banach space, and $\QM{\Gamma}$ embeds isometrically in it as $L\to (\restr{L}{S_e},\delta L)$. We claim that the image of this map is closed. Let $(l_n,u_n)_n=(\restr{L_n}{S_e},\delta L_n)$ be a sequence in $\mc{S}$ and assume that $l_n \xrightarrow[n\mapsto\oo]{} l$ and $u_n=\delta L_n \xrightarrow[n\mapsto\oo]{} u$, thus in particular $\norm[S]{(l,u)}<\oo$: we want to show that $l=L|S$ and $\delta L=u$ for some $L:\Gamma\to\R$ (which necessary will satisfy $\norm{\delta L}<\oo$, hence $L\in\QM{\Gamma}$). Let us show inductively that $L_n$ converges pointwise in $B_k=\{w\in\Gamma:\abs[S]{w}\leq k\}$.

For $k=0$ and $k=1$ this is true since $\restr{L_n}{S_e}=l_n \xrightarrow[n\mapsto\oo]{}l$. Assume it is convergent for $k$, take $x\in B_{k+1}$ and consider $s\in S, \gamma\in S^k$, so that $x=s\gamma$. It follows that
 \[
 L_n(x)=L_n(s\gamma)=u_n(s,\gamma)+l_n(s)+\restr{L_n}{B_k}(\gamma).
 \]
 It is no loss of generality to suppose $\abs[S]{x}=k+1$, otherwise we are already in the induction hypotheses. Note also that $x=s_1\gamma_1$, then 
 if $x=s_1\gamma_1$, then 
  \[
  u_n(s,\gamma)+l_n(s)+\restr{L_n}{B_k}(\gamma)=L_n(x)=u_n(s_1,\gamma_1)+l_n(s_1)+\restr{L_n}{B_k}(\gamma)(\gamma_1),
 \]
 Since $u_n, l_n$ and $\restr{L_n}{B_k}$ are convergent, we get that $L_n(x)$ is Cauchy, hence convergent, so $\restr{L_n}{B_{k+1}}$ converges pointwise to a map that we denote $\restr{L}{B_{k+1}}$. 
 This way we have $L$ defined on $\bigcup_{k\geq 0} B_k=\Gamma$. From the pointwise convergence we also get that if $a\in B_k, b\in B_{k'}$ ($\Rightarrow ab\in B_{k+k'}$)
   \begin{align*}
     &\delta L_n(a,b)=u_n(a,b)\xrightarrow[n\mapsto\oo]{} u(a,b)\\
     &\delta L_n(a,b)=L_n(ab)-L_n(a)-L_n(b)\xrightarrow[n\mapsto\oo]{} L(ab)-L(a)-L(b)=\delta L(a,b).
     \end{align*}
   This concludes the proof.
   \end{proof}

\begin{remark}
    In the proof above, note that $\delta L(a,a^{-1})=L(1_{\Gamma})-L(a)-L(a^{-1})$, hence $\delta L$ and $\restr{L}{S_e}$ determine $L$ on $S^{-1}$. We deduce that we can assume that $S$ is symmetric in what follows.
    
    In other geometrical problems it would be important to consider distances in the space of quasimorphisms that do not possess this feature.  
\end{remark}

Since $\ell^{\oo}(\Gamma) \subset \QM{\Gamma}$ is closed with respect to $\norm[S]{\cdot}$, from the above it follows directly that $\QMt{\Gamma}$ is a Banach space with respect to the quotient norm
\begin{align}
 \norm[S]{[L]}=\inf\{\norm[S]{L'}:L'\sim L\}.
 \end{align}
 Moreover, the induced norm on $\frac{\QM{\Gamma}}{\Hom{\Gamma,\R}\oplus\ell^\oo(\Gamma)}\approx\frac{\QMt{\Gamma}}{\Hom{\Gamma,\R}}$ is given by
 \begin{align}\label{eq:normakerc2}
 \norm{[L]}=\inf\{\norm{\delta L'}:L'\sim L\},
 \end{align}
 which in particular does not depend on $S$.

We collect the relation between quasimorphism and bounded cohomology in the following.

\begin{proposition}\label{pro:exactsequence}
If $\Gamma$ is finitely generated then there exists a short exact sequence of  linear maps
\[
    0\rightarrow \Hom{\Gamma,\R}\rightarrow \frac{\QM{\Gamma}}{ \ell^\oo(\Gamma)} \rightarrow H^2_{\mrm b}(\Gamma;\R) \xrightarrow[]{c_2} H^2(\Gamma;\R)\rightarrow 0.
\]
Moreover, there exists a linear isometry
\[
     \frac{\QM{\Gamma}}{\Hom{\Gamma,\R}\oplus\ell^\oo(\Gamma)}\approx \Ker c_2.
\]
\end{proposition}

 \begin{proof}
     The existence of the short exact sequence is well known, and can be found for example in \cite{Frigerio2017}: the linear map $\Gamma:\QM{\Gamma}\to \Ker c_2$ given by $\Gamma(L)=[\delta L]$
     is surjective and vanishes precisely on $\Hom{\Gamma, \R}\oplus\ell^\oo(\Gamma)$. On the other hand, the norm on $\Ker c_2$ is of the form
     \[
     \norm{[B]}=\inf\{\norm[\ell^{\oo}]{B+\delta b}:b\in\ell^\oo(\Gamma)\}
     \]
     (see \cite{Ivanov2017}), thus 
    \[
     \norm{\Gamma([L])}=\inf\{\norm[\ell^{\oo}]{\delta(L+b)}:b\in\ell^\oo(\Gamma)\}=\inf\{\norm{\delta(L')}:L'\sim L\}=\norm{[L]}.
     \]
 \end{proof}

Due to the above, we will focus on the space of unbounded quasimorphisms $\QMt{\Gamma}=\frac{\QM{\Gamma}}{\ell^\oo(\Gamma)}$. Each cohomology class of quasimorphism has a well behaved representative, namely a (unique) quasimorphism $\cl{L}$ so that for every $g\in\Gamma, n\in\Z$, $\cl{L}(g^n)=nL(g)$. From here it follows that 
 \[
     \cl{L}(g):=\lim_{n\to\oo}\frac{L(g^n)}{n}
\]
See \cite{Frigerio2017}, or \Cref{lem:homquasimorphism} where an analogous statement is discussed. 

\begin{definition}\label{def:homogeneous}
$L\in\QM{\Gamma}$ is homogeneous if for every $g\in\Gamma, n\in\Z$ it holds $L(g^n)=nL(g)$. The set of homogeneous quasimorphisms on $\Gamma$ is denoted by $\HQM{\Gamma}$.
\end{definition}

 \begin{corollary}\label{cor:homogeneousclass}
 If $L\in \HQM{\Gamma}$, then it is constant on conjugacy classes of $\Gamma$.
 \end{corollary}

\begin{proof}
 Indeed, $|L(h\cdot g h^{-1})-L(g)|\leq 2\norm{\delta L}$ for every $g,h$. It follows that for $h$ fixed the map $L'(g)=L(h\cdot g\cdot h^{-1})-L(g)$ is a bounded homogeneous quasimorphism, hence identically zero.
\end{proof}

Consider the inclusion $\iota: \HQM{\Gamma}\to \QM{\Gamma}$, This is a continuous map and has a left inverse $p(L)=\cl{L}$. Note that $|p(L)(g)-L(g)|\leq \norm{\del L}$ for every $g$ and so $\norm[S]{p(L)}\leq \norm[S]{L}$, thus $p:\HQM{\Gamma}\to \QM{\Gamma}$ is continuous.

Since any bounded homogeneous quasimorphisms is necessarily trivial, we obtain:

\begin{corollary}\label{cor:homogeneousqmareiso}
    The spaces $\QMt{\Gamma}$ and $\HQM{\Gamma}$ are Banach isomorphic when equipped with the norm associated to some finite symmetric generating set $S \subset \Gamma$.
\end{corollary}

\noindent\textbf{Assumption:} From now on we work with homogeneous quasimorphisms.

\subsection{quasimorphisms in negative curvature} 
\label{sub:quasi_morphisms_in_negative_curvature}

Now let $M$ be a closed locally $\CAT$  space, and let $\lie{g}: \mathcal{E}\toit$ be a geodesic flow associated to $\Gamma=\pi_1(M,*)$, as in \Cref{sub:geodesic_flows}. In this part we will show that elements of $\QMt{\Gamma}$ can be uniquely represented by some equivalence classes of functions defined on the set of periodic orbits of $\lie g$. We introduce the notation
\begin{align}
\Per{\lie g}:=\{\alpha:\text{ oriented closed orbit of $\lie{g}$}\}
\end{align}
and for $\alpha\in \Per{\lie g}$, we denote its minimal period by $\per(\alpha)$. We recall that for $\alpha_1, \alpha_2\in \Per{\lie g}$ we write $\alpha_3=\alpha_1\star\alpha_2$ if $\alpha_3$ verifies the conclusion of \Cref{thm:shadowing}. From now on we choose $\delta$ sufficiently small so that if $c,d$ are curves in $M$ whose Hausdorff distance is less than $2\delta$, then they are homotopic. By compactness of $M$ we can also guarantee the existence of some $E>0$ so that if $\alpha_3=\alpha_1\star\alpha_2, \alpha_4=\alpha_1\star\alpha_2$ then
\begin{align}\label{eq:boundstar}
 \sup_{t\in[0,\per(\alpha_3)=\per(\alpha_4)]}\{\dis[\mc E]{\alpha_3(t)}{\alpha_4(t)}\}<E.
\end{align}

We fix $S$ a finite symmetric set of generators for $\Gamma$ as before. By word-hyperbolicity of $\Gamma$, any $g\in\Gamma$ is conjugate to some element that can be written in a reduced cyclic word on $S$, and if $g_1, g_2$ are two such elements, then $g_1$ is conjugate to some cyclic permutation on the letters of $g_2$. See Chapter 4 in \cite{Coornaert1990}. For $g\in\Gamma$ its translation length is defined as
 \[
     \tlength_{\Gamma}(g)=\lim_{n\to\oo}\frac{\dis[\Gamma]{g^n}{1}}{n}.
 \]
 It is well known that this is a class function; moreover, $\tlength_{\Gamma}(g)=\dis[\Gamma]{\lift g}{1}$, where $\lift{g}$ is conjugate to $g$ and can be written as a reduced cyclic word on $S$ of minimal size. Word-hyperbolicity of $\Gamma$ also implies that any of its elements $g$ fixes a unique geodesic $c_g:\R\to\lift{M}$ (called the axis of $g$), and for $x\in \Im(c_g)$ it holds that $|\tlength_{\Gamma}(g)-\dis[\lift M]{\lift g\cdot x}{x}|<C$ , where $\lift g$ is as in the previous line and $C$ does not depend on $g, x$. For $\phi:\Per{\lie g}\to\R$ define
\begin{align*}
&|\delta \phi(\alpha_1,\alpha_2)|=\inf_{\alpha_1\star\alpha_2}\set*{|\phi(\alpha_1\star\alpha_2)-\phi(\al_1)-\phi(\al_2)|},\\
&\norm{\delta \phi}:=\sup_{\alpha_1,\alpha_2}\set{|\delta \phi(\alpha_1,\alpha_2)|},
\end{align*}
and let
\begin{align*}
 \QM{\Per{\lie g}}:=\{\phi:\Per{\mathfrak g}\to\R:\norm{\delta \phi}<\oo\}.
 \end{align*}

\begin{definition}\label{def:cohomologousperiodicorbits}
Two functions $\phi_1, \phi_2\in \QM{\Per{\lie g}}$ are cohomologous if their difference is uniformly bounded. We denote
\begin{align}\label{eq:setperiodicqm}
\QMt{\Per{\lie g}}\defeq \frac{\QM{\Per{\lie g}}}{\ell^{\oo}(\Per{\lie g})}.
\end{align}
\end{definition}

To define a norm on $\QM{\Per{\lie g}}$ choose $N\in\N$ so that (the finite set) $V=V_N=\{\al\in\Per{\lie g}: \tlength_{\Gamma}(g)\leq N\}$ satisfies
 \[
     \mc E=\bigcup_{\al\in V}\bigcup_{x\in\al} B(x,\frac{\eps_0}{2},\per(\al)):
 \]
we say that $V$ is generating. Let 
 \[
     \norm[V]{\phi} = \norml{\restr{\phi}{V}}+\norm{\delta \phi}.
 \]
Proceeding as in the proof of \Cref{pro:BowenBanach,pro:qmareBanach} one verifies that this is a Banach norm, and for different $N', V_{N'}$ the corresponding norms are equivalent. Moreover, with
\[
     \norm[V]{[\phi]} \defeq \inf\{\norm[V]{\phi'}:\phi\sim \phi'\}
\]
the space $\QMt{\Per{\lie g}}$ inherits the Banach structure. Denote by
\[
     \QM[a]{\Per{\lie g}}
 \]
the (closed) subspace of $\QM{\Per{\lie g}}$ consisting of functions $\phi$ that are cohomologous to $-\phi\circ I$, and by
\[
    \QMt[a]{\Per{\lie g}}
 \]
the corresponding quotient.

The relevance of periodic orbits comes from the following observation. Since homogeneous quasimorphisms are constant on conjugacy classes, and conjugacy classes correspond to closed geodesics in $M$, it is natural to interpret a quasimorphism as a function on the periodic orbits of the geodesic flow. The quasimorphism relation
\[
|L(gh)-L(g)-L(h)|\leq C
\]
translates into the following: although periodic orbits cannot be concatenated, the shadowing property produces a periodic orbit that successively follows two given ones, and the value assigned to this shadowing orbit differs from the sum of the individual values by a uniformly bounded error.

\begin{theorem}\label{thm:quasimorphismgrouptoporbit}
There exists a bijection $B:\mrm{Conj}(\Gamma)\to \Per{\lie g}$ and constants $b_1, b_2$ so that
\[
   \forall g\in\Gamma,\  b_1\cdot \tlength_{\Gamma}(g)-b_2\leq \per(B([g]))\leq b_1\cdot \tlength_{\Gamma}(g)+b_2. 
\]
If $S \subset \Gamma, V \subset \Per{\lie g}$ are corresponding generating sets, then this bijection induces a Banach isomorphism 
\[
   B_{*}:\left(\QMt{\Gamma},\norm[S]{\cdot}\right)\to \left(\QMt[a]{\Per{\lie g}},\norm[V]{\cdot}\right). 
\]
\end{theorem}

\begin{proof}
We divide the proof into three steps.

\noindent\textbf{Step 1: correspondence between conjugacy classes and periodic orbits.}

Since $M$ is negatively curved, every free homotopy class of loops in $M$ contains a unique closed geodesic. Due to compactness of $M$, every $g\in \Gamma=\pi_1(M,*)$ is represented by a curve of the form $h_g\ast c_g\ast h_g^{-1}$, where $c_g$ is a closed (parametrized) geodesic, and $h_g$ is a geodesic segment of length bounded by $a$, where $a>0$ does not depend on $g$. This defines a bijection between $\mrm{Conj}(\Gamma)$ and the set of closed geodesics in $M$, $[g]\mapsto c_g$.

If $\lift c_g:\R\to\lift M$ is a lift of $c_g$, then $\lift c_g$ determines a periodic orbit of the geodesic flow $\lie g$. Denote this orbit by $\alpha_g$. Altogether, this defines a bijection
\[
B:\mrm{Conj}(\Gamma)\to \Per{\lie g}, \quad B([g])=\al_g.
\]

\noindent\textbf{Step 2: comparison of periods and translation lengths.}

Let $g\in\Gamma$ and let $c_g$ be the closed geodesic corresponding to $[g]$. Since $\lie g$ is a reparametrization of the Gromov geodesic flow, the period of $\alpha_g$ is uniformly comparable to the length of $c_g$. Using the relation between the length of $c_g$ and the translation length $\tlength_\Gamma(g)$, we obtain constants $b_1,b_2>0$ such that
\[
b_1\,\tlength_\Gamma(g)-b_2
\le
\per(B([g]))
\le
b_1\,\tlength_\Gamma(g)+b_2.
\]

\noindent\textbf{Step 3: correspondence of quasimorphisms.}

Let $L\in\HQM{\Gamma}$ and define, for a closed geodesic $\al$,
\[
B_*(L)(\alpha)=L(B^{-1}(\alpha)).
\]
Since homogeneous quasimorphisms are constant on conjugacy classes (see \Cref{cor:homogeneousclass}), this definition is well posed.

Consider $g_i\in \Gamma, c_i=c_{g_i}, \al_i=B(g_i), i=1,2$ and let $\al_3=\al_1\star \al_2$: recall that (\Cref{thm:shadowing}) this denotes a periodic orbit of $\lie g$ that remains at short distance of $\al_1,\al_2$, with bounded transitions between them. It follows that the corresponding parametrized geodesic $c_3$ is (freely) homotopic to a closed curve of the form
\[
    \gamma=u\ast c_1\ast  v\ast c_2\ast w
\]
where the length of $u,v$ and $w$ are uniformly bounded, independently of $g_1,g_2$. Choose any $g_3$ in the conjugacy class of $\gamma$ and note that $B(g_3)=\alpha_3$: in particular the translation length of $g_3$ is uniformly comparable with 
\[
 \per(c_3)\approx \per(\alpha_3)\approx \per(\al_1)+\per(\al_2)\approx \tlength(g_1)+\tlength(g_2).  
\]
Moreover, $|L(g_3)-L(g_1)-L(g_2)|$ is uniformly bounded by a multiple of $\norm[S]{L}$, hence the same is true for $|B_*(L)(\alpha_1\star\alpha_2)-B_*(L)(\alpha_1)-B_*(L)(\alpha_2)|$. Note also that by construction, 
\begin{equation}\label{eq:IBandL}
    B_{*}(L)(I(\alpha))=-B_{*}(L)(\alpha)
\end{equation}

Recall that $\HQM{\Gamma}$ is Banach isomorphic to $\QMt{\Gamma}$ (\Cref{cor:homogeneousqmareiso}). We thus have constructed a continuous linear map 
 $B_{*}:\QMt{\Gamma}\approx\HQM{\Gamma}\to \QMt[a]{\Per{\lie g}}$, which is easily seen to be injective: if $L-L'\in\ell^{\oo}(\Gamma)$ then 
 $B_{*}(L)-B_{*}(L')\in \ell^{\oo}(\Per{\lie g})$. 

Let us prove that it is also surjective. Take $K\in \QM[a]{\Per{\lie g}}$ and define $L:\mrm{Conj}(\Gamma)\to\R$ by
\[
    g,h\in\Gamma\Rightarrow L(hgh^{-1}) \defeq K(\al_g). 
\]
We consider $R>\Csha$ so that any two pair of points in $M$ can be joined by a broken geodesic of size $R$, and say that $g\in \Gamma$ is small if 
there exists a closed geodesic $c$ and paths $u,v,w$ of size $\leq R$ so that $g\approx ucvc^{-1}w$. The fact that $K$ is bounded and antisymmetric implies that $L$ is bounded on small words. Now suppose that we are given $g_1,g_2\in\Gamma$, that we write in terms of the set of generators $S$,
\begin{align*}
g_1=s_1\cdots s_k\\
g_2=r_1\cdots r_l.
\end{align*}
Let $g_3=g_1\ast g_2$ with corresponding closed geodesic $c_3$, and consider the closed curve $\gamma=u\ast c_1\ast  v\ast c_2\ast w$ as above. If $s_k\neq r_1^{-1}$ then $\gamma$ and the closed geodesic corresponding to $g_3$ differ by (a uniformly bounded amount of) small words, hence
\[
    L(g_3)\approx K(c_3)\approx K(c_1)+K(c_2).
\]
In general, $g_1\ast g_2=s_1\cdots s_i r_j\cdots r_l$ with $s_i\neq r_j^{-1}$ and by the previous part and antisymmetry of $K$,    
\begin{align*}
L(g_3) &=L(s_1\cdots s_i)+L(r_j\cdots r_l)\approx  L(s_1\cdots s_i s_{i+1}\cdots s_k)+ L((s_i\cdots s_k)^{-1} r_{i+1}\cdots r_l)\\
&=L(g_1)+L(g_2).
\end{align*}
We thus get that $L$ defines an homogeneous quasimorphism and $B_{*}L=K$. 

By the open mapping theorem (or just by following the constants along in the construction of $L$ above) one gets that $B$ is a Banach isomorphism.
\end{proof}

At this point the natural idea would be to use \Cref{thm:symbolicrepresentation} and try to pass from $\QMt[a]{\Per{\lie g}}$ to an analogous space defined for the suspension flow $\lie t$. Since the periodic orbits of $\lie{t}$ are in one-to-one correspondence with the periodic points of the shift $\tau:\Sigma\toit$, we could reduce even further and work at the symbolic level. Unfortunately, this approach has a significative drawback. The semi-conjugacy between $\lie{t}$ and $\lie{g}$ is one-to-one only for periodic points in an open and dense subset of $\mc E$, hence a given $\phi\in \QM{\Per{\lie g}}$ could be lifted only on a subset of its domain, and indeed this set of problematic periodic points is typically exponentially large. This would make it very difficult to achieve a lift that preserves the  ``quasimorphism'' property on periodic points of $\lie{\tau}$, which is central to what we are doing. Instead, we work at a coarser level and show that a certain cohomology class of $\phi$ is well defined even with this partially available information. To achieve this we now study the corresponding concept of quasimorphism on a SFT.


\section{Quasimorphisms on subshifts and their cohomology}
\label{sec:quasimorphismsSFT}


 Given $\ms A$ a finite set (the alphabet) and $R:\ms A\times\ms A\to\{0,1\}$ consider the one-sided shift
\[
 	\Sigma=\Sigma(R)=\{\seq{x}=(x_n)_{n\in \N}:x_n\in\ms A, R(x_n,x_{n+1})=1, \forall n\in\N\}.
 \]
As in \Cref{def:SFT} we identify $\ms A=\{1,\cdots,d\}, R=(R_{ij})_{1\leq i, j\leq d}\in\Mat_{d}{\{0,1\}}$, and suppose that $R$ is irreducible and aperiodic. The product topology on $\Sigma$ is induced by the metric
\[
 	\dis[\Sigma]{\seq x}{\seq y}=\frac{1}{2^{N(\seq x,\seq y)+1}},\quad N=\max\{n: x_i=y_i, i=0\ldots N\}.
 \] 
The set of finite allowed words is $\ms{W}=\ms{W}_R=\bigcup_{n\geq 0}\ms{W}_n$, where 
\begin{align*}
 \mathscr{W}_n=\begin{dcases}
 \{*\}\quad (*\text{ is the empty word})& n=0\\
 \{\bm a\in \ms{A}^n: \exists\ \seq x\in \Sigma, x_i=a_{i-1}\ \forall 0\leq i\leq n-1\} & n\geq 1.
 \end{dcases}
 \end{align*}
We also write $\ms{W}_{\leq n}=\bigcup_{k=0}^n \ms{W}_k$. If $\bm a\in \ms{W}$ we define the quantity $|\bm a|$ to be $n$ if and only if $\bm a\in \ms{W}_n$; $|\bm a|$ is the length of $\bm a$. Given finite words $\bm a=a_1\cdots a_{n},\bm b=b_1\cdots b_{m}\in \ms{W}$ of lengths $n,m\geq 1$ and such that $R(a_{n},b_1)=1$ the concatenated word $\bm{a}\bm{b}$ is,
\[
	\bm{a}\bm{b}=a_1\cdots a_{n} b_1\cdots b_{m}.
\]
Concatenation with the empty word is defined in the obvious way. It is direct to check that $\bm{a}\bm{b} \in \ms{W}$.

 \begin{definition}\label{def:periodicword}
 A word $\bm{a}\in \ms{W}\setminus\{*\}$ is said to be periodic if $R(a_{|\bm{a}|},a_1)=1$. For a such a word we let $\seq{p}(\bm{a})=\bm{a}\bm{a}\bm{a}\cdots\in \Sigma$, and write
 \begin{align*}
 &\Fix[][n]{\ms W}=\{\bm{a}\text{ is periodic and }|\bm{a}|\mid n\}\\
 &\Fix{\ms W}=\bigcup_{n\geq 1} \Fix[][n]{\ms W}.
 \end{align*}
 \end{definition}

If $\bm{a}\in\Fix{\ms W}, n\geq 1$ we denote $\bm{a}^n=\underbrace{\bm{a}\cdots \bm{a}}_{n \text{ times}}$, and set $\bm{a}^0=*, \forall \bm{a}\in \ms{W}$. The shift map on $\Sigma$ is the continuous endomorphism $\tau:\Sigma\toit$, $\tau(\seq x)=(x_{n+1})_{n\geq 0}$. 

\begin{remark}\label{rem:periodicpoint}
If $\bm{a}\in \Fix[][n]{\ms W}$ then $\tau^n(\seq{p}(\bm{a}))=\seq{p}(\bm{a})$. We will write 
\[
\Fix[][n]{\tau}=\{\seq{p}(\bm a):\bm{a}\in \Fix[][n]{\ms W}\},	
\]
and likewise $\Fix{\tau}=\bigcup_{n\geq 1} \Fix[][n]{\tau}$, the set of periodic points of $\tau$.
\end{remark}

\begin{lemma}\label{lem:dinbasicashift}
 It holds
 \begin{enumerate}
 	\item $\Fix{\tau}$ is dense in $\Sigma$.
 	\item If $i,j\in\ms{A}$, then there exists $\bm{u}^{(ij)}\in \ms{W}_{M-1}$ so that $i\bm{u}^{(ij)}j\in \ms{W}$. As a consequence, for every pair of words $\bm{a},\bm{b}\in \ms{W}$ there exist $\bm{u},\bm{v}\in \ms{W}_{M}$ so that $u_1=a_1, v_1=b_m$ and $\bm{a}\bm{u}\bm{b}\bm{v}\in\Fix{\ms W}$.
 \end{enumerate}
 \end{lemma} 
This is a direct consequence of irreducibility. See Lemma $1.3$ in \cite{EquSta}. 

For $\bm{a}\in \ms{W}_n, n\geq 1$ the cylinder in $\Sigma$ corresponding to this word is
\begin{equation}
 	[\bm{a}]=\{\seq x\in \Sigma: x_i=a_i, 0\leq i\leq n-1\},
 \end{equation}
and for $\seq x\in [\bm{a}]$ we also denote $[\seq{x}]_n=[\bm{a}]$. Observe that 
\[
	[\seq{x}]_n=\{\seq{y}\in\Sigma:\dis[\Sigma]{\tau^i \seq{x}}{\tau^i \seq{y}}<1/2, i=0,1,\ldots n-1\};
\]
the set $[\seq{x}]_n$ will be referred to as the $n^{th}-$Bowen ball of $\seq{x}$. To conclude, we write $\PTM{\tau}{\Sigma}$ for the set of Borel probability measures $\mu$ on $\Sigma$ that are invariant under $\tau$ (meaning, $\mu(A)=\mu(\tau^{-1}A)$ for every Borel set $A$). We remind the reader that $\mu\in \PTM{\tau}{\Sigma}$ is ergodic if every Borel set $A$ satisfying $A=\tau^{-1}(A)$ also verifies $\mu(A)\mu(A^c)=0$. 

Here are the main definitions of this part. 

\begin{definition}\label{def:quasimorphismSFT}
For $L:\ms{W}\to \R$ denote
\begin{align*}
&\delta L(\bm{a},\bm{b})=L(\bm{a}\bm{b})-L(\bm{a})-L(\bm{b})\\
&|\delta L_{n,m}|=\max\{|\delta L(\bm{a},\bm{b})|:|\bm{a}|=n,|\bm{b}|=m\}\\
&\norm{\delta L}=\sup_{n,m}\{|\delta L_{n,m}|\}. 
\end{align*}
$L$ is a quasimorphism if $\norm{\delta L}<\oo$: in this case $\norm{\delta L}$ is the defect of $L$. The set of quasimorphisms is denoted by $\QM{\ms W}$.
\end{definition}

 $\QM{\ms W}$ forms naturally a vector space and $L\mapsto \norm{\delta L}$ is a semi-norm on it; moreover $\norm{\delta L}=0$ if and only if for every pair of concatenable words $\bm{a}, \bm{b}$, 
 $L(\bm{a}\bm{b})=L(\bm{a})+L(\bm{b})$.  

 From the discussion in the last part of the previous session it follows that we are interested in quasimorphisms that are defined in principle only on periodic points of $\tau$. To make this precise we rely on the second part of \Cref{lem:dinbasicashift}. 

\begin{definition}\label{def:starwords}
If $\bm{a}\in \ms W$ we denote by $\lift{\bm a}$ any word of the form $\bm{a}\bm{u}\in \Fix{\ms W}$, with $\bm{u}\in \ms{W}_{\leq M}$.

For $\bm{a},\bm{b}\in\Fix{\ms W}$ we denote by $\bm{c}=\bm{a}\star\bm{b}$ any word of the form
\[
	\bm c=\bm{a} \bm{u} \bm{b} \bm{v}\in\Fix{\mc W},
\]
where $\bm{u},\bm{v}\in\ms{W}_{\leq M}$.
\end{definition}

\begin{remark}
By \Cref{lem:dinbasicashift}, if $\bm{a}\in\Fix[][n]{\ms W}, \bm{b}\in\Fix[][m]{\ms W}$ then there exists $\bm{a}\star\bm{b}\in \Fix[][n+m+2M]{\ms W}$.
\end{remark}

The next step is to restrict the theory from arbitrary admissible words to periodic words, which are the symbolic counterparts of periodic orbits.

\begin{definition}\label{def:qmenperiodico}
By a quasimorphism on $\Fix{\ms W}$ we mean a map $L:\Fix{\ms W}\to\R$ verifying: given $\bm{a},\bm{b}\in \ms W$, it holds 
 \[
 	\norm{\delta L} \defeq \sup\set*{|L(\bm{c})-L(\lift{\bm{a}})-L(\lift{\bm{b}})|:\bm a,\bm b\in \ms W, \lift{\bm{a}},\lift{\bm{b}}\in\Fix{\ms W}, \bm c=\bm a\star\bm b}<\oo.
 \]
The set of such functions is denoted by $\QM{\Fix{\ms W}}$, and we call $\norm{\delta L}$ the defect of $L$.
\end{definition}

Note that $r:\QM{\ms W}\to \QM{\Fix{\ms W}}, r(L)=\restr{L}{\Fix{\ms W}}$ verifies 
\[
\norm{\delta r(L)}\leq \norm{\delta L}+2\norml{\restr{L}{\ms{W}_{\leq M}}}.	
\]
We equip $\QM{\ms W}, \QM{\Fix{\ms W}}$ with the norms
\begin{align}\label{eq:normaqm}
&\norm{L}=\norml{\restr{L}{\ms{W}_{\leq M}}}+\norm{\delta L}\\
&\norm{L}=\norml{\restr{L}{\Fix{\ms{W}_{\leq M}}}}+\norm{\delta L},
\end{align}
and leave it to the reader to verify that these are complete norms (see \Cref{pro:BowenBanach,pro:qmareBanach}).

As in the case of groups, we will be concerned with cohomology classes of quasimorphisms.

 \begin{definition}\label{def:cohomologicallytrivialquasimorphism}
 The quasimorphism $L$ is said to be cohomologically trivial if $\norml{\restr{L}{\ms W}}<\oo$. Two quasimorphisms $L, L'$ are cohomologous ($L\sim L'$) if their difference is cohomologically trivial.

 We write $[L]:=\{L':L'\sim L\}$ for the cohomology class of $L$, and denote
 \[
 	\QMt{\ms W}\defeq \frac{\QM{\ms W}}{\ell^{\oo}(\ms W)}.
 \]
 Similarly, 
\[
 	\QMt{\Fix{\ms W}}\defeq \frac{\QM{\Fix{\ms W}}}{\ell^{\oo}(\Fix{\ms W})}.
 \]
\end{definition}

It is clear that $\QMt{\ms W},\QMt{\Fix{\ms W}}$ are vector spaces over $\R$. We equip them with the (Banach) norm
\begin{align}
 \norm{[L]}:=\inf\{\norm{L'}:L'\sim L\}.
 \end{align}

 The following holds true.

 \begin{proposition}\label{pro:cohomologyvscohomologyfixed}
 The spaces $\QMt{\ms W}$ and $\QMt{\Fix{\ms W}}$ are Banach isomorphic. 
 \end{proposition}

\begin{proof}
Fix $s:\ms{W}\to \Fix{\ms W}$ so that $s(\bm a)=\lift{\bm a}=\bm a\bm{u}$, where $\bm u\in\ms{W}_{\leq M}$, and furthermore $s(\bm a)=\bm a$ if $\bm{a}\in\Fix{\ms W}$. Given $L\in \QM{\Fix{\ms W}}$ define $\lift{L}(\bm{a})=L(s(\bm a))$. It is direct that $\lift{L}\in \QM{\ms W}$, with $\norm{\delta \lift{L}}\leq \norm{\delta L}$, $\norml{\restr{\lift{L}}{\ms{W}_{\leq M}}}\leq \norml{\restr{L}{\Fix{\ms{W}_{\leq M}}}}$. 

Clearly $L\sim L'$ implies $\lift{L}\sim \lift{L}'$, and conversely. From here one constructs a continuous linear bijection between the two Banach spaces $\QMt{\ms W}$ and $\QMt{\Fix{\ms W}}$, hence a continuous isomorphism.
\end{proof}

\subsection{Basic properties of quasimorphisms}\label{sec:basicquasimorphisms}

 For later use, in this part we establish some properties of quasimorphisms and introduce notation. Fix $L\in\QM{\ms W}$.

\begin{lemma}\label{lem:boundswords}
Let $\{\bm{a}_1,\cdots, \bm{a}_k\}\subset \ms{W}\setminus\{*\}$ be such that $\bm{a}_1\cdots\bm{a}_k\in \ms{W}$. Then
\[
 |L(\bm{a}_1\cdots\bm{a}_k)-\sum_{i=1}^{k} L(\bm{a}_i)|
 \leq
 \sum_{j=1}^{k-1} \left|\delta L_{\sum_{i=1}^{j}|\bm{a}_i|,\sum_{i=j+1}^{k}|\bm{a}_i|}\right|
 \leq
 (k-1)\norm{\delta L}.
\]
\end{lemma}

\begin{proof}
We argue by induction on $k$. For $k=2$ the claim is immediate from the definition. Assume it holds for $k-1$. Then
\begin{align*}
&|L(\bm{a}_1\cdots\bm{a}_k)-\sum_{i=1}^k L(\bm{a}_i)| \\
&=
\abs{
L(\bm{a}_1\cdots\bm{a}_k)-L(\bm{a}_1\cdots\bm{a}_{k-1})-L(\bm{a}_k)
+
L(\bm{a}_1\cdots\bm{a}_{k-1})-\sum_{i=1}^{k-1}L(\bm{a}_i)
} \\
&\le
\left|\delta L_{\sum_{j=1}^{k-1}|\bm{a}_j|,\,|\bm{a}_k|}\right|
+
|L(\bm{a}_1\cdots\bm{a}_{k-1})-\sum_{i=1}^{k-1}L(\bm{a}_i)| \\
&\le
\left|\delta L_{\sum_{j=1}^{k-1}|\bm{a}_j|,\,|\bm{a}_k|}\right|
+
\sum_{j=1}^{k-2}
\left|\delta L_{\sum_{i=1}^{j}|\bm{a}_i|,\sum_{i=j+1}^{k-1}|\bm{a}_i|}\right| \\
&=
\sum_{j=1}^{k-1}
\left|\delta L_{\sum_{i=1}^{j}|\bm{a}_i|,\sum_{i=j+1}^{k}|\bm{a}_i|}\right|.
\end{align*}
This proves the result.
\end{proof}

\begin{corollary} Let $\{\bm{a}_1,\cdots, \bm{a}_k\}, \{\bm{b}_1,\cdots, \bm{b}_k\} \subset \ms{W}\setminus\{*\}$ such that $\bm{a}_1\cdots\bm{a}_k, \bm{b}_1\cdots\bm{b}_k\in \ms{W}$. It holds 
\[
   |L(\bm{a}_1\cdots\bm{a}_k)-L(\bm{b}_1\cdots\bm{b}_k)|\leq 2(k-1)\|\delta L\|+\abs{\sum L(\bm{a}_i)-\sum L(\bm{b}_i)}.
   \]
 \end{corollary}

 We turn our attention to quasimorphisms defined on $\Fix{\ms W}$. For $\bm{a}\in\Fix{\ms W}$ its cyclic permutation group is denoted by $\mc{C}_{\bm{a}}$ ($\Rightarrow |\mc{C}_{\bm{a}}|=|\bm{a}|$). Define 
\begin{align}\label{def:Lciclico}
 L_{\mrm{cyc}}(\bm{a})=\frac{1}{|\bm{a}|}\sum_{\pi\in\mc{C}_{\bm{a}}}L(\pi(\bm{a})).
 \end{align}
We say that $L$ is cyclic if $\forall\bm{a}$ periodic, $L_{\mrm{cyc}}(\bm{a})=L(\bm{a})$.

\begin{remark}\label{rem:cyclicquasimorphism}
 From the definition of quasimorphism we get that for $\bm{a}\in\Fix{\ms W}, \pi\in \mc{C}_{\bm{a}}$ then $|L(\pi(\bm{a}))-L(\bm{a})|\leq 2\norm{\delta L}$, and thus $|L_{\mrm{cyc}}(\bm{a})-L(\bm{a})|\leq 2\norm{\delta L}$. Note also that for any $m\in\N$,
 \[
   |L_{\mrm{cyc}}(\bm{a}^m)-mL_{\mrm{cyc}}(\bm{a})|\leq \norm{\delta L}.
 \]
 \end{remark}

For $L\in\QM{\ms W}\cup \QM{\Fix{\ms W}}$ define
\begin{align}\label{eq:Lhomogoeneo}
\cl{L}(\bm a) \defeq \lim_{n\to\infty}\frac{1}{n}L(\lift{\bm{a}}^n)=\inf_{n} \frac{L(\lift{\bm{a}}^n)+\norm{\delta L}}{n}.
\end{align}
The limit exists since $(L(\lift{\bm{a}}^n)+\norm{\delta L})_n$ is a sub-additive sequence. 

 \begin{definition}\label{def:homogeneo}
 A quasimorphism $L\in\QM{\ms W}\cup \QM{\Fix{\ms W}}$ is called homogeneous if for all $\bm{a}\in \Fix{\ms W}, n\geq 1$ it holds $L(\bm{a}^n)=nL(\bm{a})$. 

 The set of homogeneous quasimorphisms on $\QM{\ms W}, \QM{\Fix{\ms W}}$ is denoted by $\HQM{\ms W}$, $\HQM{\Fix{\ms W}}$, respectively.
 \end{definition}

 \begin{lemma}\label{lem:homquasimorphism}
 The function $\cl{L}:\ms{W}\to \R$ is a homogeneous quasimorphism. Moreover, 
   \begin{itemize}
      \item $\norm{\cl{L}}\leq 32\norm{L}$;
      \item $\norml{\restr{(L-\cl{L})}{\Fix{\ms W}}}\leq 2\norm{L}$.
    \end{itemize} 
 It follows that $\cl{L}\sim L$, and is the unique homogeneous quasimorphism in that cohomology class.
\end{lemma}

 \begin{proof}
Let $\bm{a},\bm{b}\in \ms{W}, \bm{u},\bm{u}',\bm{v},\bm{v}'\in\ms{W}_M$ such that $\lift{\bm{a}}=\bm{a}\bm{u}, \lift{\bm{b}}=\bm{b}\bm{v}, \bm{c}=\lift{\bm{a}}\star \lift{\bm{b}}=\bm{a}\bm{u}'\bm{b}\bm{v}'\in\Fix{\ms W}$. By definition of quasimorphism, 
\begin{align*}
& |L\left(\lift{\bm{a}}\star\lift{\bm{b}}\right)-L(\lift{\bm{a}})-L(\lift{\bm{b}})|\leq 4\norm{L}\\
& |L\left((\lift{\bm{a}}\star\lift{\bm{b}})^{2}\right)-L((\lift{\bm{a}})^{2})-L((\lift{\bm{b}})^{2})|\leq 16\norm{L}.
\end{align*}
Arguing by induction we get 
\begin{align*}
|L\left((\bm{c})^{n}\right)-L(\lift{\bm{a}})^{n})-L((\lift{\bm{b}})^{n})|=|L(\bm{a}\bm{u}'(\bm{b}\bm{v}'\bm{a}\bm{u}')^{n-1}\bm{b}\bm{v}')-L(\bm{a}\bm{u}(\bm{a}\bm{u})^{n-1})-L(\bm{b}\bm{v}(\bm{b}\bm{v})^{n-1})|\\
\leq 16\norm{L}+|L((\bm{b}\bm{v}'\bm{a}\bm{u}')^{n-1})-L((\lift{\bm{a}})^{n-1})-L((\lift{\bm{b}})^{n-1})|
\leq 16\norm{L}+16(n-1)\norm{L}=16n\norm{L}.
\end{align*}
Hence $\cl{L} :\ms W\to\R$ is a quasimorphism with $\norm{\delta \cl{L}}\leq 16\norm{L}$, and by the same token, $\norm{\cl{L}}\leq 32\norm{L}$. Define $L' :\Fix{\ms W} \to\R$, $L'(\bm{a})=L(\bm{a})+\norm{\delta L}$. Then $\cl{L'}=\cl{L}$ and $\cl{L'}(\bm{a})\leq L'(\bm{a})+\norm{\delta L}=L(\bm{a})+2\norm{\delta L}$. Applying the same $-L$, and noting that $\cl{-L}=-\cl{L}$ we get
\[
	\norml{\restr{(L-\cl{L})}{\Fix{\ms W}}}\leq 2\norm{L}.
\] 

Uniqueness is obvious.
 \end{proof}

 \begin{remark}\label{rem:homogeneoesciclico}
 Notice that $\cl{L_{\mrm{cyc}}}=\cl L$, and in particular $\cl L$ is cyclic.
 \end{remark}

One sees directly that  $\HQM{\ms W},\HQM{\Fix{\ms W}}$ are Banach isomorphic. 

\begin{corollary}
	The space $\QMt{\ms W}$ is Banach isomorphic to $\HQM{\Fix{\ms W}}$. 
\end{corollary}

\begin{proof}
	This is completely analogous to \Cref{cor:homogeneousqmareiso}, using \Cref{pro:cohomologyvscohomologyfixed}.
\end{proof}

 \subsection{Quasicocycles}\label{subsec:representationthm}

 When dealing with quasimorphisms $L:\Gamma\to\R$ defined on groups the basic problem concerning bounded cohomology is to determine whether or not $L$ is at bounded distance from a zero defect quasimorphism (i.e.\@ a group morphism). To translate the same question to quasimorphisms defined on a SFT we have to face the difficulty that the zero defect quasimorphism will typically be ill-defined, or trivial. To get around this, we will use that quasimorphisms can be used to induce a sequence of functions $(L^{(n)}:\Sigma\to\R)_{n\geq 0}$ with certain properties, and translate the corresponding cohomological problems into this setting.

\begin{definition}\label{def:quasicocycleL}
 For $L\in\QM{\ms W}$ and $n\geq 0$ we define $L^{(n)}:\Sigma\to\R$ by $L^{(n)}(\seq{x})=L(x_0\cdots x_{n-1})$. The family of continuous functions $(L^{(n)})_{n\geq 0}$ is the quasicocycle associated to $L$.
 \end{definition}

Given a a sequence of bounded Borel functions $\bm{B}=(B_n:\Sigma\to\R)_{n\geq 1}$ we denote by 
\begin{align}
&|\delta \bm{B}_{n,m}|=\sup_{\seq{x}} |B_{n+m}(\seq{x})-B_{n}(\seq{x})-B_m(\tau^n\seq{x})|\\
&\norm{\delta \bm{B}}=\sup_{n,m}|\delta \bm{B}_{n,m}|\\
&\norm[B]{\bm{B}}=\sup_{n\geq 1} \var_n(B_n)
\end{align}
where for a function $B:\Sigma\to\R$ we write $\var_n(B)=\sup\{|B(\seq{x})-B(\seq{y})|:[\seq{x}]_n=[\seq{y}]_n\}$.

\begin{definition}\label{def:quasicocycles}
A (Bowen) quasicocycle is a sequence $\bm{B}=(B_n)_{n\geq 1}$ of bounded Borel functions satisfying $\norm{\delta \bm{B}}+\norm[B]{\bm{B}}<\oo$. The set of all Bowen quasicocycles is denoted by $\QCB{\Sigma}$.

The quasicocycle is continuous if it consists of continuous functions, and is locally constant if $[\seq{x}]_n=[\seq{y}]_n$ implies $B_n(\seq{x})=B_n(\seq{y})$.

We denote $\QCB[c]{\Sigma}, \QCB[l]{\Sigma}$ the subspaces of continuous and locally constant quasicocycles, respectively. 
\end{definition}

In the literature sequences satisfying $\norm{\delta \bm{b}}<\oo$ are also called almost additive sequences. Note that $\norm{\delta \bm{B}}=0$ if and only if $B_n=\sum_{k=0}^{n-1}B_1\circ \tau$, and $\norm[B]{\bm{B}}=0$ if and only if $B_n$ is locally constant. Denote
\begin{align*}
\tnorm[B]{\bm{B}}=\norm{\delta \bm{B}}+\norm[B]{\bm{B}}+\norml{B_1}.
\end{align*}
It follows that $\left(\QCB{\Sigma},\tnorm[B]{\cdot}\right)$ is a Banach space.

We remark the following consequence of \cite{Cuneo2020}.

\begin{theorem}\label{thm:cuneo}
If $\bm{B}\in \QCB[c]{\Sigma}$ then there exists $\phi:\Sigma\to\R$ continuous so that
\[
\lim_n \frac{1}{n}\norml{B_n-S_n\phi}{\oo}=0.	
\]
\end{theorem}

The previous theorem however does not guarantee that $(S_n\phi)_n$ has also finite norm (i.e. that $\phi$ has the Bowen property).

Next we address cohomology of quasicocycles.

 \begin{definition}\label{def:cohomologyqc}
The quasicocycle $\bm{B}$ is said to be cohomologically trivial if $\sup_n \norml{B_n}<\oo$. Two quasicocycles $\bm{B}, \bm{B}'$ are cohomologous ($\bm{B}\sim\bm{B}'$) if their difference is cohomologically trivial.

We denote $[\bm{B}]:=\{\bm{B}':\bm{B}'\sim \bm{B}\}$ the cohomology class of $\bm{B}$, and write
\[
\QCBt{\Sigma}=\left\{[\bm{B}]:\bm{B}\in \QCB{\Sigma}\right\}.
 \]
 \end{definition}

The vector space $\QCBt{\Sigma}$ is equipped with the (complete) norm 
\begin{align}
 \tnorm[B]{[\bm{B}]}=\inf\{\tnorm[B]{\bm{B}'}:\bm{B}\sim\bm{B}'\}.
 \end{align}

\begin{proposition}\label{pro:cohqmigualcohqc}
The map $\Gamma: \QMt{\ms W}\to \QCBt{\Sigma}$ defined by $\Gamma([L])=[(L^{(n)})_n]$ is a continuous linear isomorphism, with $\frac{1}{M}\leq \norm[\mrm{OP}]{\Gamma}\leq 1$.

As a consequence, the spaces $\QCBt{\Sigma}, \QMt{\ms W}, \QMt{\Fix{\ms W}}$ are Banach isomorphic.
 \end{proposition}

\begin{proof}
It is direct to check that $\Gamma$ is an injective linear map, and since for every $L\in\QM{\ms W}$
it holds that $\norm[B]{(L^{(n)})_n}=0, \norm{\delta (L^{(n)})_n}\leq \norm{\delta L}, \norml{L^{(1)}}=\norml{\restr{L}{\ms{W}_1}}$, we deduce that $\norm[\mrm{OP}]{\Gamma}\leq 1$. It remains to show that $\Gamma$ is surjective and compute the norm of its inverse.

Let $\bm{B}$ be a given quasicocycle. For each $\bm{a}\in \ms{W}_n$ choose a point $\seq{x}_{\bm{a}}\in [\bm{a}]$ and define
\[
   L^{\bm{B}}(\bm{a})=B_{|\bm{a}|}(\seq{x}_{\bm{a}}).
 \]
If $\bm{a}\in \ms{W}_n, \bm{b}\in \ms{W}_m$ are concatenable words, then  
 \begin{align*}
 \begin{rcases}
 L^{\bm{B}}(\bm{a}\bm{b})=B_{n+m}(\seq{x}_{\bm{a}\bm{b}})\Rightarrow \left|L^{\bm{B}}(\bm{a}\bm{b})-B_{n}(\seq{x}_{\bm{a}\bm{b}})+B_{m}(\tau^{n}\seq{x}_{\bm{a}\bm{b}})\right|\leq |\delta\bm{B}_{n,m}|\\
|B_{n}(\seq{x}_{\bm{a}\bm{b}})-L^{\bm{B}}(\bm{a})|,\ |B_{m}(\tau^n\seq{x}_{\bm{a}\bm{b}})-L^{\bm{B}}(\bm{b})|\leq \norm[B]{\bm{B}}
 \end{rcases}\Rightarrow 
  \end{align*}
$\left|L^{\bm{B}}(\bm{a}\bm{b})-L^{\bm{B}}(\bm{a})-L^{\bm{B}}(\bm{b})\right|\leq 2\norm[B]{\bm{B}}+\norm{\delta\bm{B}}$ and $L^{\bm{B}}$ is a quasimorphism. Moreover $\Gamma([L])=[\bm{B}]$. Since $\norml{\restr{L^{\bm{B}}}{\ms{W}_{\leq M}}}\leq \norml{B_M}{\oo}\leq M\norml{B_1}$, it follows that
 \[
 	\norm{L^{\bm{B}}}\leq M\tnorm[B]{\bm{B}}\Rightarrow \norm[\mrm{OP}]{\Gamma^{-1}}\leq  M,
 \]
and this proves the first part. The second is direct from \Cref{pro:cohomologyvscohomologyfixed}.
\end{proof}

Thus we see that the cohomology class of any quasimorphism on a SFT is uniquely determined by a zero defect quasicocycle. Compared to the group case, this may seem surprising at first sight. However, a closer look reveals that the non-triviality of these cohomology classes is due to the fact that they are represented by locally constant functions (which necessarily carry some choices in their construction). At this point we have reached the point where we need further technology to understand the space $\QMt{\ms W}$.  For this, we introduce a coarse version of Livsic cohomology for dynamical systems, and show that both theories are isomorphic.


\section{Livšic cohomology on SFTs}\label{sec:LivsicSFT}

We start with an extension of a concept introduced by Bowen \cite{Bowen1974}. Throughout this section, $\Sigma$ denotes a fixed SFT.

\begin{definition}\label{def:Bowenpropertyfunction}
  A bounded function $\varphi:D_{\varphi}\subset\Sigma\to\R$ is said to be a weak Bowen function if 
    \begin{enumerate}
      \item $D_\varphi$ is dense and $\tau$-invariant.
      \item There is $C>0$ such that if $n\geq 1$ and $\seq{x},\seq{y}\in D_\varphi$ then
  \[
  \seq{y}\in[\seq{x}]_n\Rightarrow \left|S_n\varphi(\seq{x})-S_n\varphi(\seq{y})\right|\leq C.
  \]
  \end{enumerate}
We denote by $\norm[B]{\varphi}=\inf C$, and call it the Bowen constant of $\varphi$.   

If, in addition,

 \begin{enumerate}
  \item[3] $D_\varphi$ contains a full measure set for some $\mu\in \PTM{\tau}{\Sigma}$ of full support, we say that $\varphi$ is a $\mu$-weak Bowen function. In this case the Bowen norm of $\varphi$ is defined as $\tnorm[B]{\varphi}=\norm[B]{\varphi}+\norml{\varphi}$.
  \item[4] $D_\varphi=\Sigma$ then we say that $\varphi$ has the Bowen property.
 \end{enumerate}
\end{definition}

 \begin{notation}
 We denote by $\Bow{\Sigma} (\text{ resp. }\Bow[][\mu]{\Sigma})$ the set of weak ($\mu-$weak) Bowen functions on $\Sigma$. 
 \end{notation}

Usually, the literature deals with Bowen functions which are furthermore continuous. Let us justify introducing the weaker definition here by stating two theorems (to be proven in the next section) which will be used to establish \Cref{thm:A}. 

\begin{theorem}\label{thm:representation1}
Let $R\in\Mat_{d}(\{0,1\})$ be an aperiodic and irreducible matrix and consider $\Sigma$ the SFT that it determines. Given $\bm{B}\in \QCB{\Sigma}$ and $\mu\in \PTM{\tau}{\Sigma}$, there exist $E=E_{\bm{B}}>0$ and $\varphi_{\bm{B},\mu}\in \Lp{\oo}{\mu}$ defined on a full measure set $\Sigma_0(\mu)$ so that
\[
(\forall \seq{x}\in\Sigma_0(\mu), n\geq 1):\     |B_n(\seq{x})-S_n\varphi_{\bm{B},\mu}(\seq{x})|<E.
\]
\end{theorem}

 \begin{corollary}\label{cor:Bowenpotential}
 In the same hypotheses as the previous theorem, if $\mu\in \PTM{\tau}{\Sigma}$ has full support, then $\varphi_{\bm{B},\mu}$ is a $\mu-$weak Bowen function. 
 \end{corollary}

In this part we will look at how unique the function constructed in the previous theorem is: this will lead us to develop some version of Livšic cohomology \cite{Livshits1971} for $\Lp{\oo}-$functions. Complementary to the above, we also have.

\begin{theorem}\label{thm:representation2}
 Let $\Sigma$ be as above. If $\varphi :D_\varphi\to\R$ is a weak Bowen function then there exists a canonically defined measure $\mu_\varphi\in \PTM{\tau}{\Sigma}$ such that $D_\varphi^c$ is a $\mu_\varphi$-null set. In particular, $\mu_\varphi$ has full support.

 Moreover, this measure is ergodic.
 \end{theorem}

This shows that there is essentially no difference between weak Bowen functions and locally constant quasicocycles. For if $\varphi :D_\varphi\to\R$ is a weak Bowen function, then we can use the measure $\mu_\varphi$ to average over cylinders the sequence $(S_n\varphi)_n$ and obtain a locally constant quasicocycle associated with $\varphi$. Since this procedure (conditioning with respect to some increasing filtration of $\sigma-$algebras) will be used repeatedly, we recall its construction below.

\paragraph{\textbf{Conditional expectation}} We consider the partitions $\xi^{n}$ defined by the Bowen balls of length $n+1$ together with the associated filtration of $\sigma-$algebras in $\BorelM[\Sigma]$,  
\[
   \ms{B}^n=\begin{dcases}
   \{\emptyset,\Sigma\} & n=-1\\
   \sigma-\text{algebra generated by }\xi^{n} & n\geq 0.
   \end{dcases}
\]
 We write $\xi=\xi^{0}=\{[a]:a\in\ms{A}\}$, and hence $\xi^{n}=\bigvee_{k=0}^{n-1}\tau^{-k}\xi$. It follows that for any probability measure $\mu\in \ProbM[\Sigma]$,
 \[
    \Emu{\mu}{f \given \xi^n}=\Emu{\mu}{f\given \ms{B}^n}=\sum_{\substack{A\in\xi^n\\ \mu(A)\neq 0}}\frac{1}{\mu(A)}\int_A f \dd\mu\cdot \one_A
  \]
where $\one_A$ is the characteristic function of $A$.

Consider $\varphi:D_\varphi\to\R$ with the weak Bowen property, and let $\mu_{\varphi}$ be the $\tau$-invariant measure associated to it given in \Cref{thm:representation2}. Define $\bm{B}^{\varphi}=(B_n^{\varphi}):\Sigma\to\R$ with

 \begin{align}\label{eq:Basociadovarphi}
 B_n^{\varphi}=\Emu{\mu_{\varphi}}{S_n\varphi \given \xi^n}.
 \end{align}

\begin{lemma}\label{lem:boundonpotential}
 It holds that $\bm{B}^{\varphi}\in \QCB[l]{\Sigma}$, and furthermore $\norm{\delta \bm{B}^{\varphi}}\leq 6\norm[B]{\varphi}$. 
 \end{lemma}

\begin{proof}
 By construction $\bm{B}^{\varphi}$ consists of locally constant functions. If $\bm{a}\in \ms{W}_n, \seq{x}\in[\bm{a}]\cap D_\varphi, \seq{y}\in [\bm a]$, then $|B_n^{\varphi}(\seq{y})-S_n\varphi(\seq{x})|\leq 2\norm[B]{\varphi}$. From here the bound on the norm follows directly, and thus $\bm{B}^{\varphi}\in \QCB[l]{\Sigma}$. 
\end{proof}

 There is a natural notion of cohomology associated with weak Bowen functions which, in case of functions with the Bowen property, coincides with that given by the Livšic theorem, i.e., two functions with the Bowen property are cohomologous in the Livšic sense if and only if they have the same integral with respect to any $\tau$-invariant measure. To explain this we first define the integral for this more general class of functions. 

\paragraph{\textbf{Integration}}We define integrals of quasicocycles with respect to invariant measures. If $\bm{B}=(B_n)_n\in\QCB{\Sigma}$ and $\mu\in \PTM{\tau}{\Sigma}$, we define
\[
 \mu(\bm{B}):=\lim_{n\to\infty}\frac{1}{n}\int B_n \dd\mu.
 \]
 Existence of this limit follows from sub-additivity of the real valued sequence $(\int B_n \dd\mu)_n$.

 \begin{example}\label{ex:integraldeL}
 Let $L\in\QM{\Fix{\Sigma}}$ and construct an associated quasicocycle $\bm{B}^L=(L^{(n)})_n\in\QCB{\Sigma}$ by proceeding first as in \Cref{pro:cohomologyvscohomologyfixed}, and then considering the associated quasicocycle. If $\mu\in \PTM{\tau}{\Sigma}$ then
 \[
 	\mu(L) \defeq \mu(\bm{B}^L)=\lim_{n}\frac{1}{n}\sum_{\bm{a}\in \ms{W}_n} \mu([\bm{a}]) L(\bm{a}).
 \]
 It is worth noticing that $\mu(L)$ does not depend on the choices made to extend $L$ to a fully defined quasimorphism, and likewise, it is invariant in the cohomology class $[L]\in \QMt{\Fix{\Sigma}}$.
\end{example}

\begin{proposition}\label{pro:integracioncontinua}
If $\bm{B}=(B_n)_n \in \QCB[c]{\Sigma}$ then the map $\PTM{\tau}{\Sigma}\ni\mu\to \mu(\bm{B})$ is weakly continuous. As a consequence, if $L$ is a quasimorphism of a mixing SFT then $\PTM{\tau}{\Sigma}\ni\mu\to \mu(L)$ is weakly continuous.
\end{proposition}

\begin{proof}
The first part follows from \Cref{thm:cuneo}: if $\phi\in \Cr{\Sigma}$ is so that $\lim_n \frac{1}{n}\norml{B_n-S_n\phi}=0$, then for every $\mu\in \PTM{\tau}{\Sigma}$ it holds $\mu(\bm{B})=\int \phi\dd\mu$.

The second part follows by observing that the quasicocycle associated to a quasimorphism consists of continuous (locally constant) functions.
\end{proof}

\paragraph{Periodic points} Let us investigate the behavior at periodic points. For $\seq{x}\in \Fix[][N]{\tau}$, let $\mu_{\seq{x}}=\frac{1}{N}\sum_{j=0}^{N-1}\delta_{\tau^j\seq{x}}$ be the invariant probability measure supported on the orbit of $\seq{x}$. Each periodic point $\seq{x}$ is associated to a periodic word $\bm{a}\in\Fix{\ms W}$ ($\seq{x}=\seq{p}(\bm{a})$), thus we write $\mu_{\bm a}=\mu_{\seq{x}}$.

 \begin{lemma}
   If $\seq{x}\in \Fix[][N]{\tau}$ and $m>0$ then 
   \[
   \left|\frac{B_{m N}(\seq{x})}{m}-B_{N}(\seq{x})\right|\leq \norm{\delta \bm B}.
   \]
 \end{lemma}
The proof is straightforward.

\begin{lemma}\label{lem:medidasperiodicasac}
 	If $\seq{x}\in \Fix[][N]{\tau}$ then 
    \[
    \mu_{\seq{x}}(\bm{B})=\lim_{n\to\infty}\frac{B_{n}(\seq{x})}{n}.
    \]
\end{lemma}

\begin{proof} By definition,
\[
 \mu_{\seq{x}}(\bm{B})=\frac{1}{N}\lim_{n\to\infty}\sum_{j=0}^{N-1}\frac{1}{n}B_{n}(\tau^j\seq{x}).
 \]
 For $j$ fixed the previous lemma implies that $(\frac{1}{n}B_n(\tau^j\seq{x}))_{n}$ is convergent, and since $|B_{n+j}(\seq{x})-B_j(\seq{x})-B_n(\tau^j\seq{x})|\leq \norm{\delta \bm{B}}$ we get that $\lim_{n\to\infty}\frac{1}{n}B_n(\tau^j\seq{x})=\lim_{n\to\infty}\frac{1}{n}B_n(\seq{x})$, which finishes the proof.
 \end{proof}

 \begin{example}\label{ex:integralLrespperiodicas}
 Let $L\in\QM{\Fix{\ms W}}$ and let $\bm{B}^L$ be an associated quasicocycle as in the previous example. For $\bm{a}\in\Fix[][N]{\ms W}$, $n\geq 1$ we have $L^{nN}(\seq{p}(\bm{a}))=L(\bm{a}^n)$, we deduce
 \[
 	\mu_{\bm a}(L)=\lim_n \frac{1}{n}L(\bm{a}^n)=\cl{L}(\bm{a}).
 \]
 In particular, if $L\in\HQM{\ms W}$, $\mu_{\bm a}(L)=L(\bm{a})$.
 \end{example}

\begin{proposition}\label{pro:integralalmostbowenfunctions}
Let $\varphi:D_\varphi\to\R$ be a weak Bowen function. Then for every $\mu\in \PTM{\tau}{\Sigma}$ such that $D_\varphi$ contains a full $\mu$-set it holds $\int \varphi\dd\mu=\mu(\bm{B}^{\varphi})$.
\end{proposition}

\begin{proof}
For each $n\ge1$ one has
\[
\int \varphi\dd\mu
=
\frac1n\int S_n\varphi \dd\mu
=
\frac1n\int B_n^\varphi\dd\mu+\frac{\varepsilon_n}{n},
\]
where $|\varepsilon_n|\le 2\norm[B]{\varphi}$. Letting $n\to\infty$ gives the result.
\end{proof}

 \begin{notation}
  Due to the above, if $\varphi:D_\varphi\to\R$ is a weak Bowen function and $\mu\in \PTM{\tau}{\Sigma}$, we write $\mu(\varphi)=\mu(\bm{B}^{\varphi})$.
  \end{notation}

\subsection{A Livšic type theorem for quasicocycles.}
 \label{ssec:Livsicquasicocycle}

Observe the following simple but important fact.

\begin{lemma}\label{lem:cohtrivialimplieszerointegral}
 If $\bm{B}\in\QCB{\Sigma}$ is cohomologically trivial then for every $\mu\in \PTM{\tau}{\Sigma}$ it holds $\mu(\bm B)=0$. 
\end{lemma}
In this part we establish the converse.

 \begin{example}\label{ex:potentialescohomologoaL}
Let $L\in\QM{\Fix{\Sigma}}$ be a quasimorphism defined on a mixing SFT, and consider the potential $\varphi_L:\Sigma_0\to\R$ associated to $L$ as given in \Cref{thm:representation2}. Then $\bm{B}^{\varphi_L}\in \QCB[l]{\Sigma}$ is cohomologous to $\bm{B}^L$.  It follows directly from \Cref{pro:integralalmostbowenfunctions} and the previous lemma that $\int \varphi_L\dd\mu_{\varphi_L}=\mu_{\varphi_L}(\bm{B}^L)$.
\end{example}

Let us spell some basic considerations.

\begin{lemma}\label{lem:integralBnper}
Let $\seq{x}\in [\bm a], \bm a\in\Fix[][n]{\ms W}$. Then 
\[
\left|B_{n}(\seq{x})-n\cdot\mu_{\bm a}(\bm{B})\right|\leq \tnorm[B]{\bm B}.
\]
\end{lemma}

\begin{proof} Direct computation:
\[
|B_n(\seq{p}(\bm{a}))-B_n(\seq{x})|\leq \norm[B]{\bm B}
\]
and
\[
\left|n\mu_{\bm a}(\bm B)-B_n(\seq{p}(\bm{a}))\right|=\left|n\lim_{m\to\infty}\frac{B_{m n}(\seq{p}(\bm{a}))}{m n}-B_n(\seq{p}(\bm{a}))\right|\leq \norm{\delta \bm B}.
\]
\end{proof}

\begin{corollary}\label{cor:medidaceroimplicaacotado}
   If $\bm B\in\QCB{\Sigma}$ and $\mu(\bm B)=0$ for every $\mu\in\PTM{\tau}{\Sigma}$, then 
  \[
   \sup_n \norml{B_n}\leq \norm[B]{\bm B}+\norm{\delta \bm B}.
   \]
  \end{corollary}

 \begin{corollary}\label{cor:wBowenboundedsum}
 Let $\varphi\in\Bow{\Sigma}$ be such that $\mu(\bm{B}^{\varphi})=0$ for every $\mu\in\PTM{\tau}{\Sigma}$. Then, for every $\seq{x}$ on its domain
  \[
   |S_n\varphi(x)|\leq 6\norm[B]{\varphi}.
   \] 
 \end{corollary}

\begin{proof}
Use \cref{lem:boundonpotential}.
\end{proof}

After these preparations we can make a meaningful remark.

\begin{proposition}[Livšic theorem for quasicocycles]\label{pro:weakLivsic}
 Let $\bm B, \bm B'\in \QCB{\Sigma}$. Then $\bm B\sim \bm B'$ if and only if for every $\bm a\in\Fix{\Sigma}$ it holds $\mu_{\bm a}(\bm B)=\mu_{\bm a}(\bm B')$.  
 \end{proposition}

\begin{proof}
 The necessary condition is spelled in \Cref{lem:cohtrivialimplieszerointegral}, while the sufficient condition follows directly from \Cref{cor:medidaceroimplicaacotado}. 
 \end{proof}

 \begin{corollary}[Livšic theorem for quasimorphisms]\label{cor:Livsicqm}
 Let $\Sigma$ be a mixing SFT and $L, L'\in \QM{\ms W}$. Then $L\sim L'$ if and only if for every $\bm a\in \Fix{\ms W}$ it holds $\cl{L}(\bm a)=\cl{L}'(\bm a)$. It follows that $L\sim L'$ if and only if $\restr{L}{\Fix{\ms W}}\sim\restr{L'}{\Fix{\ms W}}$.
\end{corollary}

\begin{proof}
It is no loss of generality to assume that $L,L'$ are homogeneous; the statement of the corollary is direct from the previous proposition and the considerations given in \cref{ex:integralLrespperiodicas}.
\end{proof}

\paragraph{\textbf{A general Livšic theorem}}Next we deal with cohomology for weak Bowen functions. The central theorem that we will establish is the following.

\begin{theorem}[Livšic's Theorem for weak Bowen functions]\label{thm:livsicWeakBowen}
Let $\varphi\in\Bow[][\nu]{\Sigma}$ for some $\nu\in \PTM{\tau}{\Sigma}$, and assume that for every $\mu\in \PTM{\tau}{\Sigma}$ it holds $\mu(\varphi)=0$. Then there exists $u\in \Lp{\oo}{\nu}$ such that 
\[
  u-u\circ \tau=\varphi \qquad  \aee{\nu}
\]
Moreover, $\norm[\Lp{\oo}{\nu}]{u}\leq 6\norm[B]{\varphi}$.
\end{theorem}

\begin{notation}
 The Koopman operator induced by $\tau$ with respect to some invariant measure $\nu$ is $T:\Lp{2}{\nu}\toit$, $T\phi=\phi\circ \tau$. We write $T^*$ for its adjoint.
\end{notation}

\begin{lemma}\label{lem:livsicL2}
  Let $\nu\in \PTM{\tau}{\Sigma}, \phi\in \Lp{2}{\nu}$ and define 
    \[
    u_N=\frac{1}{N}\sum_{k=1}^NS_k\phi
    \]

 Assume that 
 \[
\liminf_{N\to\oo} \norm[L^2]{u_N}<\infty,
 \]
 
 (this is true for example if $\sup_k \norm[L^2]{S_k\phi}<\infty$). Then the sequence $(u_N)_N$ converges both in $\Lp{2}{\nu}$ and $\aee{\nu}$ to some function $u\in \Lp{2}{\nu}$ which is orthogonal to the invariant $\sigma$-algebra of $\tau$. That is, for every invariant set $A$, it holds true $\int_A u\dd\nu=0$. Moreover
 \begin{enumerate}
 	\item $\phi=u-Tu$.
 	\item \(u-\frac{1}{N}S_Nu=u_N.\)
 	 \end{enumerate}
\end{lemma}

\begin{proof}
Take $u$ to be a weak accumulation point of $(u_N)_N$ in $\Lp{2}{\nu}$: for some sub-sequence $(u_{N_i})_i$ and for every $w\in \Lp{2}{\nu}$, it holds $\inner{u_{N_i}}{w}\xrightarrow[i\to\oo]{} \inner{u}{w}$. Notice that if $A$ is invariant, $\inner{u_N}{\one_A} =\frac{1}{N}\sum_{k=1}^Nk \inner{\phi}{\one_A}=\frac{N+1}{2}\inner{\phi}{\one_A}$. 

To satisfy the convergence condition, necessarily $\inner{\phi}{\one_A}=0$ for every invariant set, hence due to the von Neumann theorem,
   \[
    \frac{1}{N}\norm*[\Lp{2}]{\sum_{k=1}^N\phi\circ \tau^k} \xrightarrow[N\mapsto\oo]{} 0.
    \]
 Now, 
 \begin{align*}
 u_N-u_N\circ \tau=\frac{1}{N}\sum_{k=1}^N(S_k\phi-S_k\phi\circ \tau)=\frac{1}{N}\sum_{k=1}^N(\phi-\phi\circ \tau^k)=\phi-\frac{1}{N}\sum_{k=1}^N\phi\circ \tau^k,
 \intertext{therefore for $w\in \Lp{2}{\nu}$,}
 \lim_{i\to\oo} \inner{u_{N_i}}{w}-\inner{u_{N_i}\circ \tau}{w}=\lim_{i\to\oo} \inner{u_{N_i}}{w}-\inner{u_{N_i}}{T^*w}=\inner{u-u\circ \tau}{w} \intertext{and}
 \inner{u-u\circ \tau}{w}=\lim_{i\to\oo} \inner{\phi-\frac{1}{N_i}\sum_{k=1}^{N_i}\phi\circ \tau^k}{w}=\inner{\phi}{w}.
 \end{align*}
 This gives
 \[
   \inner{u-u\circ \tau-\phi}{w}=0\quad\forall w\in \Lp{2}{\nu}
 \]
 which shows that any accumulation point $u$ of $(u_N)_N$ solves the cohomological equation for $\phi$.

Next, for $u$ as above we have $u-u\circ \tau^k=S_k\phi$ and hence summing over $k=1,\ldots, N$ we get that 
 \[
 u-\frac{1}{N}\sum_{k=1}^Nu\circ \tau^k=u_N\Rightarrow u-u_N=\frac{1}{N}\sum_{k=1}^Nu\circ \tau^k.		
 \]
Since we showed that $\int_A u\dd\nu=0$ for every invariant set $A$, the Birkhoff and von Neumann ergodic theorems applied to $u$ give the pointwise and $\Lp{2}-$convergence. 
\end{proof}

\begin{remark}
In the previous lemma, and assuming $\sup_k \norm[L^2]{S_k\phi}<\infty$, it is classical to show that the cohomological equation has a solution in $\Lp{2}$: consider $\lie{X}$ the set of convex combinations of $(S_k\varphi)$ and note that its weak $\Lp{2}$ closure is compact, and of course convex. The map $u\in\cl{\lie{X}}\to \phi+u\circ T$ preserves $\cl{\lie{X}}$ and is continuous, so by the Schauder-Tychonoff theorem, there exists $u\in \cl{\lie{X}}$ satisfying $u=\phi+u\circ T$, hence $u$ satisfies the cohomological equation in $\Lp{2}$. It follows that there exists $(v_n)_n \subset \lie{X}$ such that $\norm[\Lp{2}]{u-v_n}\xrightarrow[n\mapsto\oo]{}0$. Note however that this argument does not give the pointwise convergence statements, which will allow us to improve the regularity of $u$. 
\end{remark}

\begin{corollary}\label{cor:GHin}
  Under the same hypotheses as in the previous theorem, if 
  \[
  \liminf_{N\to\oo} \norm[\Lp{\oo}]{u_N} <\infty,
  \]
  (for example, if $\sup_k \norm[\Lp{\oo}]{S_k\phi}<\infty$), then there exists $u\in \Lp{\oo}(\nu)$ with $\norm[\Lp{\oo}]{u}\leq \liminf_N\norm[\Lp{\oo}]{u_N}$ so that $\lim_N u_N=u$ both $\aee{\nu}$ and in $\Lp{\oo}$. Moreover, $\phi=u-u\circ\tau$.
\end{corollary}
\begin{proof}
Since $\liminf_{N\to\oo} \norm[\Lp{2}]{u_N} <\infty$, we can apply the previous theorem and deduce $\aee{\nu}$ convergence $u_N\xrightarrow[N\to\oo]{} u$. The dominated convergence theorem implies the rest.
\end{proof}

We now are ready to prove \Cref{thm:livsicWeakBowen}.

\begin{proof}
 We can assume that the domain of $\varphi$ is a measurable, invariant $\nu-$full measure set. Let $u_N=\frac{1}{N}\sum_{k=1}^NS_k\varphi$. By hypotheses and \Cref{cor:wBowenboundedsum} we get that $|u_N(x)|\leq 6\norm[B]{\varphi}$, for every $x\in D_\varphi$ and $N\geq 1$, therefore the result follows from \Cref{cor:GHin}.
 \end{proof}

 \begin{definition}\label{def:Livsiccoboundaries}
 Let $\nu\in \PTM{\tau}{\Sigma}$. A function $\phi\in\Lp{\oo}{\nu}$ is a $\Lp{\oo}$-coboundary if $\phi=u-u\circ\tau$ for some $u\in \Lp{\oo}{\nu}$. The space of coboundaries in $\Lp{\oo}{\nu}$ is denoted by $\Cob[\oo][\nu]{\Sigma}$.
 \end{definition}

Note that $\Cob[\oo][\nu]{\Sigma} \subset \left(\Bow[][\nu]{\Sigma},\tnorm{\cdot}\right)$ is a closed subspace. We equip 
 \begin{align}\label{def:modcoboundary}
  \Bow[][]{\Sigma}/\thicksim \defeq \frac{\Bow[][\nu]{\Sigma}}{\Cob[\oo][\nu]{\Sigma}}
 \end{align}
 with the induced norm.

\begin{theorem}\label{thm:isomorphismqmBowmeasure}
Let $\nu\in \PTM{\tau}{\Sigma}$ be a fully supported measure. Then there exists a Banach isomorphism between $\QMt{\ms W}$ and $\Bow[][\nu]{\Sigma}/\thicksim$.
\end{theorem}

\begin{proof}
Given $L\in\QM{\ms W}$ consider its associated quasicocycle $\bm B^{L}$, and the corresponding associated potential $\varphi_{L}:\Sigma_0\to \R$ as given in \Cref{thm:representation1}. Using \Cref{pro:cohqmigualcohqc} and \Cref{lem:boundonpotential} we get that $\bm{B}^{\varphi_L}\sim \bm{B}^L$, and since $\bm{B}^L$ is cohomologically trivial if and only if $L$ is cohomologically trivial in $\QM{\ms W}$, we deduce that $L$ is cohomologically trivial if and only if $\bm B^{\varphi_L}$ is uniformly bounded. Due to \Cref{thm:livsicWeakBowen}, this happens if and only if $\varphi_L=u-u\circ \tau$, for some $u\in \Lp{\oo}(\nu)$. 

The above shows that the linear map $\Gamma: \QMt{\ms W}\to \Bow[][\nu]{\Sigma}/\thicksim, \Gamma([L])=[\varphi_L]$ is bijective, and it is direct to check that it is continuous, hence a Banach isomorphism.
\end{proof}

\subsection{Livšic cohomology for suspension flows}

\paragraph{\textbf{Bowen functions in the natural extension}}

 Consider $\lift\Sigma$, a two-sided SFT. The notion of weak Bowen ($\mu-$ weak Bowen function) is exactly the same in this context as in the one-sided shift. Note that the proof of \Cref{thm:livsicWeakBowen} applies to this type of function without modifications.

 Let $\pi:\lift\Sigma\to\Sigma$ be the projection onto the non-negative coordinates, $\pi(\seq{x})=\seq{x}^{+}=(x_n)_{n\in\N}$: this map is continuous, surjective and verifies $\pi\circ \tau=\tau\circ\pi$ (that is, it is a semi-conjugacy between $\tau:\lift\Sigma\toit$ and $\tau:\Sigma\toit$). Moreover, it induces an isomorphism $\pi_{\#}:\PTM{\tau}{\lift \Sigma}\to \PTM{\tau}{\Sigma}$ which can be characterized as follows (cf.\@ page $28$ in \cite{zeta}). Fix $\mu\in \PTM{\tau}{\Sigma}$. Given $f\in \Cr{\Sigma}$, it induces naturally an element $\lift{f}\in \Cr{\lift \Sigma}$ by $\lift{f}(\seq{x})=f(\seq{x}^+)$. Define, for $n\geq 0$,
\[
   \lift \mu(\lift{f}\circ \tau^{-n}) \defeq \mu(f).
 \]
This determines a bounded positive linear functional $\lift\mu$ on the dense subspace $\{\lift{f}\circ \tau^{-n}:n\geq 0, f\in \Cr{\Sigma}\} \subset \Cr{\lift \Sigma}$, hence it extends uniquely to a bounded linear functional $\lift\mu:\Cr{\lift \Sigma}\to \R$ with $\norm[OP]{\lift\mu}=\norm[OP]{\mu}=\mu(1_{\Sigma})=1$, therefore it is a probability on $\lift\Sigma$.  Furthermore, $\mu$ is ergodic if and only if $\lift\mu$ is ergodic.

\begin{notation}
  If  $\mu \in \PTM{\tau}{\Sigma}$, $\lift\mu$ denotes its unique lift to $\PTM{\tau}{\lift\Sigma}$.
\end{notation}

\begin{proposition}\label{pro:Bowenonesidedcohtwosided}
  If $\mu\in \PTM{\tau}{\Sigma}$ is of full support, then the spaces $\Bow[][\mu]{\Sigma}/\thicksim$ and $\Bow[][\lift\mu]{\lift\Sigma}/\thicksim$ are Banach isomorphic.
 \end{proposition}

\begin{proof}
Consider the linear continuous mappings 
\begin{align*}
&\Gamma_1: \Bow[][\lift\mu]{\lift\Sigma}\ni\phi\mapsto \left(\bm B=(\Emu{\lift\mu}{S_n\phi \given \xi^n})_n\right)\in\QCB{\Sigma} \mapsto \varphi_{\bm B,\mu}\in \Bow[][\mu]{\Sigma}\\
&\Gamma_2:  \Bow[][\mu]{\Sigma}\ni \phi\to \phi\circ \pi\in \Bow[][\lift\mu]{\lift\Sigma}.
\end{align*}
By definition we have that $\lift\mu(\phi)=\lift\mu(\bm B)$, and since $\bm B$ is constant on $\xi^n$, $\lift\mu(\bm B)=\mu(\bm B)$. Now $\bm{B}\sim (S_n\varphi_{\bm B,\mu})_n$, and thus $\lift\mu(\phi)=\mu(\varphi_{\bm B,\mu})$. From this and \Cref{thm:livsicWeakBowen} one gets that $\Gamma_1$ induces a continuous injective linear map $\Bow[][\lift\mu]{\lift\Sigma}/\thicksim\to\Bow[][\mu]{\Sigma}/\thicksim$. Using $\Gamma_2$ it follows that this map is surjective as well, hence the result.
\end{proof}

\begin{remark}\label{rem:onesidedvstwo}
It is a classical result of Sinai that any Hölder function on $\lift\Sigma$ is continuously cohomologous to a Hölder function on $\Sigma$. On the other hand, even for well behaved measures (as the entropy maximizing measure of $\tau$) there are $\Lp{\oo}-$coboundaries that are not continuously cohomologous to any function depending only on the non-negative coordinates, thus, not continuously cohomologous to any Hölder function \cite{Quas_1997}. In particular, if $\phi$ is such a function, then $\phi$ is not continuously cohomologous to $\Gamma_1(\phi)\circ\pi$.
This shows the necessity of working with these broader cohomology theories instead of the classical versions using continuous functions.
\end{remark}

\paragraph{\textbf{Flows}}Consider $\lie{t}=(\tau_t)_{t\in\R}:\lift{\Sigma}_f\toit$ a suspension flow as in \Cref{thm:symbolicrepresentation}, and for $\nu \in \PTM{\tau}{\lift{\Sigma}}$ let $\nu_{\lie t} \in \PTM{\lie t}{\lift{\Sigma}_f}$ be the unique invariant measure that it determines (see \Cref{eq:medidasuspension}). Since $\lie{t}$ is a metric Anosov flow, one can define $\nu_{\lie t}-$Bowen functions and their Livšic cohomology in the same way as in \Cref{def:weakBowenfunction,def:Livsiccohomologous}.

It is convenient to define  $P:\lift\Sigma_f\toit$ the return map to the section $S=\{[\seq{x},0]:\seq{x}\in\lift\Sigma\} \subset \lift\Sigma$,
 \[
   P([\seq{x},t])=[\tau(\seq{x}),0].
 \]
 Note that $P$ preserves $\nu_{\lie t}$. If $\psi\in\Bow[][\nu_{\lie t}]{\Sigma_f}$ define
\begin{align}\label{eq:potencialintegrado}
&\hat \psi([\seq{x},t])=\int_0^{f(\seq{x})-t} \psi([\seq{x},s])\dd s\\
&\hat \psi(\seq{x})=\hat \psi([\seq{x},0])=\int_0^{f(\seq{x})} \psi([\seq{x},s])\dd s.
 \end{align}

 \begin{lemma}\label{lem:Livsicreturnmap}
 Let $\psi\in \Bow[][\nu_{\lie t}]{\Sigma_f}$. Then $\psi$ is a Livšic coboundary with respect to $\lie t$ if and only if $\hat \psi$ is a Livšic coboundary for the return map $P$. 
 \end{lemma}
The proof is standard, and is essentially an application of the fundamental theorem of Calculus.

 On the other hand, if $\hat\psi$ is a Livšic coboundary with respect to the dynamics $\tau$, $\hat\psi(\seq{x})=u(\seq{x})-u(\tau x)$ let
\[
 U([\seq x,t])=u(\seq{x})-\int_0^t \psi([\seq{x},s])\dd s.   
 \]
$U\in\Lp{\oo}{\lift \nu}$ and by direct computation $U([\seq{x},t])-U(P[\seq{x},t])=\hat\psi([\seq{x},t])$, which shows that $\hat\psi$ is a coboundary with respect to $P$, hence $\psi$ is a Livšic coboundary for the flow $\lie t$.

If $\alpha$ is a periodic orbit of $\lie t$ we denote, as in the case of the geodesic flow,
\[
   \av_\alpha(\phi)=\lim_{T\to\oo}\frac{S_T\phi([\seq{x},0])}{T}
 \]
where $[\seq{x},0]\in\alpha$: of course, $\av_\alpha(\phi)$ is just the integral of $\phi$ with respect to the periodic measure supported on $\alpha$. We remark the following simple fact.

\begin{lemma}\label{lem:bijectionperiodic}
There exists a bijection between the periodic orbits of $\lie t$ and the periodic orbits of $\tau:\Sigma\toit$.
\end{lemma}

In view of the above it makes sense to consider $\av_{\bm a}$ for $\bm a\in\Fix{\ms W}$ 

\begin{corollary}\label{cor:Livsicsuspensions}
The map $\Gamma: \Bow[][\nu_{\lie t}]{\Sigma_f}\to \Bow[][\lift\nu]{\lift\Sigma}$ given by $\Gamma(\phi)=\hat \phi$ induces a Banach isomorphism $\Gamma_*: \Bow[][\nu_{\lie t}]{\Sigma_f}/\thicksim\to \Bow[][\lift\nu]{\lift\Sigma}/\thicksim$.

Moreover, $\psi,\psi'\in \Bow[][\nu_{\lie t}]{\lift{\Sigma}_f}$ are Livšic cohomologous if and only if for every $\bm a\in\Fix{\Sigma}$, $\av_{\bm a}(\hat\psi)=\av_{\bm a}(\hat\psi')$. 
\end{corollary}

\begin{proof}
It is direct to check that $\Gamma$ is linear and continuous. Note that given $\varphi \in \Bow[][]{\lift\Sigma}$, $\varphi=\Gamma(\psi)$, where
$\psi([\seq x,0])=\frac{1}{f(\seq x)}\varphi(\seq x)$. Hence by the discussion above induces a bijective linear isomorphism $\Gamma_*$ between the quotients. Note that for $\seq{y}\in [\seq{x}]_n$,
\[
  |S_n\psi([\seq{x},0])-S_n\psi([\seq{y},0])|=|S_n\hat \psi(\seq{x})-S_n\hat \psi(\seq{y})|.
\]
From here one deduces that $\Gamma_*$ is continuous with respect to the corresponding Bowen norms, hence a Banach isomorphism.

The second part is a consequence of \Cref{thm:livsicWeakBowen}: necessity is clear, while for sufficiency, $\av_{\bm a}(\psi)=\av_{\bm a}(\psi')$ implies that $\hat\psi\sim\hat\psi'$, hence $\psi\sim\psi'$.
\end{proof}

\paragraph{\textbf{Back to geodesic flows}} Finally, let $\lie g=(g_t)_{t\in\R}:\ms E\toit$ be a geodesic flow corresponding to a closed locally $\CAT$ space, and consider a suspension flow $\lie{t}= (\tau_t):\lift{\Sigma}_f\toit$ as in \Cref{thm:symbolicrepresentation}; the corresponding semi-conjugacy is $h:\lift{\Sigma}_f\to\ms E$.

 \begin{corollary}\label{cor:Livsicgeodesicflo}
In the hypotheses above, consider a fully supported measure $\nu_{\lie g}\in \PTM{\lie{g}}{\ms E}$. Then $\phi \in \Bow[][\nu_{\lie g}]{\ms E}$ is a Livšic coboundary if and only if for every $\alpha$ periodic orbit of $\mathfrak g$, $\av_{\alpha}(\phi)=0$.
 \end{corollary}

\begin{proof}
Consider the isomorphism $h_*:\Bow[][\nu_{\lie t}]{\lift{\Sigma}_f}\to\Bow[][\nu_{\lie g}]{\ms E}$ (\Cref{cor:BowengBowent}), which preserves cohomology classes. If $\av_{\alpha}(\phi)=0$ for every periodic orbit of $\lie{g}$, then $\av_{\alpha'}(h^{-1}_*\phi)=0$ for every periodic orbit $\alpha'$ of $\lie t$ that lies in some open dense set $D \subset \lift{\Sigma}_f$.

This is enough to guarantee that the quasicocycle defined by $h^{-1}_*\phi$ is uniformly bounded, hence its integral with respect to every invariant measure of $\lie t$ is zero. Due to \Cref{cor:Livsicsuspensions}, we deduce that $h^{-1}_*\phi$ is a Livšic coboundary, hence $\phi$ is a Livšic coboundary as well. Necessity is obvious.
\end{proof}

The previous Corollary in fact implies \Cref{thm:E}, where no assumptions on ergodicity are made. This is consequence of the fact that any invariant measure can be written as an integral of ergodic measures (Ergodic Decomposition Theorem). The proof of this theorem can be found for example in \cite{Petersen1983}.

\section{Ergodic theory of quasimorphisms}

In this section we develop a thermodynamic formalism for quasimorphisms on SFTs (equivalently, quasicocycles), and use it to characterize their cohomology classes by means of certain invariant measures for the shift. These measures are uniquely associated with Livšic cohomology classes of weak Bowen functions (this is the content of \Cref{thm:representation1}, which we prove here), and we explain how to construct them and analyze their dynamical properties.

\subsection{Pressure for quasimorphisms}
\label{ssec:pressureSFT}

\paragraph{\textbf{Quasi-bounded sequences}}Given a sequence of real numbers $\seq{a}=(a_n)_n$, denote
 \begin{align}
  &\delta_{n,m} (\seq{a})=a_{n+m}-a_n-a_m\\
  &\norm{\delta \seq a}=\sup_{n,m} |\delta_{n,m} \seq a|.
 \end{align}

Recall that $\seq a$ is subadditive if $\forall n, m$, $\delta(\seq{a})_{n,m}\leq 0$ (i.e.\@ $a_{n+m}\leq a_n+a_m$). We denote by $\mc{S}_{sa}$ the set of all sub-additive sequences.  If $\seq a \in \mc{S}_{sa}$, then for every $m=kn+r, 0\leq r<n$ we have
\[
  \frac{a_m}{m}-\frac{a_n}{n}\leq \frac{1}{r}a_r.
\]
It follows that for such a sequence,
\[
  \exists \lim_n \frac{a_n}{n}=\inf \frac{a_n}{n}\in \R\cup\{-\oo\}.
\]
 
\begin{corollary}\label{cor:subadditivepositive}
 If $\seq a \in \mc{S}_{sa}$ and $\lim_n \frac{a_n}{n}=0$ then $a_n\geq 0$ for every $n$.
\end{corollary}

\begin{definition}\label{def:quasiboundedseq}
A sequence $\seq a$ is quasi-bounded if $\norm{\delta \seq a}<\oo$. The set of quasi-bounded sequences is denoted by $\mc{S}_{qb}$.
\end{definition}

For $\seq a\in \mc{S}_{qb}$ the sequence $\seq b=(b_n=a_n+\norm{\delta\seq a})_n$ satisfies
\begin{align*}
b_{n+m}=a_{n+m}+\norm{\delta\seq a}=\delta_{n,m}\seq a+a_n+a_m+\norm{\delta \seq a}\leq b_n+b_m
\end{align*}
and therefore there exists $P(\seq a)=\lim_n\frac{a_n}{n}$. Note that for any $C\geq \norm{\delta \seq(a)}$, $P(a)=\inf \frac{a_n+C}{n}$.

\begin{corollary}\label{cor:presioninseq}
If $\seq a\in \mc{S}_{qb}$  then $\sup_n |a_n-nP(\seq a)|\leq \norm{\delta\seq a}$. In particular, if $P(\seq a)=0$ then $\seq a\in\ell^{\oo}$.
\end{corollary}
\begin{proof}
Suppose first that $P(\seq a)=0$. Applying \Cref{cor:subadditivepositive} to $\seq b=(a_n+\norm{\delta\seq a})_n$ we get $b_n\geq 0\Rightarrow a_n\geq -\norm{\delta\seq a}$, and by considering
$\seq c=(a_n-\norm{\delta\seq a})_n$ we get $a_n\leq \norm{\delta\seq a}$ and the claim follows in this case. In general, the sequence $\lift{\seq a}=(a_n-nP(\seq a))_n$ satisfies $\norm{\delta \lift{\seq a}}=\norm{\delta\seq a}$ and $\lim_n\frac{\lift{a}_n}{n}=0$, hence the result.
\end{proof}

Fix $L\in \QM{\ms W}$ and denote  

\begin{align}
&Z_n(L)=\sum_{\bm a \in \Fix[][n]{\ms W}} e^{L(\bm a)}\\
&P_n(L)=\log Z_n(L).
\end{align}

 Recall that $M\in\N$ is so that $R^k>0$, for every $k\geq M$, and we remark that given $\bm a=a_0\cdots a_{n-1}, \bm b=b_0\cdots b_{m-1}\in \ms W$ the number of words $\bm c\in \ms{W}_{k}$ such that $\bm a\bm c\bm b\in \ms W$ is $R^{k+1}_{a_{n-1}b_{0}}$.

\begin{convention} Recall that for $\bm a\in\ms W$ we are denoting $\lift{\bm a}$ any periodic word starting with $\bm a$ of size at most $|\bm a|+M$. To avoid carrying constants, given a quasimorphism on $\ms W$ we re-define $\norm{\delta L}$ so that for any $\bm a, \bm b\in\ms W$
\begin{align*}
|L(\lift{\bm a}\star \lift{\bm b})-L(\lift{\bm a})-L(\lift{\bm b})|\leq \norm{\del L}
\end{align*}
and furthermore $\norm{\delta L_{\mrm{cyc}}}\leq \norm{\del L}$.
\end{convention}

\begin{lemma}\label{lem:Znsubad}
There exists some $D>0$ so that for every $n,m\geq 3M$ it holds
\begin{align*}
D^{-1}\leq \frac{Z_{n+m}(L)}{Z_n(L)\cdot Z_m(L)}\leq D.
\end{align*}
\end{lemma}

\begin{proof}

\noindent{\textbf{Step 1:}} comparison between $Z_n(L)$ and $Z_{n+M}(L)$.

For $i,j\in\ms{A}$ we fix $\bm{u}^{i,j}\in \ms{W}_M$ so that $i\bm{u}^{i,j}j\in\ms{W}$. Given $\bm a=a_1\cdots a_n\in\Fix[][n]{\ms W}$ we get
\[
	Z_n(L)\leq e^{\norm{L}} \sum_{\bm\in\Fix[][n]{\ms W}} e^{L(\bm a\bm{u}^{a_n,a_1})}\leq e^{\norm{L}}Z_{n+M}(L).
\]
Assume that $n>M$: if $\bm a\bm v=a_1\cdots a_{n-M}v_1\cdots v_{2M}\in\Fix[][n+M]{\ms W}$, then $\bm a\bm{u}^{a_{n-M},a_1}\in\Fix[][n]{\ms W}$, and $|L(\bm a\bm v)-L(\bm a\bm{u}^{a_{n-M},a_1})|\leq \norm{L}$. We can thus define a function $\kappa:\Fix[][n+M]{\ms W}\to \Fix[][n]{\ms W}$ by 
\[
\kappa(\bm a\bm v)=\begin{dcases}
\bm a v_1\cdots v_M & \text{if }\bm a v_1\cdots v_M\in\Fix[][n]{\ms W}\\
\bm a \bm{u}^{a_{n-M},a_1} & \text{ otherwise.}
\end{dcases}	
\]
By construction $\kappa$ is surjective. According to \cite{BowenPeriodic}, there is some constants $C_1, \htop(\tau)>0$ ($\htop(\tau)$ is the topological entropy of $\tau$) so that
\[
  C_1^{-1}e^{M\htop(\tau)}\leq \frac{\#\Fix[][n+M]{\ms W}}{\#\Fix[][n]{\ms W}}\leq C_1e^{M\htop(\tau)},
\]
which implies that $\kappa$ is at most $C_2=[C_1e^{M\htop(T)}+1]$-to-one. Therefore  
\begin{align*}
Z_{n+M}(L)\leq e^{\norm{L}}\sum_{\bm a\bm v\in\Fix[][n+M]{\ms W}}e^{L(\kappa(\bm a\bm v))}\leq C_2\cdot e^{\norm{L}} Z_n(L).
\end{align*}

\noindent{\textbf{Step 2:}} comparison between $Z_{n+n}(L)$ and $Z_n(L)\cdot Z_m(L)$.

We fix $n,m>M$ and argue as above to compute (using the previous equation)
\begin{align*}
 Z_n(L)\cdot Z_m(L)\leq e^{\norm{L}}Z_{n+m+2M}(L)\Rightarrow Z_n(L)\cdot Z_m(L)\leq C_2^2e^{3\norm{L}}Z_{n+m}(L). 
 \end{align*} 
For the reverse inequality we use again \cite{BowenPeriodic} and \Cref{lem:dinbasicashift} to construct surjective map $\kappa: \Fix[][n+m+2M]{\ms W}\to \Fix[][n]{\ms W}\times \Fix[][m]{\ms W}$, which is $C_3$-finite to one, where $C_3$ does not depend on $n,m$. We thus get
\begin{align*}
Z_{n+m}(L)\leq e^{2\norm{L}} Z_{n+m+2M}(L)\leq C_3 e^{3\norm{L}}Z_n(L)Z_m(L). 
\end{align*}
Re-arranging the constants we obtain the claimed inequalities.
\end{proof}

\begin{corollary}\label{cor:Presion}
For $L\in \QM{\ms W}$ it holds $(P_n(L))_n\in\mc{S}_{qb}$.
\end{corollary}

\begin{definition}
The pressure of $L\in \QM{\ms W}\cup \QM{\Fix{\ms W}}$ is 
\begin{align}\label{eq:topologicalpressure}
\Ptop(L)=\lim_n \frac{P_n(L)}{n}
\end{align} 
\end{definition}

By \Cref{cor:presioninseq} it holds
\begin{align}
\sup_n |P_n-n\Ptop(L)|<\oo;
\end{align}
we write
\begin{align}\label{eq:qmnormalsft} 
\hat L(\bm a)=L(\bm a)-P_{|\bm a|}(L), \bm a\in \Fix[][n]{\ms W}.
\end{align} 
Naturally,  
\begin{align*}
Z_n(\hat L):=\sum_{\bm a\in \Fix[][n]{\ms W}} e^{\hat L(\bm a)}=1
\end{align*}
and $\Ptop(\hat L)=0$.

\subsection{The invariant measure associated to a quasimorphism}
\label{ssub:invariantmeasureassociated}

Recall that if $\bm a\in \Fix{\ms W}$ then $\mu_{\bm a}\in \PTM{\tau}{\Sigma}$ denotes the measure supported on the orbit of $\seq{p}(\bm a)=\bm a\bm a\cdots$. We write $\delta_{\bm a}=\delta_{\seq{p}(\bm a)}$, so $\mu_{\bm a}=\frac{1}{|\bm a|}\sum_{\pi\in \mc{C}_{\bm{a}}} \delta_{\pi\bm a}$.

Given a quasimorphism $L$ we consider the ($\tau$-invariant) measures
\begin{align}\label{eq:medidaspersft}
\mu_L^N=\frac{1}{Z_N(L)} \sum_{\bm a\in\Fix[][N]{\ms W}}e^{L(\bm a)}\delta_{\bm a}=\sum_{\bm a\in\Fix[][N]{\ms W}}e^{\hat L(\bm a)}\delta_{\bm a}.
\end{align}

 \begin{lemma}
 There exists $E>0$ so that for any $N$ sufficiently large and any $\bm a\in \ms{W}_n$ it holds
 \begin{align*}
 E^{-1}\leq \frac{\mu_L^{n+N}([\bm a])}{\exp\left(L(\lift{\bm a})-P_n(L)\right)}\leq E.
 \end{align*}
  \end{lemma}

\begin{proof}
Let $A=\{\bm c=\bm a\bm u\bm b\bm v\in\Fix[][n+2M+N]:|\bm u|=|\bm v|=M\}$. It follows that there exists some constant $C>0$ which does not depend on $n$, so that $C^{-1}\leq \frac{\# A}{\# \Fix[][N]{\ms W}}\leq C$. If $\bm c\in A$, then
\[
  |L(\bm c)-L(\lift{\bm a})-L(\lift{\bm b})|\leq \norm{\delta L}
\]
and thus 
\begin{align*}
\mu_L^{n+N}([\bm a])Z_{n+N+2M}(L) &=\sum_{\bm c\in A}e^{L(\bm c)}=e^{L(\lift{\bm a})}\sum_{\bm c\in A}e^{L(\bm c)-L(\lift{\bm a})-L(\lift{\bm b})}e^{L(\lift{\bm b})}
\end{align*}
is uniformly comparable to $e^{L(\lift{\bm a})}Z_{N+M}(L)$. By \Cref{lem:Znsubad},
\begin{align*}
\mu_L^{n+N}([\bm a])\approx e^{L(\lift{\bm a})}\frac{Z_{N+M}(L)}{Z_{n+N+2M}(L)}\approx e^{L(\lift{\bm a})-P_n(L)}.
\end{align*}
\end{proof}

This directly implies.

\begin{corollary}
  For every $n\geq 1$ and $N$ sufficiently large it holds: for every $\bm a\in\ms W$ 
  \[
  E^{-1}\leq \frac{\mu_L^{n+M}([\bm a])}{\mu_L^{n+M}([\lift{\bm a}])}\leq E.
  \]
\end{corollary}

\begin{corollary}\label{cor:decaimientomedidasft}
For every $n,m\geq 1$ and $N$ sufficiently large, it holds: if $\bm a\bm b\in \ms{W}_{n+m}, \bm a\in \ms{W}_n, \bm {b}\in \ms{W}_m$ then
\[
  E^{-1}\leq \frac{\mu_L^{N+n+m}([\bm a\bm b])}{\mu_L^{N+n+m}([\bm a])\cdot \mu_L^{N+n+m}([\bm b])}\leq E.
\]
\end{corollary}

\begin{proof}
Indeed,
\begin{align*}
\frac{\mu_L^{N+n+m}([\bm a\bm b])}{\mu_L^{N+n+m}([\lift{\bm a}])\cdot \mu_L^{N+n+m}([\lift{\bm b}])}&\approx 
e^{L(\bm a\star\bm b)-L(\lift{\bm a})-L(\lift{\bm b})}e^{P_{n+m}(L)-P_n(L)-P_m(L)}
\end{align*}
which is bounded, from above and below independently of $n, m, N$.
\end{proof}

Recall that an invariant measure $\mu\in \PTM{\tau}{\Sigma}$ is mixing if for every $A,B\in \BorelM[\Sigma]$,
\[
  \lim_{k\to\oo} \mu(A\cap\tau^{-k}B)=\mu(A)\mu(B).
\]
Clearly such a measure is ergodic.

\begin{proposition}\label{pro:propiedadesmedidamul}
The sequence $(\mu_L^N)_N$ is weakly convergent to some probability $\mu_L\in \PTM{\tau}{\Sigma}$. Moreover, $\mu_L$ is mixing, has full support, and there exists some uniform constant $E>0$ so that for every pair of concatenable words $\bm a\in \ms{W}_n, \bm b\in \ms{W}_m$ and $k\geq n$, it satisfies
\begin{align}
  \label{eq: gibbs1fs} E^{-1}&\leq \frac{\mu_L([\bm a])}{\exp\left(L(\lift{\bm a})-n\Ptop(L)\right)} \leq  E\\
  \label{eq: gibbs2fs} E^{-1} &\leq \frac{\mu_L([\bm a]\cap \tau^{-k}[\bm b])}{\mu_L([\bm a])\cdot \mu_L([\bm b])}\leq E.
  \end{align} 
\end{proposition}

\begin{proof}
 Due to \Cref{cor:presioninseq}, $\sup_n|P_n-n\Ptop(L)|<\oo$, and thus by the corollary above and re-defining $E$ if necessary, we deduce that property \eqref{eq: gibbs1fs} is true for any accumulation point $\mu_L$ of $(\mu_L^N)_N$. On the other hand, by \Cref{cor:decaimientomedidasft}, since $[\bm a]\cap [\tau^{-n}\bm b]=[\bm a\bm b]$, we deduce that 
\[
   \bm a\in \ms{W}_n, \bm b\in \ms{W}_m, \bm a\bm b\in \ms{W}\Rightarrow E \leq \frac{\mu_L([\bm a]\cap \tau^{-n}[\bm b])}{\mu_L([\bm a])\cdot \mu_L([\bm b])}\leq E.
\]
For $l=k-n+1>0$ we can write $[\bm a\cap \tau^{-k}\bm b]=\bigcup_{\bm c\in A_l} [\bm a\bm c\bm b]$, where
$A_l={\bm c\in \ms{W}_{l}:\bm  a\bm c \bm b\in\ms{W}}$: note that the sum $\sum_{\bm c} \mu_L([\bm a\bm c])$ is uniformly comparable to $\sum_{\bm d\in\ms{W}_{l+M}} \mu([\bm a\bm d])=\mu_L([\bm a])$, thus
\begin{align*}
 \mu_L([\bm a\cap \tau^{-n}\bm b])&=\sum_{\bm c\in \ms{W}_l} \mu_L([\bm a\bm c\bm b])\approx \mu_L([\bm a])\cdot\mu_L([\bm b]).
 \end{align*}
Since the algebra of cylinders is generating, a classical measure theory argument implies the following:  for every $A, B\in \BorelM[\Sigma]$,
\begin{enumerate}
  \item $\liminf_{k\to\oo}\mu(A\cap \tau^{-k}B)\geq E^{-1}\mu_L(A)\cdot\mu_L(B)$;
  \item $\limsup_{k\to\oo}\mu(A\cap \tau^{-k}B)\leq E\mu_L(A)\cdot\mu_L(B)$.
\end{enumerate}
The first condition implies that every power of $\tau$ is ergodic: this fact and $2)$ tell us that we are in the hypotheses of a theorem of D. Ornstein \cite{rootOrnstein} that guarantees that $\mu$ is mixing.

By \eqref{eq: gibbs1fs} we have any pair of accumulation points of $(\mu_L^N)_N$ are non-singular with respect to each other and have full support, hence by ergodicity they have to coincide. It follows that there is only one (necessarily mixing) accumulation point of $(\mu_L^N)_N$. 
\end{proof}

\begin{remark}\label{rem:cohomologousqmhavesameeq}
If $L\sim L'$ then it is direct from the construction that $\mu_L=\mu_{L'}$.
\end{remark}

 \subsection{The potential associated to the quasimorphism}\label{ssec:potential}

Here we will show that to each measure $\mu_L$ as constructed above, it corresponds a weak Bowen function.

Define for $\seq{x}\in \Sigma, k\in\N_{>0}$
\begin{align}
\varphi_L^k(\seq x)=\log \frac{\mu_L([\seq{x}]_{k})}{\mu_L([\tau \seq{x}]_{k-1})}=\log \frac{\mu_L([x_0\cdots x_{k-1}])}{\mu_L([x_1\cdots x_{k-1}])}.
\end{align}

\begin{lemma}\label{lem:potentialbounded}
The sequence $(\varphi_L^k(\seq x))_{k\geq 1}$ is uniformly bounded.
\end{lemma}

\begin{proof}
Fix $k\in\N$ and let $\seq{x}\in \Sigma$. The quotient \(\frac{\mu_L([x_0\cdots x_{k-1}])}{\mu_L([x_1\cdots x_{k-1}])}\) can be written as
\begin{align*}
&\frac{\mu_L([x_0\cdots x_{k-1}])}{\exp\paren{L(\lift{\seq{x}}_{(n)})-n\Ptop(L)}}\cdot\frac{\exp\paren{L(\lift{\tau\seq{x}}_{(n-1)})-(n-1)\Ptop(L)}}{\mu_L([x_1\cdots x_{k-1}])}\\
&\cdot \exp\paren*{L(\lift{\seq{x}}_{(n)})-L(\lift{\tau\seq{x}}_{(n-1)})+\Ptop(L)}
\end{align*}
so by \eqref{eq: gibbs1fs},
\[
  |\varphi_L^k(\seq x)|\leq 2E+\norm{\delta L}+\Ptop(L).
\]
\end{proof}

\begin{lemma}\label{lem:Bowenproperties}
For every $n\geq 1, k>n$ and $\seq{x}\in \Sigma$,
\[
  \left|\sum_{l=0}^{n-1}\varphi_L^{k-l}(\tau^l \seq{x})-L(\widetilde{\seq{x}}_{(n)})+n\Ptop(L)\right|\leq 2\log E+\Ptop(L)+\norm{\delta L}.
\]
\end{lemma}

\begin{proof}
Let us compute 
\begin{align*}
\log\left( \frac{\mu_L([\seq{x}]_k)}{\mu_L([\tau^n \seq{x}]_{k-n})}\right)&=\log \mu_L([\seq{x}]_k)- \log \mu_L([\tau^n \seq{x}]_{k-n})\\
&=\sum_{l=0}^{n-1} \log \mu_L([\tau^l\seq{x}]_{k-l})- \sum_{l=1}^{n} \log \mu_L([\tau^l\seq{x}]_{k-l})\\
&=\sum_{l=0}^{n-1} \log \mu_L([\tau^l \seq{x}]_{k-l}) - \sum_{l=0}^{n-1} \log \mu_L([\tau^{l+1}\seq{x}]_{k-l-1})\\
&=\sum_{l=0}^{n-1}\log\left( \frac{\mu_L([\tau^l \seq{x}]_{k-l})}{\mu_L([\tau^{l+1}\seq{x}]_{k-l-1})}\right)=\sum_{l=0}^{n-1}\varphi_L^{k-l}(\tau^l \seq{x})
\end{align*}
Using \Cref{eq: gibbs1fs} we get
\begin{align*}
&|\log \mu_L([\seq{x}]_k)-L(\widetilde{\seq{x}}_{(k)})+k\Ptop(L)|\leq \log E
\shortintertext{and}
&|\log \mu_L([\tau^n\seq{x}]_{k-n})-L(\widetilde{\tau^n\seq{x}}_{(k-n)})+(k-n)\Ptop(L)|\leq \log E
\shortintertext{hence}
&\left|\sum_{l=0}^{n-1}\varphi_L^{k-l}(\tau^l \seq{x})- L(\tilde{\seq{x}}_{(n)})+n\Ptop(L)\right|\leq 2\log E+\Ptop(L)+\norm{\delta L}.
\end{align*}
\end{proof}

We are ready to show the existence of a weak Bowen function associated to $L$.

\begin{theorem}\label{thm:potencialloo}
The sequence $(\varphi_L^n)_{n}$ converges both $\aee{\mu_L}$ and in $\Lp{p}, \forall p\in[1,+\oo)$ to some function $\varphi_L\in \Lp{\oo}(\mu_L)$. 
\end{theorem}

\begin{proof}
The sequence of continuous functions $(e^{\varphi_L^n})_{n}$ is bounded by \Cref{lem:potentialbounded}, and one has
\[
  e^{\varphi_L^n}=\sum_{A\in\xi}\Emu{\mu_L}{\one_{A}|\tau^{-1}\ms{B}^n}\one_{A}
\]
By the (increasing) Martingale convergence theorem and Doob's inequality (see for example \cite{NeveuProb}) it follows that this sequence converges almost everywhere and in $\Lp{p}, \forall p\in[1,+\oo)$ to some function $e^{\varphi_L}\in \Lp[\oo](\mu_L)$. 
\end{proof}

\begin{remark}\label{rem:convergeeninvariante}
We can argue similarly and guarantee that the set of points $\seq{x}$ where 
\[
 \exists \lim_{k\to\oo} \log \frac{\mu_L([x_r\cdots x_{k-1}])}{\mu_L([x_{r+1}\cdots x_{k-1}])} 
\]
is of full $\mu_L$ measure for every $r\geq 0$, hence by induction we get that the set $\Sigma_0$ where $(\varphi_L^n)_{n}$ converges is invariant under $\tau$. Replacing $\Sigma_0$ by $\bigcup_{n\geq 0}\tau^{-n}\Sigma_0$ it is no loss of generality to assume that $\Sigma_0=\tau(\Sigma_0)=\tau^{-1}(\Sigma_0)$.   
\end{remark}

\begin{definition}\label{def:potencialasociado}
The function $\varphi_L:\Sigma_0\to\R$ is the potential associated to $L$.
\end{definition}

Due to \cref{lem:Bowenproperties} we have that:

\begin{corollary}\label{cor:SnvsL}
For $\seq{x}\in \Sigma_0, n\in\N$ it holds 
\[
  |S_n\varphi_L(\seq{x})+n\Ptop(L)-L(\lift{\seq{x}}_{(n)})|\leq 2\log E+\log D+\norm{\delta L}.
\]
Therefore $\varphi_L$ has the weak Bowen property.
\end{corollary}

\subsection{The equilibrium state associated to the quasicocycle}

We are now interested in determining how unique $\mu_L$ is. Here we establish the following.

\begin{theorem}[Variational principle]\label{thm:variationalprinciple}
It holds that $\Ptop(L)=\sup_{\mu\in \PTM{\tau}{\Sigma}}\{h_{\mu}(\tau)+\mu(L)\}=h_{\mu_L}(\tau)+\mu_L(L)$. Moreover, the measure $\mu_L$ is the unique $\tau$-invariant measure where the equality is attained.
\end{theorem}

We recall that for $\mu\in \PTM{\tau}{\Sigma}$ the number $h_{\mu}(\tau)$ is the metric entropy of $\tau$ with respect to $\mu$. It can be computed as
\[
  h_{\mu}(\tau)=\lim_{n\to\oo} -\frac{H_{\mu}(\xi^{n})}{n}=\lim_{n\to\oo} -\frac{1}{n}\sum_{A\in \xi^{(n)}}\mu(A)\log \mu(A);
\]
this is consequence of the classical Sinai-Kolmogorov theorem (see for example \cite{Cornfeld1982}). If $\mu$ is ergodic, then the Brin-Katok formula \cite{brinkatok} permits to compute the entropy also as
\[
  h_{\mu}(\tau)=\lim_{n\to\oo} -\frac{\log \mu([\seq{x}]_n)}{n}\quad \aee{\mu}\ \seq{x}.
\]
The number $P_{\mu}(L)\defeq h_{\mu}(\tau)+\mu(L)$ will be referred to as the metric pressure of the pair $(\mu,L)$.

\begin{proof}[Proof of the Variational Principle]
Since $(\tau,\mu_L)$ is ergodic, using \eqref{eq: gibbs1fs} and Birkhoff's ergodic theorem we get that for $\aee{\mu_L}\ \seq{x}$,
\begin{align*}
h_{\mu_L}(\tau)=\lim \frac{1}{n}-\log \mu_L([\seq{x}]_n)=\lim_n\frac{-S_n\varphi_L(\seq{x})}{n}=-\int \varphi_L \dd\mu_L=\Ptop(L)-\mu_L(L). 
\end{align*}
Therefore $P_{\mu_L}(L)=\Ptop(L)$. We now show that $\mu_L$ is the unique shift invariant measure satisfying the equality, and that for any other invariant measure $\mu$ the metric pressure of $(\mu,L)$ is strictly smaller than $\Ptop(L)$.

To verify the remaining parts we modify an argument due to P. Spitzer. For $\mu \in \PTM{\tau}{\Sigma}$ given we consider the Borel functions
\[
  M_n(\seq{x}):=\begin{dcases}
  \frac{\mu_L([\seq{x}]_n)}{\mu([\seq{x}]_n)} & \mu([\seq{x}]_n)\neq 0\\
  0 & \mu([\seq{x}]_n)= 0.
  \end{dcases}
\]
Then $(M_n)_n$ is a martingale relative to $\{\ms{B}^n\}_n$ in $(\Sigma,\mu)$, and converges to either the zero function if $\mu\perp\mu_L$, or to $\frac{\dd \mu_L}{\dd \mu}$ if $\mu_L\ll\mu$. In the case when $\mu$ is singular with respect to $\mu_L$, we let $N_n=-\log M_n$: then $(N_n)_n$ is a super-martingale, with
\[
  \int N_n^- \dd \mu=\int (\log M_n\wedge 0)\dd \mu\leq \int (M_n-1\wedge 0)\dd \mu \leq \int M_n \dd \mu=1, \forall n.
\]
Due to the convergence of $(M_n)_n$, it follows that $\lim \int N_n\dd\mu=+\oo$. Note that 
\begin{align*}
\int N_n\dd \mu =\sum_{A\in \xi^{n}}\mu(A)\log \frac{\mu(A)}{\mu_L(A)}
\end{align*}
and therefore by the Gibbs property of $\mu_L$ it holds 
\begin{align*}
&H_{\mu}(\xi^{n})=-\sum_{A\in \xi^{n}} \mu(A)\log \mu_L(A)-\int N_n\dd \mu\leq E+n\Ptop(L)-\int L^{(n)}\dd\mu-\int N_n\dd \mu\\
&\Rightarrow h_{\mu}(\tau)+\mu(L)=\inf_n \frac{H_{\mu}(\xi^{n})}{n}+\mu(L)\leq  \frac{H_{\mu}(\xi^{n})}{n}+\frac{\int L^{(n)}\dd\mu}{n}+\left(\mu(L)-\frac{\int L^{(n)}\dd\mu}{n}\right)\\
&\leq \frac{E+\norm{\delta L}-\int N_n\dd \mu}{n}+\Ptop(L). 
\end{align*}
For $n$ large the quantity $\frac{E+\norm{\delta L}-\int N_n\dd \mu}{n}$ is strictly negative, and therefore if $\mu$ is singular with respect to $\mu_L$ then $h_{\mu}(\tau)+\mu(L)<\Ptop(L)=h_{\mu_L}(\tau)+\mu_L(L)$. 

Finally, one notes that if $\mu=\int \mu_{\alpha}\dd\Psi(\alpha)$ for some $\Psi\in \ProbM[\PTM{\tau}{\Sigma}]$, then
$P_\mu(L)=\int P_{\mu_\al}\dd\Psi(\alpha)$: this follows due to convexity of the metric entropy with respect to such decomposition, plus the fact that $\mu\in \PTM{\tau}{\Sigma}\to \mu(L)$ is convex and continuous (cf. \Cref{pro:integracioncontinua}). As a consequence, applying this to an ergodic decomposition of a given invariant measure, we get to conclude that $\sup_{\mu\in \PTM{\tau}{\Sigma}}\{h_{\mu}(\tau)+\mu(L)\}=\sup_{\mu\in \ETM{\tau}{\Sigma}}\{h_{\mu}(\tau)+\mu(L)\}$, where $\ETM{\tau}{\Sigma} \subset \PTM{\tau}{\Sigma}$ denotes the set of ergodic invariant measures of $\tau$.

Recalling the basic fact that two ergodic invariant measures are either mutually singular or they coincide, using that $\mu_L$ is ergodic and applying the previous reasoning we conclude that it is the unique equilibrium state.
\end{proof}

\begin{remark}
  For the first part of the previous theorem one could use the work of Cao, Feng and Huang \cite{Cao_2008}, where they establish the variational principle for continuous sub-additive functions on compact metric spaces. The proof given above gives simultaneously the variational principle and the  uniqueness of the equilibrium measure in our context.
\end{remark}

We thus see that in fact, for $L\in\QM{\ms W}$ (equivalently, for $\phi \in \Bow{\Sigma}$ or $\bm B\in \QCB{\Sigma}$) the corresponding measure $\mu_L$ is canonically defined to its cohomology class. In fact, there is a one to one correspondence as we show below.

\begin{corollary}\label{cor:muLdeterminesclass}
Let $L_1 ,L_2\in\QM{\ms W}$ be such that they have the same equilibrium state $\mu$, with $\mu(L_1)=\mu(L_2)$. Then, they are cohomologous.
\end{corollary}

\begin{proof}
  Indeed, the hypothesis shows that the potentials associated to $L_1, L_2$ coincide, thus by \Cref{cor:SnvsL} we deduce that $(L^{(n)}_1-L^{(n)}_2)_n$  is uniformly bounded on a  dense set of $\Sigma$. This is enough to show that they are uniformly bounded everywhere, and thus $L_1\sim L_2$.
\end{proof}

\begin{remark}
  For Hölder functions $\phi_1, \phi_2$, the corresponding statement is due to Bowen (Proposition $4.5$ in \cite{EquSta}). Bowen's result includes the fact that that transfer function is Hölder. This can be deduced from the Corollary above, applying \cite{Quas_1997} to obtain a bounded continuous transfer function, and then using \cite{Livshits1971} to improve the regularity.
\end{remark}

At this stage we have essentially proved the statements of \Cref{thm:representation1}. Indeed, defining $\psi_L=\varphi_L+\Ptop(L)$ it follows by the previous inequality that $S_n\psi_L(\seq{x})$ is uniformly close to $L(\lift{\seq{x}}_{(n)})$. Moreover, given any other $\psi_L'$ with the weak Bowen property and satisfying the same, the quasicocycle $\bm{B}^{\psi_L'}$ is cohomologous to $\bm{B}^L$, and therefore it determines $L'\in\QM{\ms W}$ cohomologous to $L$, which then has the same invariant measure $\mu_L$. Applying the previous Corollary we deduce that $\psi_L'=\psi_L+u-u\circ \tau$ for some $u\in \Lp{\oo}{\mu_L}$, due to \Cref{thm:E}.

\subsection{The transfer operator associated to \texorpdfstring{$L$}{L}}\label{ssec:transfer}

Here we give another characterization of $\mu_L$. For this we introduce the transfer operator associated to $\varphi_L$, which will also be used later. Fix $p\in [1,+\oo]$ and define $\Rop[\varphi_L]=\Rop[L]:\Lp{p}{\mu_L}\toit$ by the formula
\[
   \Rop[L]{\psi}(\seq x)=\sum_{\tau\seq y=\seq{x}} e^{\varphi_L(\seq{y})}\psi(\seq{y}).
\]

The following is verified by direct computation.
\begin{lemma}[Transference property]
Let $\psi\in \Lp{p}{\mu_L}$ and $\phi\in \Lp{q}{\mu_L}$ where $p^{-1}+q^{-1}=1$ if $p<\oo$, and $q=\oo$ is $p=\oo$. Then
\begin{align}\label{eq:transferenceLcal}
\Rop[L]{(\psi\cdot \phi\circ \tau)}=\Rop[L]{\psi}\cdot \phi. 
\end{align}
\end{lemma}

We are interested in the action of the transpose of $\Rop$ when $p=\oo$. Consider $\mrm{Add}(\mu_L)=\mrm{Add}(\Sigma,\BorelM[\Sigma],\mu_L)$ the set of finitely additive measures on $(\Sigma,\BorelM[\Sigma])$ that vanish on the null-sets of $\mu_L$. Due to the Yosida-Hewitt representation theorem we can identify $\mrm{Add}(\mu_L)=\Lp{\oo}{\mu_L}^{*}$; since $\mu_L(\Sigma)=1$ (finite) it follows that $\mu_L\in \mrm{Add}(\Sigma)$. We can then consider $\Rop^{*}:\mrm{Add}(\mu_L)\toit$ the corresponding adjoint operator, with
\[
  \nu\in \mrm{Add}(\mu_L),\ \varphi\in \Lp{\oo}{\mu_L}\Rightarrow \Rop[L]^*\nu(\varphi)=\nu(\Rop[L]{\varphi}).
\]

\begin{lemma}\label{lem:RPFnormalizado}
 It holds $\Rop[L]{1}=1$ for $\aee{\mu_L}$
 \end{lemma}

\begin{proof}
Using that $\Rop[L]{1}=\sum_{a\in\ms A}\Rop[L]{1_{[a]}}$ and linearity we obtain
\begin{align*}
\Rop[L]{1_{[a]}}(\seq{x})&=R_{ax_0}e^{\varphi_L(ax_0x_1\cdots)}=\lim_n\frac{\mu_L([ax_0\cdots x_{n-1}])}{\mu_L([x_0\cdots x_{n-1}])}\\
&\Rightarrow \Rop[L]{1}(\seq{x})=\lim_n \sum_{a\in\ms A} R_{ax_0}\frac{\mu_L([ax_0\cdots x_{n-1}])}{\mu_L([x_0\cdots x_{n-1}])}=1=1(\seq{x}),
\end{align*}
for $\aee{\mu_L}$ point.
\end{proof}

\begin{corollary}
It holds $\Rop[L]^*\mu_L=\mu_L$.
\end{corollary}

\begin{proof}
By the transference property it follows that for $\phi, \psi \in \Lp{\oo}{\mu_L}$
\begin{align*}
\int \phi \dd\mu_L=\int \phi\circ \tau \dd\mu_L=\int \phi\circ\tau\cdot \Rop[L]{1}  \dd\mu_L=\int \Rop[L]{\phi}\dd\mu_L
\end{align*}
and the claim follows.
\end{proof}

 \begin{remark}
The same argument shows that if $\nu\in \mrm{Add}(\mu_L)$ is invariant under $\Rop[L]^*$, then it is $\tau$-invariant.
 \end{remark}

\begin{corollary}\label{cor:uniquestationary}
If $\nu\in \mrm{Add}(\mu_L)$ is $\sigma-$additive and $\Rop[L]^*$ invariant, then $\nu=\mu_L$.
\end{corollary}

\begin{proof}
Previous remark and ergodicity of $\mu_L$.
\end{proof}

Now take $\psi\in\Bow[][\mu_L]{\Sigma}, \psi\sim \varphi_L$: by \Cref{thm:livsicWeakBowen} there exists $u\in \Lp{\oo}{\mu_L}$ so that $\varphi_L=\psi+u-u\circ T$. It follows that 
\begin{align*}
1=\Rop[\varphi_L](1)=\frac{\Rop[\psi]{e^u}}{e^u}
\end{align*}
and $e^u$ is an eigenvector of $\Rop[\psi]$ corresponding to the eigenvalue $1$. Proposition $4.3$ of \cite{Haydn1992} implies that in fact $e^u$ is a simple eigenfunction of $\Rop[\psi]:\Lp{1}{\mu_L}\toit$.

\subsection{The potential for other reference measures}\label{sub:potentialothermeasures}

In order to address more precise statistical properties of a given quasimorphism (as the Central Limit Theorem), the construction of its associated potential has a major drawback, namely, it is defined only for its equilibrium measure $\mu_L$. In this part we remedy this: given some invariant measure $\mu\in \PTM{\tau}{\Sigma}$ we construct an associated potential $\varphi_{L,\mu} \in \Lp{\oo}{\mu}$
 with the weak Bowen property, and so that $(S_n\varphi_{L,\mu})_{n\geq 1}$ is cohomologous to $(L^{(n)})_{n\geq 1}$. Although this method also works for $\mu_L$, unlike the previous construction, the potentials obtained do not seem to be unique, thus preventing a unique characterization of $\varphi_L$.

 Fix $L\in\QM{\ms W}$ and consider the sequence of locally constant functions $\zeta_n:\Sigma\to\R$,
\begin{align*}
\zeta_n(\seq{x})&=\frac{1}{n}\sum_{k=1}^n L(x_0\ldots x_{k})-L(x_1\ldots x_{k})\\
&=\frac{1}{n}\left(L(x_0x_1)-L(x_1)+L(x_0x_1x_2)-L(x_1x_2)+\cdots L(x_0\ldots x_{n})-L(x_1\ldots x_{n})\right).
\end{align*}
 The following is clear.
 \begin{lemma}
 $(\zeta_n)_{n\geq 1} \subset \Cr{\Sigma}$ is uniformly bounded: $\norml{\zeta_n}\leq \norm{L}$.
 \end{lemma}

Let us now fix $T\in \N$ and consider $n>T$.

\begin{lemma}\label{lem:estimatezn}
 It holds
 \[
  |S_T \zeta_n(\seq{x})-L(x_0\cdots x_{T-1})|\leq \norm{\delta L}(\frac{2T}{n}+1-\frac{T}{n})+\frac{T}{n}\norml{\restr{L}{\ms A}}.
 \]
 \end{lemma}

\begin{proof}
 Since $\zeta_n\circ \tau^l=\frac{1}{n}\sum_{k=1}^n L(x_l\ldots x_{l+k})-L(x_{l+1}\ldots x_{l+k})$, one gets
 \begin{align*}
 S_T \zeta_n(\seq{x})&=\frac{1}{n}\left(\sum_{l=0}^{T-1}\sum_{k=1}^n L(x_l\ldots x_{l+k})-L(x_{l+1}\ldots x_{l+k})\right)\\
 &=\frac{1}{n}\left(\sum_{k=0}^{n-1}L(x_0\cdots x_k)+\sum_{k=0}^{T-1}L(x_k\cdots x_{n+k})-\sum_{k=0}^{T-1}L(x_k)-\sum_{k=0}^{n-1}L(x_T\cdots x_{T+k})\right).
 \end{align*}
 Therefore,
 \begin{align*}
 &|S_T \zeta_n(\seq{x})-L(x_0\cdots x_{T-1})|\leq\\
 &\frac{T}{n}\norml{\restr{L}{\ms A}}+\frac{1}{n}\left|\sum_{k=0}^{T-2}L(x_0\cdots x_k)+L(x_{k+1}\cdots x_{n+k})-
 L(x_0\cdots x_{T-1})-L(x_T\cdots x_{T+n-2})\right|\\
 &+\frac{1}{n}\left|\sum_{k=T}^n L(x_0\cdots x_{k})-L(x_0\cdots x_{T-1})-L(x_T\cdots x_k)\right|\\
 &\leq \norm{\delta L}(\frac{2T}{n}+1-\frac{T}{n})+\frac{T}{n}\norml{\restr{L}{\ms A}}.
 \end{align*}
\end{proof}

To construct the potential we now argue as in Burkholder's proof of Kingman subadditive ergodic theorem, \cite{Burkholder1973}. The following is delicate result due to Komlós.

 \begin{theorem}[Komlos, \cite{Komlos1967}]
 Let $\mu$ be a probability measure on a measurable space $(X,\BorelM[X])$ and suppose that $(f_n)_n\in \Lp{1}{\mu}$ is uniformly bounded in norm. Then there exists $f\in \Lp{1}{\mu}$ and a sub-sequence $(r(n))_{n\in\N} \subset \N$ so that for $\aee{\mu}\ x\in X$ it holds
\[
   \lim_{N\to\oo}\frac{1}{N}\sum_{n=1}^N f_{r(n)}(x)=f(x).
 \]
 \end{theorem}

 \begin{proposition}\label{pro:potentialarbitrary}
 For any $\mu\in \PTM{\tau}{\Sigma}$ there exists $\varphi_{L,\mu}:D_{\mu}\to\R\in \Lp{\oo}{\mu}$ so that
 \[
   \seq{x}\in D_{\mu}, n\in \N\Rightarrow |S_n\varphi_{L,\mu}(\seq{x})-L(x_0\ldots x_{n-1})|\leq \norm{\delta L}.
 \]

In particular, if $\mu$ has full support then $\varphi_{L,\mu}$ has the weak Bowen property. 
\end{proposition}

 \begin{proof}
Applying Komlós' theorem to $(\zeta_n)_{n\geq 1}$, let $\varphi_{L,\mu}:D_{\mu}\to\R$ and $(r(n))_{n\in\N}$ be as in its conclusion. Fixing $T\in\N, \eps>0$, for sufficiently large $n$ one has that $\sup_{\seq{x}\in\Sigma}|S_T\zeta_n(\seq{x})-L(x_0\ldots x_{T-1})|\leq \norm{\delta L}+\eps$, therefore (since $\varphi_{L,\mu}$ is limit of a convex combination of $(\zeta_n)$),
\[
\sup_{\seq{x}\in\Sigma}|S_T\varphi_{L,\mu}(\seq{x})-L(x_0\ldots x_{T-1})|\leq \norm{\delta L}+\eps.  
\]
This implies the first part, while the second is automatic. 
 \end{proof}

\begin{remark}\label{rem:convergenceLp}
Since $(\zeta_n)_n$ are uniformly bounded in $\Lp{\oo}{\mu}$, we have 
\[
  \lim_{N\to\oo} \frac{1}{N} \sum_{n=1}^N L(x_0\ldots x_{r(n)-1})-L(x_1\ldots x_{r(n)-1})=\varphi_{L,\mu}(\seq{x})
\]
both $\aee$ and in $\Lp{p}{\mu}$ for $1\leq p<\oo$.
\end{remark}

\begin{remark}\label{rem:convergencesequence}
  The following generalization is useful to compare weak Bowen functions associated to different measures. Suppose that $(\mu_k)_{k\in\N} \subset \PTM{\tau}{\Sigma}$ is given: then there exists a $\varphi:D\to\R$ so that
  \begin{enumerate}
    \item $\forall k, \mu_k(D)=1$.
    \item $\seq{x}\in D, n\in \N\Rightarrow |S_n\varphi(\seq{x})-L(x_0\ldots x_{n-1})|\leq \norm{\delta L}.$
    \end{enumerate}
To see this it suffices to apply the previous proposition to $\mu=\sum_{k=1}^{\oo}\frac{\mu_k}{2^k}$. It would be interesting to establish the existence of a cohomologous weak function defined for every invariant measure, but this remains as an open problem. 
\end{remark}

\section{Finer properties of the equilibrium measure associated to a quasimorphism}\label{sec:finer}

In this section we investigate the dynamical properties of the equilibrium state, and elucidate the structure of the set where the potential $\varphi_L$ is well defined. We use this knowledge to deduce the Bernoulli property of the system $(\Sigma,\tau,\mu_L)$. 

This part is more dynamically flavored than others in the article, and contains results that are well known for the case of H\"older functions. For weak Bowen functions we need to give different proofs. These methods also reveal non-obvious structure of weak Bowen functions, and hence of quasimorphisms.  It is also likely that the machinery that we present here will be applicable in other instances of thermodynamic formalism for non-continuous potentials, which is still a largely undeveloped area.

\subsection{The transpose quasimorphism and the natural extension} 

It will be important to work in the natural extension of the shift dynamical system, which can be canonically identified with the two-sided SFT $\lift{\Sigma}=\lift{\Sigma}(R)$. The canonical projection is $\pi:\lift{\Sigma}\to\Sigma, \pi(\seq{x})=\seq{x}^{+}$; as was explained in \Cref{sec:LivsicSFT}, it induces an isomorphism $\PTM{\tau}{\lift{\Sigma}}\ni\lift \mu\to \mu\in \PTM{\tau}{\Sigma}$, preserving ergodicity.

\begin{notation}
For points $\seq{x} \in \lift{\Sigma}$ we use a dot to indicate the zero entry, i.e.\@ $$\seq{x}=(\cdots x_{-n}\cdots x_{-1}\bpto x_0 \cdots x_{n}\cdots ).$$ If $\seq{x}\in \lift{\Sigma}, k\leq 0\leq  l$ we write
\[
  [x_{k}x_{k+1}\cdots x_{-1}\bpto x_0x_1\cdots x_l]=\{\seq{y}\in\lift{\Sigma}:y_i=x_i,\ k\leq i\leq l\}.
\] 
Given $\seq{x}\in \lift{\Sigma}$ we denote $\seq{x}^{+}=(x_n)_{n\geq 0}, \seq{x}^{-}=(x_n)_{n\leq 0}$, and write $\seq{x}=\inner{\seq{x}^{-}}{\seq{x}^{+}}$.
\end{notation}

The space non-positive entries of $\lift{\Sigma}$ is also a one-sided shift space, which can be identified with the subshift of finite type $\Sigma(R^{\dagger})$ corresponding to $R^{\dagger}$, the transpose of $R$. Indeed, we can write
$\Sigma(R^{\dagger})=\{\seq{x}=(x_n)_{n\leq 0}: R_{x_{n-1}x_n}=1,\forall n\geq 0\}$, and observe that in this notation the shift endomorphism is given by $(x_n)_{n\geq 0}\mapsto (x_{n-1})_{n\geq 0}$: we denote this map by $\tau^{-1}$. The projection $\pi^{\dagger}:\lift{\Sigma}(R)\to \Sigma(R^{\dagger})$,
\[
  \pi^{\dagger}(\seq{x})=(x_{-n})_{n\geq 0}
\]
semi-conjugates $\tau^{-1}:\lift{\Sigma}\toit$ and\footnote{Here $\tau^{-1}:\lift{\Sigma}\toit$ is the inverse of the shift map in $\lift{\Sigma}$.} $\tau^{-1}:\Sigma(R^{\dagger})\toit$. It follows in particular that $\PTM{\tau}{\lift{\Sigma}}=\PTM{\tau^{-1}}{\lift{\Sigma}}$ is in bijection with $\PTM{\tau^{-1}}{\Sigmad}$, and therefore there exists a bijection $\mu \in \PTM{\tau}{\Sigma}\to \mu_{\dagger}\in \PTM{\tau^{-1}}{\Sigmad}$.

We now fix $L\in\QM{\ms W}$ and lift the measure $\mu_L$ to obtain $\lift{\mu}_L\in \PTM{\tau}{\lift{\Sigma}}$. Let $\varL:\Sigma_0\to\R$ be the associated potential, where $\Sigma_0$ is completely invariant and $\mu_L(\Sigma_0)=1$. The lifted potential $\lift{\varL}(\seq{x})=\varL(\seq{x}^{+})$ is defined in the full measure invariant set $\lift{\Sigma}_0=\pi^{-1}(\Sigma_0)$. Observe that $\lift{\varL}$ satisfies the weak Bowen property and only depends on the positive coordinates of $\seq{x}$.

\begin{lemma}\label{lem:eeshiftpositivo}
$\lift{\mu}_L$ is the (unique) equilibrium state for $\lift{\varL}$.
\end{lemma}

\begin{proof}
We use the following facts: 
\begin{enumerate}
    \item for every $\mu\in \PTM{\tau}{\Sigma}$, $h_{\mu}(\tau)=h_{\lift{\mu}}(\tau)$; 
    \item for every $\bm a\in \Fix{\ms W}$, 
    \[
    \lift{\mu}_{\bm{a}}(\lift{\varL})=\mu_{\bm a}(\varL).
    \]
    This in turn implies, by density of measures supported on periodic orbits \cite{Sigmund1974} and \Cref{pro:integracioncontinua}, that
    \[
    (\forall \lift{\mu}\in \PTM{\tau}{\lift{\Sigma}}):\ \lift{\mu}(\lift{\varL})=\mu(\varL)=\mu(L).
    \]
\end{enumerate}
The result then follows from the variational principle (note that proof of \Cref{thm:variationalprinciple} applies without modifications to the system $\left(\lift{\Sigma},\tau,\lift{\varL}\right)$).
\end{proof}

Next we look at $\mu_{L^{\dagger}}$. For a word $\bm a=a_1\cdots a_{n}$ write $\bm{a}^{\dagger}=a_{n}\cdots a_1$.  Then $\dagger: \ms{W}=\ms{W}(R)\to \ms{W}^{\dagger}=\ms{W}(R^{\dagger})$ induces a linear isomorphism $\dagger:\QM{\ms W}\to\QM{\ms{W}^{\dagger}}$ (respectively, $\dagger: \QM{\Fix{\ms W}}\to \QM{\Fix{\ms{W}^{\dagger}}}$) by the formula
\[
  L^{\dagger}(\bm a)=L(\bm{a}^{\dagger});
\]
accordingly, $\norm{\delta L^{\dagger}}=\norm{\delta L}, \norm{L^\dagger}=\norm{L}$. 

We define $\varLd:\Sigma^{\dagger}_0\to\R$ as before,
\begin{align}
e^{\varLd(\seq{x}^-)}=\lim_{n\to\oo} \frac{\lift{\mu}_{L^{\dagger}}([x_{-n}\cdots\bpto x_{0}])}{\lift{\mu}_{L^{\dagger}}([x_{-n}\cdots\bpto x_{-1}])}
\end{align}
and consider its lift to $\lift{\Sigma}$, $\lift{\varLd}(\seq{x})=\varLd(\seq{x}^-)$.

\begin{lemma}\label{lem:eeshiftnegativo}
$\lift{\mu}_L$ is the equilibrium state for $\lift{\varLd}$. Therefore $\pi^{\dagger}\lift{\mu}_L=\mu_{L^{\dagger}}$.
\end{lemma}

\begin{proof}
We argue as in the previous Lemma, using that
\begin{enumerate}
  \item for every $\mu\in \PTM{\tau^{-1}}{\Sigmad}$, $h_{\mu}(\tau^{-1})=h_{\lift{\mu}}(\tau^{-1})=h_{\lift{\mu}}(\tau)$. 
 \item For every $\bm a\in \Fix{\ms W}$ it holds that $\bm{a}^{\dagger}\in \Fix{\ms{W}^{\dagger}}$ and 
 \[
 \lift{\mu}_{\bm{a}^{\dagger}}(\lift{\varLd})=\mu_{\bm{a}^{\dagger}}(\varLd)=\mu_{\bm{a}^{\dagger}}(L^{\dagger})=\mu_{\bm a}(L).
 \]
Hence,
 \[
 \forall \nu\in \PTM{\tau}{\lift{\Sigma}},\ \nu(\lift{\varLd})=\nu(\lift\varL).
 \]
\end{enumerate}
The result follows again by the variational principle and the uniqueness of the equilibrium measure for $\lift{\varLd}$.
\end{proof}

During the proof of \Cref{lem:eeshiftpositivo,lem:eeshiftnegativo} we established that $\forall \nu \in \PTM{\tau}{\lift{\Sigma}}$ it holds $\nu(\lift{\varL})=\nu(\lift{\varLd})$. For this reason:

\begin{corollary}\label{cor:liftpotencialessoncoh}
The functions $\lift{\varL}, \lift{\varLd}$ are cohomologous in $\Lp{\oo}{\lift{\mu}_L}$: there exists $u\in \Lp{\oo}{\lift{\mu}_L}$ so that $\lift{\varL}=\lift{\varLd} +u-u\circ \tau$.
\end{corollary}

\begin{convention}
From now on $\lift{\varL},\lift{\varLd}$ are defined in the same invariant set of full $\lift{\mu}_L$ measure, $\lift{\Sigma}_0=\pi^{-1}(\Sigma_0)\cap (\pi^{\dagger})^{-1}(\lift{\Sigma}_0^{\dagger})$. The transfer function $u$ is defined in the full $\lift{\mu}_L$ measure and invariant set $\lift{\Sigma}_1 \subset \lift{\Sigma}_0$.
\end{convention}

\begin{remark}\label{rem:gmeasuresinverse}
  The previous Corollary evidences further the necessity of working with the coarser version if Livšic cohomology. Indeed, in \cite{BERGHOUT_2018} the authors give an example of a continuous Bowen potential $\varphi:\Sigma\to\R$ for which the reverse measure $\mu_{\dagger}$ is not associated to any continuous potential.
\end{remark}

\subsection{The conditionals of \texorpdfstring{$\lift{\mu}_L$}{muL}}

Our next goal is to prove that $\lift{\mu}_L$ is equivalent to the product $\mu_L\times\mu_{L^{\dagger}}$, and give a dynamical interpretation of the transfer function $u$. First, we recall the basics of the theory of measurable partitions and disintegration of measures in the sense of Rohklin. Given a Borel probability space $(X,\BorelM[X],\mu)$ and a family of partitions (by measurable sets) $(\mrm{P}_i)_{i\in I}$, its join is $\bigvee_{i\in I}\mrm{P}_i=\sigma\paren{\bigcup_{i\in I} \mrm{P}_i}<\BorelM[X]$. A $\sigma$-algebra $\xi$ is called measurable if there exists a countable family of finite partitions $(\mrm{P}_n)_n$ so that $\xi=\bigvee_{n}\mrm{P}_n$, while a partition $\mrm{P}$ is called measurable if the $\sigma$-algebra that it generates is measurable. 

\begin{theorem}[Rohklin]\label{thm:Rohklin}
If $\xi$ is a measurable $\sigma-$algebra for $\mu$, there exists $X_0 \subset X$ of full measure, and a family $\{\mu_{x}^\xi\}_{x\in X} \subset \ProbM[X]$ (called the conditionals of $\mu$ with respect to $\xi$) so that for every $\psi\in \Lp{\oo}{\mu}$ the map $x\in X_0\to \mu_{x}^\xi(\psi)$ is a measurable version of $\Emu{\mu}{\psi|\xi}$. In particular, $\mu(\psi)=\int\left(\int \psi\dd\mu^{\xi}_x \right)\dd\mu(x)$.
\end{theorem}

Since $\xi$ is measurable, for $\mu$-almost every $x\in X$ the set $\xi(x)=\bigcap_{x\in A\in\xi} A$ (called the atom of $\xi$ containing $x$) is in $\BorelM[X]$, and $X_0$ can be chosen so that $x\in X_0$ implies
\begin{itemize}
  \item $\mu_{x}^\xi(\xi(x))=1$,
  \item $y\in X_0\cap\xi(x)\Rightarrow \mu_{x}^\xi=\mu_{y}^\xi$.
\end{itemize}
Define $X_{\xi}$ as the space of atoms of $\xi$ and equip it with the final measure structure induced by the projection $p_{\xi}(x)=\xi(x)$. Due to the above considerations, the family of conditionals can also be thought of as a measurable map $(X/\xi, p_{\xi}\mu)\mapsto \ProbM[X]$ sending $A=\xi(x)\to \mu^{\xi}_A=\mu^{\xi}_x$. The measure $\mu$ is thus completely determined by $\left\{\mu^{\xi}_A\right\}_{A\in X/\xi}$ and the quotient measure $\mu_{X/\xi}=p_{\xi}\mu$. We refer to \cite{Rokhlin} for further discussion.

\paragraph{\textbf{Stable and unstable sets}} For $\seq{x}\in \lift{\Sigma}$ denote
 \begin{align}
 &\Wsloc{\seq{x}} \defeq\{\seq{y}:\seq{y}^+=\seq{x}^+\}\\
 &\Wuloc{\seq{x}} \defeq \{\seq{y}:\seq{y}^-=\seq{x}^-\}
 \end{align}

 \begin{definition}
 $\Wsloc{\seq{x}}, \Wuloc{\seq{x}}$ are the local stable and unstable sets of $\seq{x}$, respectively.
 \end{definition}

These sets can be parametrized as follows. Fix $\seq{x}\in\lift{\Sigma}$ and let 
 \begin{align*}
 &p^{u}_{\seq{x}}:\Sigma_{x_0}=\Sigma\cap [x_0]\to \Wuloc{\seq{x}}&p^{u}_{\seq{x}}(\seq z)=\inner{\seq{x}^-}{\seq z}\\
 &p^{s}_{\seq{x}}:\Sigma^{\dagger}_{x_0}=\Sigma^{\dagger}\cap [x_0]\to \Wsloc{\seq{x}}&p^{s}_{\seq{x}}(\seq z)=\inner{\seq z}{\seq{x}^{+}}.
 \end{align*}
Both $p^{u}_{\seq{x}}, p^{s}_{\seq{x}}$ are homeomorphisms.

Write $\xi^{u}=\{\Wuloc{\seq{x}}:\seq{x}\in \lift{\Sigma}\}$ for the partition into local unstable sets: it holds that
 \begin{itemize}
   \item  $\xi^{u}$ is finer than $\tau\xi^u=\{\tau\left(\xi(\tau^{-1}(\seq{x}))\right):\seq{x}\in\lift{\Sigma}\}$;
   \item  $\bigvee_{n\leq 0} \tau^n\xi^u=\BorelM[\lift{\Sigma}]$;
   \item  $\bigcap_{n\geq 0} \tau^n\xi^u=\{\Wu{\seq{x}}:\seq{x}\in\lift{\Sigma}\}=:\ms{B}^u$, where $\Wu{\seq{x}}=\{\seq{y}:\lim_{n\to \oo}\dis{\lift{\Sigma}}{\tau^{-n}\seq{x}}{\tau^{-n}\seq{y}}=0\}$.  
 \end{itemize}
Similarly for the partition $\xi^{s}=\{\Wsloc{\seq{x}}:\seq{x}\in\lift{\Sigma}\}$, replacing $\tau$ by its inverse. Next we compute the disintegration of $\lift{\mu}_L$ with respect to the partition $\xi^u$.

\begin{notation}
 Denote $\mu^u_{\seq{x}}=(\lift{\mu}_L)^{\xi^u}_{\seq{x}}$, whenever is defined. For $\seq{x}\in\Sigma, \seq{y} \in \Wuloc{\seq{x}}$ and $k\geq 0$ we write
  \begin{align*}
  [y_0\cdots y_k]^u_{\seq{x}}=\{\seq z\in \Wuloc{\seq{x}}: z_n=y_n, \forall n\leq k\}=p^u_{\seq{x}}([x_0 y_1\cdots y_k]).
  \end{align*}
 Likewise if $\seq{y}\in \Wsloc{\seq{x}}$, we write
 \begin{align*}
  [y_{-k}\cdots y_{0}]^s_{\seq{x}}=\{\seq z\in \Wsloc{\seq{x}}: z_n=y_n, \forall n\geq -k\}=p^s_{\seq{x}}([y_{-k}\cdots y_{0}]).
 \end{align*} 
 \end{notation}

If we now assume that $\seq{x}\in\lift{\Sigma}_0$ it follows that
 \begin{align}\label{eq:medidau}
 \nonumber\mu^u_{\seq{x}}([y_0\cdots y_k]^u_{\seq{x}})&=\lim_{n\to \oo} \frac{\lift{\mu}_L([x_{-n}\cdots x_{-1}\bpto x_0 y_1\cdots y_{k}])}{\lift{\mu}_L([x_{-n}\cdots x_{-1} \bpto x_0])}=\lim_{n\to \oo} \frac{\lift{\mu}_L([x_{-n}\cdots x_{-1}x_0 y_1\cdots \bpto y_{k}])}{\lift{\mu}_L([x_{-n}\cdots x_{-1} \bpto x_0])}\\
 &=\exp(S_{k}\lift{\varLd}(\tau\seq{y}))=\exp\left(S_{-k}\varLd(\cdots x_{-n}\cdots x_{-1}x_0 y_1\cdots  y_k)\right)
 \end{align}
with the convention that $S_0\varphi\equiv 0, S_{-k}\psi=\sum_{i=0}^{k-1}\psi\circ\tau^{-i}$. We remark that we have used the invariance of $\lift{\Sigma}_0$ the guarantee the existence of the limit (cf. \Cref{rem:convergeeninvariante}): more generally, this permits to define $\mu^u_{\seq{x}}$ whenever $\seq{x}$ belongs to the set $(\pi^{\dagger})^{-1}(\Sigma_0^{\dagger})$, which is saturated by the partition $\xi^u$. These unstable measures only depend on the past of $\seq{x}$, and we can safely write $\mu^{u}_{\seq{x}}=\mu^{u}_{\seq{x}^-}$.

\begin{corollary}\label{cor:dominiodisintegracionu}
The unstable disintegration can be defined on $\lift{\Sigma}_0$.
\end{corollary}

Similarly for the stable disintegration we get,
 \begin{align}
\nonumber \mu^s_{\seq{x}}([y_{-k}\cdots y_{0}]^s_{\seq{x}})&=\lim_{n\to \oo} \frac{\lift{\mu}_L([y_{-k}\cdots y_{-1}\bpto x_0 x_1\cdots x_{n}])}{\lift{\mu}_L([x_{-n}\cdots x_{-1} \bpto x_0])}
=\lim_{n\to \oo} \frac{\lift{\mu}_L([\bpto y_{-k}\cdots y_{-1}x_0 x_1\cdots x_{n}])}{\lift{\mu}_L([\bpto x_{0}\cdots x_{n} ])}\\
&=\exp(S_k\lift{\varL}(\tau^{-k}\seq{y}))=\exp\left(S_{k}\varL(y_{-k}\cdots y_{-1}x_0 x_1\cdots x_{n}\cdots)\right).
\end{align}

It follows that $\seq{x}\mapsto \mu^s_{\seq{x}}$ is well defined for $\seq{x}\in \lift{\Sigma}_0$. Next we compute the behavior of these stable/unstable measures under iteration.

\begin{corollary}\label{cor:invarianciadesintegra}
For every $\seq{x}\in \lift{\Sigma}_0$ it holds 
\begin{enumerate}
  \item if $\seq{y}\in \tau^{-1}(\Wuloc{\tau\seq{x}})$, then
  \[
\frac{\dd \tau^{-1}\mu^u_{\tau\seq{x}}}{\dd \mu^u_{\seq{x}}}(\seq{y})=\exp(-\lift{\varLd}(\tau \seq{y})).
  \]
  \item If $\seq{y}\in \tau(\Wsloc{\tau^{-1}\seq{x}})$, then
  \[
\frac{\dd \tau\mu^s_{\tau^{-1}\seq{x}}}{\dd \mu^s_{\seq{x}}}(\seq{y})=\exp(-\lift{\varL}(\tau^{-1} \seq{y})).
  \]
  
\end{enumerate}
\end{corollary}

\begin{proof}
Compute
\begin{align*}
\tau^{-1}\mu^u_{\tau\seq{x}}([y_0\cdots y_{k}]^{u}_{\seq{x}})&=\mu^u_{\tau\seq{x}}\left([y_1\cdots y_{k}]^{u}_{\tau\seq{x}}\right)\\
&=\begin{dcases}
\exp(S_{k-1}\lift{\varLd}(\tau^2 \seq{y}))=\exp(-\lift{\varLd}(\tau \seq{y}))\cdot \mu^u_{\seq{x}}([y_0\cdots y_{k}]^{u}_{\seq{x}}) & x_1=y_1\\
0 &x_1\neq y_1.
\end{dcases}
\end{align*}
This shows that $\tau^{-1}\mu^u_{\tau\seq{x}}\sim \mu^u_{\seq{x}}$. By taking $k\mapsto\oo$ and using the Radon-Nikodym theorem we obtain the claim in the first case. The second is analogous.
 \end{proof}

It follows that we can write
 \[
   \mu^u_{\seq{x}}=\sum_{i\in \ms A} R_{x_0i}1_{[x_0i]} e^{\varLd(\seq{x}^{-}i)}\tau^{-1}\mu_{\seq{x} i}^u.
 \]

\paragraph{\textbf{Holonomy.}}Given $\seq{x},\seq{y}\in\lift{\Sigma}$ with $\seq{y}\in\Wuloc{\seq{x}}$ we define $\hu{\seq{x}, \seq{y}}:\Wsloc{\seq{x}}\to\Wsloc{\seq{y}}$ 
 \begin{align*}
 \hu{\seq{x}, \seq{y}}(\inner{\seq z^{-}}{\seq{x}^+})=\inner{\seq z^{-}}{\seq{y}^+}.
 \end{align*}
We call this map the \emph{local unstable holonomy} between $\Wsloc{\seq{x}}$ and $\Wsloc{\seq{y}}$. Note that
\[
   \forall \seq z\in\Wsloc{\seq{x}},\ \hu{\seq{x},\seq{y}}([z_{-k}\cdots z_{0}]_{\seq{x}}^s)=[z_{-k}\cdots z_{0}]_{\seq{y}}^s .
\]

Now suppose that $\seq{x},\seq{y}\in \lift{\Sigma}_0$ and let $\seq z=\inner{\seq z^{-}}{\seq{x}^{+}}\in \Wsloc{\seq{x}}$. Then for every $k\geq 0$ one has

\begin{align}\label{eq:holonomiastos}
\nonumber\hu{\seq{y}, \seq{x}}\mu^s_{\seq{y}}\left([z_{-k}\cdots z_{0}]^{s}_{\seq{x}}\right)&=\exp\left(S_k\lift{\varL}(\tau^{-k}\inner{\seq{z}^{-}}{\seq{y}^{+}}))\right)\\
\nonumber&=\exp\left(S_k\lift{\varL}(\tau^{-k}\inner{\seq{z}^{-}}{\seq{y}^{+}})-S_k\lift{\varL}(\tau^{-k}\seq{z})\right)\mu^s_{\seq{x}}\left([z_{-k}\cdots z_{0}]^{s}_{\seq{x}}\right)\\
&=\prod_{i=1}^{k}\frac{\exp\left(\lift{\varL}(\tau^{-i}\inner{\seq{z}^{-}}{\seq{y}^{+}})\right)}{\exp\left(\lift{\varL}(\tau^{-i}\inner{\seq{z}^{-}}{\seq{x}^{+}})\right)}\cdot \mu^s_{\seq{y}}\left([z_{-k}\cdots z_{0}]^{s}_{\seq{x}}\right)\\
\shortintertext{which if $\seq{z}, \hu{\seq{y},\seq{x}}(\seq{z})\in\lift{\Sigma}_1$ is equal to}
&= \frac{\exp\left(u(\inner{\seq{z}^{-}}{\seq{y}^{+}})\right)}{\exp\left(u(\inner{\seq{z}^{-}}{\seq{x}^{+}})\right)}:\frac{\exp\left(u(\tau^{-k}\inner{\seq{z}^{-}}{\seq{y}^{+}})\right)}{\exp\left(u(\tau^{-k}\inner{\seq{z}^{-}}{\seq{x}^{+}})\right)}\cdot \mu^s_{\seq{y}}\left([z_{-k}\cdots z_{0}]^{s}_{\seq{x}}\right). 
\end{align}

\begin{definition}\label{def:unstablejacobian}
The unstable Jacobian between $\Wsloc{\seq{x}}$ and $\Wsloc{\seq{y}}$ is the function 
\[
  \Jacu{\seq{x},\seq{y}}(\seq{z})=\lim_{n\to\oo} \prod_{k=1}^{n}\frac{\exp\left(\lift{\varL}(\tau^{-k}\inner{\seq{z}^{-}}{\seq{y}^{+}})\right)}{\exp\left(\lift{\varL}(\tau^{-k}\inner{\seq{z}^{-}}{\seq{x}^{+}})\right)}
\]
for $\seq{z}\in \Wsloc{\seq{x}}$ (whenever it exists).
\end{definition}

By the computation above, and since $\lift{\varL}$ has the weak Bowen property, the measures $\hu{\seq{y}, \seq{x}}\mu^s_{\seq{y}}, \mu^s_{\seq{x}}$ are equivalent. The Radon-Nikodym theorem now implies that there exists a full $\mu^s_{\seq{x}}$ measure set where the limit $\Jacu{\seq{x},\seq{y}}(\seq{z})$ exists.

\begin{corollary}\label{cor:Jacobianu}
For every $\seq{x},\seq{y}\in \lift{\Sigma}_0$ with $x_0=y_0$ the measures $\hu{\seq{y}, \seq{x}}\mu^s_{\seq{y}}, \mu^s_{\seq{x}}$ are equivalent: $\hu{\seq{y}, \seq{x}}\mu^s_{\seq{y}}=\Jacu{\seq{x},\seq{y}}\cdot \mu^s_{\seq{x}}$, where $\Jacu{\seq{x},\seq{y}}\in \Lp{\oo}{\mu^s_{\seq{x}}}$ is bounded away from zero and infinity, with constants that do not depend on $\seq{x},\seq{y}$.
\end{corollary}

Similarly, the stable holonomy is the map $\hs{\seq{x},\seq{y}}:\Wuloc{\seq{x}}\to\Wuloc{\seq{y}}$ defined by $\hs{\seq{x},\seq{y}}(\inner{\seq{x}^{-}}{\seq{z}^+})=\inner{\seq{y}^{-}}{\seq{z}^+}$. As before, we compute for $\seq{z}\in \Wuloc{\seq{x}}$
\begin{align}\label{eq:holonomiautou}
\nonumber&\hs{\seq{x},\seq{y}}\mu^u_{\seq{y}}([x_0z_1\cdots z_k]_{\seq{x}}^u)=\exp\left(S_k\lift{\varLd}(\tau\inner{\seq{y}^{-}}{\seq{z}^+})\right)\\
&=\exp\left(S_k\lift{\varLd}(\tau\inner{\seq{y}^{-}}{\seq{z}^+})-S_k\lift{\varLd}(\tau\inner{\seq{x}^{-}}{\seq{z}^+})\right)\cdot\mu^u_{\seq{x}}([x_0z_1\cdots z_k]_{\seq{x}}^u)\\
\shortintertext{which if $\seq{z}, \hs{\seq{y},\seq{x}}(\seq{z})\in\lift{\Sigma}_1$ is equal to}
&=\frac{\exp(u(\tau\inner{\seq{x}^{-}}{\seq{z}^+}))}{\exp(u(\tau\inner{\seq{y}^{-}}{\seq{z}^+}))}:\frac{\exp(u(\tau^{k+1}\inner{\seq{x}^{-}}{\seq{z}^+}))}{\exp(u(\tau^{k+1}\inner{\seq{y}^{-}}{\seq{z}^+}))}\cdot\mu^u_{\seq{x}}([x_0z_1\cdots z_k]_{\seq{x}}^u).
\end{align}

\subsection{The set of convergence of the unstable Jacobian}\label{ssec:thesetofconvergence}

 The previous \Cref{cor:Jacobianu} guarantees that, if $\seq{x},\seq{y}\in \lift{\Sigma}_0$ verify $x_0=y_0$, then for $\mu^s_{\seq{x}}$ almost every $\seq{z}\in \Wsloc{\seq{x}}$ the limit 
\[
  \lim_{k\to\oo} S_k\lift{\varL}(\tau^{-k}\hu{\seq{x},\seq{y}}(\seq{z}))-S_k\lift{\varL}(\tau^{-k}\seq{z})
\]
exists. For this, implicitly we have used that $\pi^{-1}(\Sigma_0)$ is saturated by the partition $\xi^s$. In principle we do not have means to guarantee that the same holds for the domain of the transfer function in order to deduce that $\lim_{k\to\oo} u(\tau^{-k}\hu{\seq{x},\seq{y}}(\seq{z}))-u(\tau^{-k}\seq{z})$ also exists. Nevertheless, by definition of disintegration we get the following:

\begin{corollary}\label{cor:existenciau}
 There exists a full measure invariant set $\lift{\Sigma}_2 \subset \lift{\Sigma}_1$ so that for $\seq{x},\seq{y}\in \lift{\Sigma}_2, x_0=y_0$ it holds
\[
   \exists \lim_{k\to\oo} u(\tau^{-k}\hu{\seq{x},\seq{y}}(\seq{z}))-u(\tau^{-k}\seq{z})\quad \aee{\mu^s_{\seq{x}}}(\seq{z})
 \]
 \end{corollary}

 In this part we analyze more carefully the previous convergence. For $\phi\in \Cr{\lift{\Sigma}^2}$ let
\begin{align*}
\bb{P}^u(\phi)=\int \Emu{\mu^u_{\seq{x}}}{\phi(\seq{x},\cdot)}\dd\lift{\mu}_L(\seq{x}).
 \end{align*}
Then $\bb{P}^u$ is a probability supported in the fiber bundle $\mrm{E}=\left\{(\seq{x},\seq{y})\in\lift{\Sigma}^2:\seq{y}\in\Wuloc{\seq{x}}\right\}$. If $\mrm{proj}^{(1)}, \mrm{proj}^{(2)}:\lift{\Sigma}^2\to\lift{\Sigma}$ denote the respective projections onto the first and second coordinates, we get in particular $\mrm{proj}^{(1)}\bb{P}^u=\lift{\mu}_L$.

Define the probability measure $\bb{P}^{us}$ on $\lift{\Sigma}^3$ by the formula
 \[
   \phi\in \Cr{\lift{\Sigma}^3},\ \bb{P}^{us}(\phi)=\int \Emu{\mu^s_{\seq{x}}}{\phi(\seq{x},\seq{y},\cdot)}\dd\, \bb{P}^u(\seq{x},\seq{y}).
 \]
Denote $\mrm{proj}:\lift{\Sigma}^3\to\lift{\Sigma}^2$ the projection onto the first two coordinates. From \Cref{cor:existenciau} we deduce:

\begin{lemma}\label{lem:existenciauprus}
There exists a full $\bb{P}^{us}$ measure set $\mathcal{D}^3\subset \lift{\Sigma}^3$ such that $\mrm{proj}(\mathcal{D}^3)=\lift{\Sigma}_2^2$, and so for $(\seq{x},\seq{y},\seq{z})\in\mathcal{D}^3$ it holds 
\[
  \exists \lim_{k\to\oo}  u(\tau^{-k}\hu{\seq{x},\seq{y}}(\seq{z}))-u(\tau^{-k}\seq{z}).
\]
\end{lemma}

Consider the family of functions $v_k:\lift{\Sigma}^2_2\to\R$, $v_k(\seq{x},\seq{y})=u(\tau^{-k}\seq{x})-u(\tau^{-k}\seq{y})$.

\begin{lemma}\label{lem:existev}
 There exists a full $\bb{P}^u$ measure set $\mathcal{D}^2\subset \lift{\Sigma}^2 $ so that
 \[
   (\seq{x},\seq{y})\in \mc{D}^2\Rightarrow \exists v(\seq{x},\seq{y}):=\lim_{k\to\oo} v_k(\seq{x},\seq{y})
 \]
 \end{lemma}

 \begin{proof}
 Indeed, $\mrm{proj}\bb{P}^{us}=\bb{P}^{u}$. The claim follows from this and the previous lemma.
 \end{proof}

Our next goal is showing that $v$ is the zero function.

 \begin{lemma}\label{lem:vzero}
 It holds that $v(\seq{x},\seq{y})=0$ for $\bb{P}^u$-almost every pair $(\seq{x},\seq{y})$.
 \end{lemma}

 \begin{proof}
 Suppose not and denote $B=\{(\seq x,\seq y)\in\mc{D}^2:|v(x,y)|>0\}$. Using Egoroff and Lusin's theorems one deduces the existence of constants $\eps,\eta>0$ and $k_0$ so that for $k\geq k_0$, $\bb{P}^u(\hat{B}_k)\geq \eta$, where $\hat{B}_k=\{(\seq{x},\seq{y}):|v_k(\seq{x},\seq{y})|>\eps\}$. Let $\hat{C} \subset \lift{\Sigma}_2^2$ be a compact set so that $(\seq{x},\seq{y})\mapsto u(\seq{x})-u(\seq{y})$ is continuous and $\bb{P}^u(\hat{C})\geq 1-\eta$. Observe that by the definition of $\bb{P}^{u}$ and the defining properties of disintegrations it follows that $\lift{\mu}_L(B)\geq \eta$, with $B=\mrm{proj}^{(1)} \hat{B}$.

 Take $\delta>0$ for which $\dis[C]{\seq{x}}{\seq{y}}<\delta\Rightarrow |u(\seq{x})-u(\seq{y})|<\eps$, and let $k_1\geq k_0$ so that
\[
   \sup\{\diam(\tau^{-k_1}\Wuloc{\seq{x}}):\seq{x}\in\lift{\Sigma}\}<\delta.
 \]
 This implies in particular that $\hat{B}\cap \tau^{k_1}\times \tau^{k_1}(\hat{C})=\emptyset$, therefore $\bb{P}^u(\tau^{k_1}\times \tau^{k_1}(\hat{B}))<\eta$, which in turn implies 
\[
   \lift{\mu}_L(\mrm{proj}^{(1)}\paren{\tau^{-k_1}\times \tau^{-k_1}(\hat{B}))=\lift{\mu}_L(\tau^{-k_1} B})<\eta,
 \]
which is a contradiction.
 \end{proof}

\begin{corollary}\label{cor:jacobianoinestablefull}
There exists $\mc{D} \subset \lift{\Sigma}_2^3 \subset \lift{\Sigma}_0^3$ of full $\bb{P}^{us}$ measure so that $(\seq{x},\seq{y},\seq{z})\in\mc{D}$ then
\begin{align*}
  \Jacu{\seq{y},\seq{x}}(\seq{z})&=\frac{e^{u(\seq{z})}}{e^{u(\inner{\seq{z}^-}{\seq{x}^+})}}=\prod_{k=1}^{\oo} \frac{\exp\left(\lift{\varL}(\tau^{-k}\inner{\seq{z}^{-}}{\seq{x}^{+}})\right)}{\exp\left(\lift{\varL}(\tau^{-k}\inner{\seq{z}^{-}}{\seq{y}^{+}})\right)}\\
  &=\lim_{k\to\oo}\exp\left(S_k\lift{\varL}(\tau^{-k}\inner{\seq{z}^{-}}{\seq{x}^{+}})-S_k\lift{\varL}(\tau^{-k}\inner{\seq{z}^{-}}{\seq{y}^{+}})\right).
\end{align*}
\end{corollary}

\begin{remark}
The previous corollary reveals yet another some subtlety of functions with the (weak) Bowen property. For $\varphi$ to have the weak Bowen property, $\sup_{n\geq 0}|\sum_{k=1}^n\varphi(\tau^{-k}\seq{x})-\varphi(\tau^{-k}\seq{y})|<\oo$ for $\seq{x},\seq{y}\in\Wuloc{\seq{z}}$, it turns out that in fact $\exists\lim_n\sum_{k=1}^n \varphi(\tau^{-k}\seq{x})-\varphi(\tau^{-k}\seq{y})<\oo$ for points $\seq{x},\seq{y}$ chosen with respect to some fully supported measure in the same unstable sets.
\end{remark}

\subsection{New conditionals and the transverse measure}\label{ssec:conditionalsandtransverse}

 We now compare the unstable disintegrations with $\mu_L$, and the the transverse measure with $\mu_{L^{\dagger}}$. The method is based on what is done in \cite{EqStatesCenter} for regular Hölder potentials (see also \cite{Carrasco2023}), which has its origins in the work of  Haydn \cite{localproductstructureHaydn}, Bowen-Ruelle \cite{BowenRuelle}, and Ledrappier-Young \cite{LedYoungI}. Here, on the other hand, we are dealing with a more delicate situation, since we have to compare disintegrations corresponding to a priori different invariant measures, which typically implies that their domain of definition are mutually singular. Another difficulty comes from the fact that we want to construct the unstable disintegration everywhere by using a measure equivalent to $\mu_L$ (which is defined on a different space). In the literature the identification between $\Sigma$ and the local unstable manifold is sometimes not made explicit, which does not pose a serious problem since the Jacobian is defined everywhere. 

 For $i\in\ms{A}$ let $\mu_L^{i}=\one_{[i]}\mu_L$. Recall the definition of the transfer operator given in \Cref{ssec:transfer}.

\begin{lemma}\label{lem:Lonea}
If $a\in \ms A$, $\seq{z}\in\Sigma_0$ it holds $\Rop[L]1_{[a]}(\seq{z})=R_{a x_0}e^{\varL(a \seq{z})}$.
\end{lemma}

\begin{proof}
Direct computation.
\end{proof}

From this we deduce:

\begin{corollary}\label{cor:iteradomuLa}
It holds $\tau \mu_L^{a}=\sum_i R_{a i} e^{\varL(a\bullet)}\mu_L^{i}= \sum_{i} R_{a i}1_{[i]}e^{\varL(a\bullet)}\mu_L$.
\end{corollary}

If $\seq{x}\in \lift{\Sigma}$ we define 
\begin{align}\label{eq:etax}
\eta_{\seq{x}}=p^{u}_{\seq{x}}\mu_L^{x_0}.
\end{align}
This is a finite measure supported on $\Wuloc{\seq{x}}$, which moreover verifies:

\begin{lemma}\label{lem:muLequivu}
For every $\seq{x}\in \lift{\Sigma}_0$ the measure $\eta_{\seq{x}}$ is equivalent to $\mu^u_{\seq{x}}$. Furthermore, the Jacobian is uniformly bounded from above and below by some constant independent of $\seq{x}$.
\end{lemma}

\begin{proof}
This follows from \eqref{eq:medidau} and the Gibbs property of $\mu_L$.
\end{proof}

Note that $\seq{x}^{-}=\seq{y}^{-}\Rightarrow \eta_{\seq x}=\eta_{\seq y}$. Given $\seq{x}\in \lift{\Sigma}$ for each symbol $i$ with $R_{x_0 i}=1$ we choose $\seq{x}_{(i)}\in \tau(\Wuloc{\seq{x}})\cap [i]$.

\begin{lemma}\label{lem:invariancemula}
For every $\seq{x}\in \lift{\Sigma}$ it holds $\tau \eta_{\seq{x}}=\sum_{i}R_{x_0 i}e^{\lift{\varL}\circ \tau^{-1}}\eta_{\seq{x}_{(i)}}=\sum_{i}R_{x_0 i}e^{\lift{\varL}\circ \tau^{-1}}\eta_{\seq{x}^{-}_{(i)}}$.
\end{lemma}

\begin{proof}
It is direct to check that $\tau\circ p^{u}_{\seq{x}}=\sum_{i} 1_{[x_0 i]}\cdot p_{\seq{x}_{(i)}}^u\circ \tau$, while
\[
  p_{\seq{x}_{(i)}}^u\left(\one_{[i]}e^{\varL(x_0\bullet)}\mu_L\right)=e^{\varL(x_0\bullet)\circ (p_{\seq{x}_{(i)}}^u)^{-1}}\eta_{\seq{x}_{(i)}}=e^{\lift{\varL}\circ \tau^{-1}}\eta_{\seq{x}_{(i)}}.
\] 
Hence, by \Cref{cor:iteradomuLa},
\begin{align*}
\tau\eta_{\seq{x}}=(\sum_{i} \one_{[x_0 i]}\cdot p_{\seq{x}_{(i)}}^u\circ \tau)\mu_L^{x_0}=\sum_{i} R_{x_0i}e^{\lift{\varL}\circ \tau^{-1}}\eta_{\seq{x}_{(i)}}.
\end{align*}
\end{proof}

For $\seq{x}\in \lift{\Sigma}_1$ the function $\Wuloc{\seq{x}}\cap\lift{\Sigma}_2\ni \seq{z}\to e^{-u(\seq{z})}$ is in $\Lp{\oo}{\eta_{\seq{x}}}$, therefore 
\begin{align}
&\dd \nu_{\seq{x}}:=e^{-u}\dd \eta_{\seq{x}}\\
&\dd \hat{\nu}_{\seq{x}}:=e^{u(\seq{x})-u}\dd \eta_{\seq{x}}=e^{u(\seq{x})}\dd\nu_{\seq{x}}
\end{align}
are well defined finite measures on $\Wuloc{\seq{x}}$. Our next goal is proving that $\{\nu_{\seq{x}}:\seq{x}\in \lift{\Sigma}_2\}$ verifies
\[
  \mu_{\seq{x}}^u=\frac{\nu_{\seq{x}}}{\nu_{\seq{x}}(\Wuloc{\seq{x}})}=\frac{\hat{\nu}_{\seq{x}}}{\hat{\nu}_{\seq{x}}(\Wuloc{\seq{x}})}.
\]

We start analyzing the behavior of $\{\nu_{\seq{x}}:\seq{x}\in \lift{\Sigma}_2\}, \{\hat{\nu}_{\seq{x}}:\seq{x}\in \lift{\Sigma}_2\}$ under iteration.

\begin{lemma}\label{lem:quasiinvariancenux}
 For every $\seq{x}\in \lift{\Sigma}_1$ it holds 
 \begin{align*}
 \tau^{-1}\nu_{\tau\seq{x}}=1_{[x_0x_1]}\exp(-\lift{\varLd}(\seq{x}))\nu_{\seq{x}}
 \end{align*}
 As a consequence,
\begin{align*}
 \nu_{\seq{x}}=e^{\lift{\varLd}(\seq{x})}\tau^{-1}\left(\sum_{i} R_{x_0i}\nu_{\seq{x}_{(i)}}\right)\\
\hat{\nu}_{\seq{x}}=e^{\lift{\varL}(\seq{x})}\tau^{-1}\left(\sum_{i} R_{x_0i}\hat{\nu}_{\seq{x}_{(i)}}\right).\\
 \end{align*}
 \end{lemma} 

 \begin{proof}
 Write $\seq{w}=\tau\seq{x}$. By \Cref{lem:invariancemula} we have $\tau\left(1_{[x_0 i]}\eta_{\seq{x}}\right)=e^{-\lift{\varL}\circ \tau^{-1}}\eta_{\seq{x}_{(i)}}$, hence in particular $\tau^{-1}\eta_{\seq{w}}=1_{[x_0x_1]}e^{-\lift{\varL}}\eta_{\seq{x}}$. Therefore 
\begin{align*}
 \tau^{-1}\nu_{\seq w}&=\tau^{-1}\left(e^{-u}\eta_{\seq{w}}\right)=1_{[x_0 x_1]}e^{-u\circ \tau-\lift{\varL}}\eta_{\seq{x}}=1_{[x_0 x_1]}e^{-u-\lift{\varLd}}\eta_{\seq{x}}=1_{[x_0 x_1]}e^{-\lift{\varLd}(\seq{x})}\nu_{\seq{x}}.
\end{align*}
\end{proof}

The above lemma allows us to extend each $\nu_{\seq{x}}$ (for $\seq{x}\in\lift{\Sigma}_1$) to a Radon measure on $\Wu{x}$ with some quasi-invariance property.

\begin{lemma}\label{lem:medidanuxglobal}
 There exists a family of measures $\left\{\nu_{\seq{x}}\right\}_{\seq{x}\in\lift{\Sigma}_1}$, where $\nu_{\seq{x}}$ is a Radon measure on $\Wu{\seq{x}}$ satisfying: 
\begin{enumerate}
   \item for every $n\in\N$, $\tau^{-n}\nu_{\tau^n \seq{x}}=e^{-S_n\lift{\varLd}(\seq{x})}\nu_{\seq{x}}$.
   \item $\seq{x}, \seq{z}\in\lift{\Sigma}_1, \seq{z}\in\Wu{\seq{x}}$ implies $\nu_{\seq{x}}=\lim_{n}e^{S_n\lift{\varLd}(\tau^{-n}\seq{x})-S_n\lift{\varLd}(\tau^{-n}\seq{z})}\nu_{\seq{z}}$.
\end{enumerate}
\end{lemma}

\begin{proof}
Fix $\seq{x}\in\lift{\Sigma}_1$, and for $n\geq 0$ write $\seq{x}^{(n)}=\tau^{-n}\seq{x}$ and $\nu^{(n)}_{\seq{x}}\defeq e^{-S_n\lift{\varLd}(\seq{x}^{(n)})}\tau^n \nu_{\seq{x}^{(n)}}$. It follows from the previous Lemma that this family of measures satisfies $\nu^{(n)}_{\seq{x}}=\nu^{(n-1)}_{\seq{x}}$ on $\tau^{n-1}\Wuloc{\seq{x}^{(n-1)}}$, for every $n\geq 1$, and hence $\lim_n\ \nu^{(n)}_{\seq{x}}$ is a well defined Radon measure on $\bigcup_{n\geq 0}\tau^n\Wuloc{\seq{x}^{(n)}}=\Wu{\seq{x}}$, which coincides with $\nu_{\seq{x}}$ on $\Wuloc{\seq{x}}$. From its definition, it is clear that $\nu_{\tau\seq{x}}=e^{-\lift{\varLd}(\seq{x})}\tau\nu_{\seq{x}}$. The last part follows by noticing that if $\seq{z}\in\Wu{\seq{z}}$, there exists some $k\geq 0$ so that $\nu_{\seq{x}^{(k)}}=\nu_{\seq{z}^{(k)}}$, and then using the quasi-invariance.
\end{proof}

We obtain a similar statement for the measures $\hat{\nu}_{\seq{x}}$. We remark the dependence on the base point of these measures, even for points in the same local unstable manifold. 

\begin{lemma}\label{lem:medidahnuxglobal}
There exists a family of measures $\left\{\hat{\nu}_{\seq{x}}\right\}_{\seq{x}\in\lift{\Sigma}_1}$, where $\hat{\nu}_{\seq{x}}$ is a Radon measure on $\Wu{\seq{x}}$ satisfying: 
\begin{enumerate}
  \item for every $n\in\N$, $\tau^{-n}\nu_{\tau^n \seq{x}}=e^{-S_n\varLd(\seq{x})}\hat{\nu}_{\seq{x}}$.
  \item $\seq{x}, \seq{z}\in\lift{\Sigma}_1, \seq{z}\in\Wu{\seq{x}}$ implies $\hat{\nu}_{\seq{x}}=e^{u(\seq{x})-u(\seq{z})}\lim_{n}e^{S_n\lift{\varLd}(\tau^{-n}\seq{x})-S_n\lift{\varLd}(\tau^{-n}\seq{z})}\hat{\nu}_{\seq{z}}$. 
\end{enumerate}
\end{lemma}

\begin{proof}
Use the previous Lemma and define $\hat{\nu}_{\seq{x}}$ globally as $e^{u(\seq{x})}\nu_{\seq{x}}$. Since $\lift{\varLd}(\seq{x})=\lift{\varL}(\seq{x})-u(\seq{x})+u(\tau \seq{x})$, the first equality follows. The others are direct from this. 
\end{proof}

Now we understand the quasi-invariance properties of these measures, and we can compute how they transform under holonomy.

\begin{corollary}\label{cor:holonomynux}
Suppose that $\seq{x},\seq{y}\in\lift{\Sigma}_1$. Then for every $\seq{z}\in\lift{\Sigma}_2\cap\Wuloc{\seq{x}}\cap \hs{\seq{y},\seq{x}}\left(\lift{\Sigma}_2\cap \Wuloc{\seq{y}}\right)$ it holds 
\[
  \frac{\dd  \hs{\seq{y},\seq{x}}\nu_{\seq{y}}}{\dd \nu_{\seq{x}}}(\seq{z})=e^{u(\seq{z})-u(\hs{\seq{x},\seq{y}}\seq{z})}.
\]
\end{corollary}

\begin{proof}
It is no loss of generality to suppose $\seq{x}\in\lift{\Sigma}_2, \seq{y}=\hs{\seq{x},\seq{y}}(\seq{x})\in\lift{\Sigma}_2$. Observe that 
\[
 \hs{\seq{y},\seq{x}}\nu_{\seq{y}}([x_0\ldots x_n]_{\seq{x}}^u)=\nu_{\seq{y}}([x_0\ldots x_n]_{\seq{y}}^u)=e^{S_n\lift{\varLd}(\seq{y})}\nu_{\tau^n \seq{y}}([x_n]^u_{\tau^n \seq{y}}), 
\]
and likewise $\nu_{\seq{x}}([x_0\ldots x_n]_{\seq{x}}^u)=e^{S_n\lift{\varLd}(\seq{x})}\nu_{\tau^n \seq{x}}([x_n]_{\tau^n\seq{x}}^u)$. Since
\[
  S_n\lift{\varLd}(\seq{y})-S_n\lift{\varLd}(\seq{x})=u(\tau^n\seq{y})-u(\tau^n\seq{x})-(u(\seq{y})-u(\seq{x})),
\]
we get 
\[
  \frac{\hs{\seq{y},\seq{x}}\nu_{\seq{y}}([x_0\ldots x_n]_{\seq{x}}^u)}{\nu_{\seq{x}}([x_0\ldots x_n]_{\seq{x}}^u)}=e^{u(\seq{x})-u(\seq{y})}\left(e^{u(\tau^n\seq{y})-u(\tau^n\seq{x})}\frac{\nu_{\tau^n \seq{x}}([x_n]_{\tau^n\seq{y}}^u)}{\nu_{\tau^n \seq{x}}([x_n]_{\tau^n\seq{x}}^u)}\right).
\]
The term in parentheses can be seen to converge to one: indeed
\begin{itemize}
  \item $\lim_{n\to\oo} u(\tau^n\seq{y})-u(\tau^n\seq{x})=0$ due to \Cref{lem:vzero}, as we are assuming that $\seq{x},\seq{y}\in\lift{\Sigma}_2$, and
  \item $\lim_{n\to\oo} \frac{\nu_{\seq{x}}([x_0\ldots x_n]_{\seq{x}}^u)}{\nu_{\seq{x}}([x_0\ldots x_n]_{\seq{y}}^u)}=1$, again by the same Lemma and arguing as in its proof (use Egoroff's theorem with respect to the reference measure $\mu_L$).
\end{itemize}
By to the Lebesgue-Radon-Nikodym theorem we conclude the claim. 
\end{proof}

\paragraph{The transverse measure} It is time to deal with the transverse measure. For $a\in\ms A$ let $\mu_{L^{\dagger}}^{a}=\one_{[a]}\mu_{L^{\dagger}}$, and if $\seq{x}\in \lift{\Sigma}$ consider the measures
 \begin{align*}
 &\zeta_{\seq{x}}=p^{s}_{\seq{x}}\mu_{L^{\dagger}}^{x_0}\\
 &\theta_{\seq{x}}:=\sum_{j} R_{j x_0}\tau \zeta_{\seq{x}^{j}}
\end{align*}
 where $\tau\seq{x}^{(j)}\in [jx_0]_{\seq{x}}^s$. It is easy to see that $\zeta_{\seq{x}}$ (which is completely analogous to $\eta_{\seq{x}}$) gives the transverse measure of the unstable disintegration on the rectangle that contains $\seq{x}$. Correspondingly, the definition of $\theta_{\seq{x}}$ is made in order to compatibilize the dynamics of $\tau$ and $\tau^{-1}$, in particular, to guarantee the following lemma. 

\begin{lemma}\label{lem:iteradotheta}
For every $\seq{x}\in \lift{\Sigma}$ it holds $1_{[x_0x_1]}\tau^{-1}\theta^{\tau \seq{x}}=e^{\lift{\varLd}}\theta_{\seq{x}}$.
\end{lemma}

\begin{proof}
By definition, $1_{[x_0x_1]}\tau^{-1}\theta_{\tau \seq{x}}=\zeta_{\seq{x}}$. Arguing analogously as in \Cref{cor:iteradomuLa} one verifies that $\tau \zeta_{\seq{x}^{j}}=1_{[j x_0]}e^{-\lift{\varLd}}\zeta_{\seq{x}}$, which leads us to
\begin{align*}
&\zeta_{\seq{x}}=\sum_{j x_0}e^{\lift{\varLd}}\tau\zeta_{\seq{x}^{j}}=e^{\lift{\varLd}}\theta_{\seq{x}}
\Rightarrow 1_{[x_0x_1]}\tau^{-1}\theta_{\tau \seq{x}}=e^{\lift{\varLd}}\theta_{\seq{x}}.
\end{align*}
\end{proof}

\begin{remark}\label{rmk:otherparametrizproduct}
When defining the product structure $\Sigma^{\dagger}_{x_0}\times\Sigma_{x_0}\xrightarrow[]{\inner{\bullet}{\bullet}} [x_0]\cap \lift{\Sigma}$, one could instead consider a different structure preserving map, changing the first factor for \(\bigcup_{j/ R_{jx_0}=1}[j]\), which also parametrizes $\Wsloc{\seq{x}}$. The measure $\theta_{\seq{x}}$ corresponds to $\mu_{L^{\dagger}}$ in these coordinates. 
\end{remark}

We keep proceeding in like manner as before: for $\seq{x}\in\lift{\Sigma}_1$ let
\begin{align}
\hat\theta_{\seq{x}}=e^{u\circ\tau}\theta_{\seq{x}}.
\end{align}
It follows that $\{\hat\theta_{\seq{x}}:\seq{x}\in\lift{\Sigma}_2\}$ can be extended to a family of Radon measures on stable sets, verifying the quasi-invariance condition 
\begin{equation}\label{eq:quasiinvariancehtheta}
   \tau^{-1}\hat\theta_{\tau\seq{x}}=e^{\varL(\seq{x})}\hat\theta.
 \end{equation}

\begin{proposition}\label{pro:quasiinvarianciatransversal}
Let $\seq{x},\seq{y}\in\lift{\Sigma}_2$ with $x_0=y_0$. Then $\hat\theta_{\seq{x}}=e^{u-u\circ \hu{\seq{y},\seq{x}}}\hu{\seq{y},\seq{x}}\hat\theta_{\seq{y}}$.
\end{proposition}

\begin{proof}
Completely analogous to \Cref{cor:holonomynux}, using the previous quasi-invariance property of the family.
\end{proof}

\paragraph{The product structure of $\lift{\mu}_L$} We now use the families $\{\hat{\nu}_{\seq{x}}:\seq{x}\in\lift{\Sigma}_2\}, \{\hat\theta_{\seq{x}}:\seq{x}\in\lift{\Sigma}_2\}$ to construct a measure on the full two-sided shift. For each $a\in \ms A$ choose $\seq w=\seq w^{a}\in \lift{\Sigma}_2\cap [a]$ and define the measure
\[
  \mm^{\seq w}=\int \hat{\nu}_{\seq{x}}\dd\ \theta_{\seq w}(\seq{x}) 
\]
This is a well defined measure on $[a]$ of full support, therefore $\mm=\sum_{a\in\ms A}\mm^{\seq{w}^{a}}$ has full support in $\lift{\Sigma}$.

\begin{lemma}\label{lem:localmeasureindependent}
If $\seq w,\seq{z}\in \lift{\Sigma}_2\cap [a]$ then $\mm^{\seq w}=\mm^{\seq{z}}$.
\end{lemma}

\begin{proof}
Indeed, by $2)$ of \Cref{lem:medidahnuxglobal} and \Cref{pro:quasiinvarianciatransversal} we get
\begin{align*}
\mm^{\seq w}=\int_{\Wsloc{\seq w}} \hat{\nu}_{\seq{x}}\dd\hat\theta_{\seq w}(\seq{x})=\int_{\Wsloc{\seq{z}}} \hat{\nu}_{\seq{x}}e^{u(\seq{x}')-u(\seq{x})}\dd \hat\theta_{\seq{z}}(\seq{x}')=\int_{\Wsloc{\seq{z}}}\hat{\nu}_{\seq{x}'}\dd\hat\theta_{\seq{z}}=\mm^{\seq{z}}. 
\end{align*}
\end{proof}

We write $\mm^{a}=\mm^{\seq w}$, where $\seq w\in \lift{\Sigma}_2\cap [a]$.

\begin{corollary}\label{cor:mminvariant}
 The measure $\mm$ is $\tau$ invariant.
 \end{corollary}

\begin{proof}
 Consider $U=[w_{-k}\cdots w_{-1}\bpto w_0 w_1\cdots w_{l}]$: $\tau(U)=[w_{-k}\cdots w_{-1} w_0\bpto w_1\cdots w_{l}]$. By the previous corollary we can write
\begin{align*}
 \mm^{w_1}(\tau U)&=\int_{[w_{-k}\cdots w_0w_1]^s_{\tau \seq w}} \hat{\nu}_{\seq{y}}(\tau [w_0w_1\cdots w_l]^u_{\seq w})\dd\ \hat\theta_{\tau \seq w}(\seq{y})\\
 \intertext{which together with $1)$ of \Cref{lem:medidahnuxglobal} gives}
 &=\int_{[w_{-k}\cdots w_0 w_1]^s_{\tau \seq w}} e^{-\varL(\tau^{-1}\seq{y})}\hat{\nu}_{\tau^{-1}\seq{y}}(\one_{[w_{0}\cdots w_l]^u_{\seq w}}\circ\tau^{-1}(\seq{y}))\dd\hat\theta_{\tau \seq w}(\seq{y}).
\intertext{Using \Cref{eq:quasiinvariancehtheta} we finally get}
 &=\int_{[w_{-k}\cdots w_0]^s_{\seq w}} \nu_{\seq{x}}([w_0\cdots w_l]^u_{\seq w})\dd\hat\theta_{\seq w}(\seq{x})=\mm^{w_0}(U).
\end{align*}
Sets $U$ of the previous type form a basis for the topology of $\lift{\Sigma}$, hence $\tau^{-1}\mm=\mm$.
 \end{proof}

Let $\hat{\mm}=\frac{\mm}{\mm(\lift{\Sigma})}\in \PTM{\tau}{\lift\Sigma}$. By definition, its unstable disintegration is given by normalizing the measures $\hat{\nu}_{\seq{x}}$ (equivalently, normalizing the measures $\nu_{\seq{x}}$). Note that $\hat\mm\ll\lift{\mu}_L$, and since the later is ergodic we conclude:

\begin{corollary}\label{cor:disintegraciondewmuL}
$\lift{\mu}_L=\hat{\mm}$. In particular, for $\lift{\mu}_L$ almost every $\seq{x}$ it holds  $\mu_{\seq{x}}^u=\frac{\nu_{\seq{x}}}{\nu_{\seq{x}}(\Wuloc{\seq{x}})}=\frac{\hat{\nu}_{\seq{x}}}{\hat{\nu}_{\seq{x}}(\Wuloc{\seq{x}})}$.
\end{corollary}

The fact that $u\in \Lp{\oo}{\lift{\mu}_L}$ implies:

\begin{corollary}\label{cor:productstructureofwmuL}
The measure $\lift{\mu}_L$ is equivalent to $\mu_L\times\mu_{L^{\dagger}}$, with uniformly bounded Radon-Nikodym derivative.
\end{corollary}

\paragraph{\textbf{The set of definition of the potential}}Denote 
\[
  c_{\seq{x}}=\nu_{\seq{x}}(\Wuloc{\seq{x}}).
\]
This defines a measurable function on the $u$-saturated set $(\pi^{\dagger})^{-1}(\Sigma_0^{\dagger})$.

 By \Cref{cor:holonomynux} and \eqref{eq:holonomiautou} we get that, for $\seq{x},\seq{y}\in\lift{\Sigma}_2, \seq{y}\in\Wsloc{\seq{x}}$ implies
 \begin{align*}
 &\frac{c_{\seq{y}}}{c_{\seq{x}}}\frac{\dd  \hs{\seq{y},\seq{x}}\mu_{\seq{y}}^u}{\dd \mu_{\seq{x}}^u}(\seq{x})=\frac{\dd  \hs{\seq{y},\seq{x}}\nu_{\seq{y}}}{\dd \nu_{\seq{x}}}(\seq{x})=e^{u(\seq{x})-u(\seq{y})}\\
 &\Rightarrow  \frac{c_{\seq{y}}}{c_{\seq{x}}} e^{u\circ \tau(\seq{x})-u\circ\tau(\seq{y})}=e^{u(\seq{x})-u(\seq{y})}\\
 &\frac{c_{\seq{y}}}{c_{\seq{x}}}=e^{\lift{\varLd}(\seq{y})-\lift{\varLd}(\seq{x})}
 \shortintertext{which together with the quasi-invariance property of $\nu_{\seq{x}}$ yield}
 &\nu_{\tau\seq{x}}(\tau\Wuloc{\seq{x}})=\nu_{\tau\seq{y}}(\tau\Wuloc{\seq{y}}).
 \end{align*}
 It follows that for every $a\in\ms A$ there exists some constant $E_a\in\R$ so that for $\aee{\mu_{L^\dagger}}\seq{x}$, 
\[
  \varLd(\seq{x})=\log c_{\seq{x}}+E_a
 \]
 One sees that defining the value of $\varLd(\seq{x})$ is equivalent to define the normalizing constant $c_{\seq{x}}$. for $\seq{x}\in\Sigma^\dagger$. Similar considerations apply to the function $\varL$.

\subsection{Bernoulli property}

We now reap the consequences of our work. For that, we need some definitions from abstract Ergodic Theory. Let $(X,\BorelM[X],\mu)$ be a Borel probability space and $T:X\toit$ a (measure preserving) automorphism. If $\ms F \subset \BorelM[X]$ is a $\sigma-$algebra or a partition by measurable sets, we write
 \[
   -\oo\leq k\leq l\leq \oo\Rightarrow \ms{F}_{k}^l=\bigvee_{i=k}^l T^{-i}\ms{F}.
 \]
When $\ms F$ is a finite partition generating $\BorelM[X]$, one considers the measure $\nu$ obtained extending 
\[
   \nu(A\cap B):=\mu(A)\mu(B)\quad A\in \ms{F}_{-\oo}^0, B\in \ms{F}_{0}^{\oo}.
 \]
Recall that $\ms F$ is a generator for $T$ if $\ms{F}_{-\oo}^{\oo}=\BorelM[X]$.

 \begin{definition}[Ledrappier]
 The process $(X,T)$ is quasi-Bernoulli if $\nu$ is equivalent to $\mu$.  
 \end{definition}
 A counterpart is given by the notion of Weak Bernoulli process of Ornstein and Friedmann. 

\begin{definition}
 The partition $\ms F$ is said to be Weak Bernoulli if for any $\eps>0$ there exists $N\in\N$ so that for every $n\geq 0$
 \[
   \sum_{\substack{A\in \ms{F}^{n}_0\\ B\in \ms{F}_{-N-n}^{-N}}} |\mu(A\cap B)-\mu(A)\mu(B)|<\eps.
 \] 
In this case $(T,\BorelM[X])$ is said to be a Weak Bernoulli process.
 \end{definition}

\begin{theorem}[Ledrappier, \cite{Ledrappier_1976}]
For a generator $\ms{F}$, $(X,T)$ is quasi-Bernoulli if and only if $\ms{F}$ is a Weak Bernoulli partition.
 \end{theorem}

\begin{theorem}[Ornstein and Friedmann, \cite{FriedOrns}]
If $(X,T)$ is a Weak Bernoulli process, then it is isomorphic to a Bernoulli process, meaning a system having a finite independent generator.
\end{theorem}

Back to $(\lift{\Sigma},\lift{\mu}_L,\tau)$, note that $\xi=\{[a]:a\in\ms A\}$ is a generator.

\begin{corollary}\label{cor:Bernoulli}
The partition  $\xi$ is Weak Bernoulli. Therefore $(\lift{\Sigma},\lift{\mu}_L,\tau)$ is isomorphic to a Bernoulli shift.
\end{corollary}

\begin{proof}
Indeed, \Cref{cor:productstructureofwmuL} shows that $(\lift{\Sigma},\tau)$ is quasi-Bernoulli, and the statement follows combining both theorems cited above.
\end{proof}


\section{Proofs of Theorems A and B}
\label{sec:proofofmain}

In this section we summarize the constructions developed above and prove the main results \Cref{thm:A,thm:B}. 
The limit theorems (\Cref{thm:C,thm:D}) and their corollaries are discussed in \Cref{sec:the_central_limit_theorem_for_quasi_morphisms}. 
We also establish the auxiliary \Cref{thm:compactification}.

Let $M$ be a closed locally $\CAT$ space and consider $\lie{g}=(g_t)_{t\in \R}$ the associated geodesic flow, parametrized in such a way that it is a metric Anosov flow. Consider the continuous linear isomorphism $B_{*}:\QMt{\Gamma}\to \QMt[a]{\Per{\lie g}}$ given in \Cref{thm:quasimorphismgrouptoporbit}. Fix $\nu_{\lie g}\in \ETM{\lie g}{\ms E}$ of full support.

Let $\mathfrak{t}=(\tau_t)_{t\in\R}:\lift{\Sigma}_f\toit$ be a suspension flow as in \Cref{thm:symbolicrepresentation} with corresponding semi-conjugacy $h$. This is also a metric Anosov flow, and we emphasize that there exists an open and dense subset $\ms E_0\subset \ms E$ such that every $\alpha\in\Per{\mathfrak g}$ intersecting $\ms E_0$ lifts to a unique periodic orbit of $\mathfrak t$. Note also that there exists a natural involution $I_f:\lift{\Sigma}_f\toit$ that interchanges the direction of orbits, and $h\circ I_f=I\circ h$.

We now describe the sequence of constructions that associate to a class $[L_{\mathfrak g}]\in \QMt[a][]{\Per{\lie g}}$ a Bowen function on $\Sigma$. Given $[L_{\lie g}]\in \QMt[a][]{\Per{\lie g}}$, it induces a function $L_{\lie t}: \Per{h^{-1}(\ms{E}_0)}\to\R$, with the quasimorphism property
\[
	\forall \alpha,\beta \text{ such that } \alpha\star\beta\in \Per{h^{-1}(\ms E_0)}\Rightarrow |L_{\mathfrak t}(\alpha\star\beta)-L_{\mathfrak t}(\alpha)-L_{\mathfrak t}(\beta)|\leq \norm{\delta L_{\mathfrak g}}.
\]
Observe that if $L_{\mathfrak g}\sim L_{\mathfrak g}'$, then $\norml{L_{\mathfrak t}-L_{\mathfrak t}'}<\oo$.

Recall that by \Cref{lem:bijectionperiodic} there exists a bijection $\Per{\lie t}\leftrightarrow \Per{\tau}$, and thus there exists a bijection $S:\Per{\ms E_0}\to\Sigma_0 \subset \Sigma$, where $\Sigma_0$ is a dense set of periodic orbits of $\tau$. We use $S$ to transfer the quasimorphism structure to $\Sigma_0$, and then extend it to a locally constant quasicocycle $\bm L=(L^{(n)})_n$.

Below, for $\al\in \Per{\lie g}$ we write 
\[
	\mm_{\al}(L_{\lie g})=\lim_{n\to \oo}\frac{1}n L_{\lie g}(\al^n), \quad \al^n=\underbrace{\al\star\al\star\cdots \al}_{n\text{ times }} 
\]
for the value of the periodic measure supported on $\al$, evaluated on $L_{\lie g}$. Similarly for periodic measures on $\Sigma$.

By construction of $\bm{L}$, 
\[
	\alpha\in \Per{\ms E_0}\rightarrow \mm_\al(L_{\lie g})=\mm_{S(\al)}(\bm{L})
\]

\begin{lemma}\label{lem:cohpartecohtotal}
	Suppose that $L_{\lie g}, L_{\lie g}'\in\QM{\Per{\ms E}}$ are such that for every $\al\in\Per{\ms E_0}$, $\mm_\al(L_{\lie g})=\mm_\al(L_{\lie g}')$.
	Then $L_{\lie g}\sim L_{\lie g}'$. 
\end{lemma}

\begin{proof}
Arguing as in \Cref{ssec:Livsicquasicocycle},  it suffices to show that $\mm_\al(L_{\lie g})=\mm_\al(L_{\lie g}')$ for every $\al\in \Per{\ms E}$. After homogenizing $L_{\lie g}, L_{\lie g}'$ (thus $L_{\lie g}(\al^n)=nL_{\lie g}(\al)$) we can assume that if $\alpha\in \Per{\ms E}$ then $\mm_{\alpha}(L_{\lie g})=L_{\lie g}(\alpha)$, and similarly for $L_{\lie g}'$.

Fix $\alpha \in \Per{\ms E}$. Using densitiy of $\ms E_0$ we can choose $\beta\in \Per{\ms E_0}$ so that $\alpha\star\beta\in \Per{\ms E_0}$. Then for every $n\in\N$,	
    \begin{align*}
    |L_{\mathfrak g}(\alpha^n)-L_{\mathfrak g}'(\alpha^n)|&\leq |L_{\mathfrak g}((\alpha\star\beta)^n)-L_{\mathfrak g}(\alpha^n)-L_{\mathfrak g}(\beta^n)|+|L_{\mathfrak g}'((\alpha\star\beta)^n)-L_{\mathfrak g}'(\alpha^n)-L_{\mathfrak g}'(\beta^n)|\\
    &\leq \norm{\delta L_{\mathfrak g}}+\norm{\delta L_{\mathfrak g}'}
    \end{align*}
and thus $L_{\mathfrak g}(\alpha)=L_{\mathfrak g}'(\alpha)$.
\end{proof}

As a consequence, if $\bm{L}, \bm{L}'$ are two induced quasi-cocyles so that for every $\seq x\in \Sigma_0$, $\mm_{\seq x}(\bm{L})=\mm_{\seq x}(\bm{L}')$, then $\bm{L}\sim\bm{L}'$. This in particular shows that the quasicocycle corresponding to $L_{\lie g}$ is cohomologous to the one corresponding to $-L_{\lie g}\circ I$.

Fix $\nu_{\lie g}\in \PTM{\lie g}{\ms E}$ of full support, and let $\mu \in \PTM{\tau}{\Sigma}$ be the corresponding measure. Since $\mu$ is fully supported we can apply \Cref{thm:representation1} and use $\bm{L}$ to obtain $\varphi_L\in \Bow[][]{\Sigma}$. Recall that the cohomology class of Bowen function $\phi\in\Bow[][\mu]{\Sigma}$ is represented uniquely by the cohomology class of some Bowen function $\hat\phi\in \Bow[][\nu_{\lie t}]{\Sigma_f}$ (\Cref{cor:Livsicsuspensions}). By the previous paragraph, $\varphi_L$ has the following antisymmetry property.

\begin{definition}
A Bowen function $\phi\in \Bow[][]{\Sigma}$ is antisymmetric if $\hat\phi\sim -\hat\phi\circ I_f$. The set of cohomology classes of antisymmetric Bowen functions will be denoted as $\Bow[][a]{\Sigma}/\thicksim$. 	
\end{definition}

Note that $\Bow[][a]{\Sigma}/\thicksim \subset \Bow[][]{\Sigma}/\thicksim$ is closed, thus a Banach space itself.

Altogether, this construction defines a map 
\[
	\Phi:  \QMt[a][]{\Per{\lie g}}\to \Bow[a][\mu]{\Sigma}/\thicksim\qquad \qquad \Phi([L_{\mathfrak g}])=[\varphi_{\bm{L}}].
\]
given as
\[
	[L_{\lie g}]\mapsto [L_{\lie t}|h^{-1}(\ms E_0)]\mapsto [L|\Sigma_0]\mapsto [\bm{L}]\mapsto [\varphi_L].
\]

\begin{lemma}
The map $\Phi$ is a Banach isomorphism.
\end{lemma}

\begin{proof}
Linearity is immediate from the construction, and continuity is obtained by following the norms during the construction. It is also injective, due to \Cref{lem:cohpartecohtotal} and the discussion after its proof. 

It remains to show surjectivity. Consider the Banach isomorphism $\Omega$ obtained concatenating the maps $\Gamma_1: \Bow[][\mu]{\Sigma}/\thicksim\to \Bow[][\lift\mu]{\lift\Sigma}/\thicksim$ and $\Gamma^{-1}: \Bow[][\lift\mu]{\lift\Sigma}/\thicksim\to \Bow[][\nu_{\lie t}]{\lift\Sigma_f}/\thicksim$  and $h_*:\Bow[][\nu_{\lie t}]{\lift\Sigma_f}/\thicksim\to \Bow[][\nu_{\lie g}]{\ms E}/\thicksim$ given by \Cref{pro:Bowenonesidedcohtwosided}, \Cref{cor:Livsicsuspensions} and \Cref{cor:BowengBowent}. Observe that $\Omega\left(\Bow[a][\mu]/\thicksim\right) \subset  \QMt[a][]{\Per{\lie g}}$, and moreover for $[\varphi]\in \Bow[a][\mu]/\thicksim$, $\Phi(\Omega([\varphi]))$ and $[\varphi]$ have the same integrals with respect to all periodic measures supported on points in $\Sigma_0$. Since periodic measures supported on $\Sigma_0$ determine the cohomology class (see \Cref{lem:cohpartecohtotal}), it follows that $\Phi(\Omega([\varphi]))=[\varphi]$, and the proof is complete.
\end{proof}

\begin{proof}[Proof of \Cref{thm:mainAB}]
Let $\Gamma=\pi_1(M,*)$. Apply successively the Banach isomorphisms 
\begin{itemize}
	\item $\Ker\left(c_2:H^2_b(M;\R)\to H^2(M;\R)\right)\to \frac{\QMt{\Gamma}}{\Hom{\Gamma,\R}}$ (\Cref{pro:exactsequence}),

	\item $B_{*}:\QMt{\Gamma}\to \QMt[a][]{\Per{\lie g}}$ (\Cref{thm:quasimorphismgrouptoporbit}), and 
    
    \item $\Phi: \QMt[a][]{\Per{\lie g}}\to \Bow[a][\mu]{\Sigma}/\thicksim$
    
\end{itemize}
given above to conclude the claim.
\end{proof}

\begin{remark}
	In the proof above, $[L]\in \QMt{\Gamma}$ is trivial if and only if for every periodic orbit $\alpha$ of $\mathfrak g$, $\mrm{av}_{\alpha}(\Psi([L]))=0$. This is consequence of \Cref{thm:E}.
\end{remark}

Now we consider \Cref{thm:compactification}.

\begin{proof}[Proof of the Compactification of conjugacy classes theorem]
Define $X_{\Gamma}=\Sigma(R), \ms{F}_n=\ms{B}^n$. Identify $\mc{R}(\Gamma;S)\approx \mrm{Conj}(\Gamma)$ and proceed as before to obtain a bijection
$S:\Per{\ms E_0}\to\Sigma_0$, where $\Sigma_0$ consists of a dense set of periodic points. For each $\alpha \in \Per{\ms E}\setminus \ms{E}_0$ we choose
$\bm{a}\in \Per{\ms W}$ so that the periodic orbit of $\mathfrak t$ determined by $\bm{a}$ is mapped onto $\alpha$ by $h$. Altogether, we have a injective map $\Phi:(\Gamma;S)\to X_{\Gamma}$, whose image is $\Per{\ms W}$. Note that if $w\in (\Gamma;S)$ then the period of $\Phi(w)$ is uniformly comparable with its translation length, thus with $|w|_S$. This implies condition $2$ of the Theorem.

Consider the finite measures \(\nu_n=\frac{1}{\# B_n}\sum_{w\in B_n}\delta_{w}\)  and use $\Phi$ to define corresponding periodic measures $\nu_n'\in \PTM{\tau}{X_{\Gamma}}$.  By using the results of \Cref{ssub:invariantmeasureassociated} we conclude that $(\nu_n')_n$ converges weakly to $\mu^{\Gamma}:=\MME$, the entropy maximizing measure of $\tau$, which is well known to be a Markov measure (see next Section). This finishes the proof.
\end{proof}

\begin{remark}\label{rem:dependeS}
Note that $\mu^{\Gamma}$ as given above in principle depends on $S$. However, choosing $S'$ another (symmetric) finite generating set for $\Gamma$ and carrying the construction, one sees first that the geodesic flows are equivalent, while the length of the periodic orbits is uniformly comparable. From this we get that we can use the same space $X_{\Gamma}$ for both flows, and also that the corresponding periodic measures have the same limit measure $\mu^{\Gamma}$. On the other hand, the entropy maximizing measure of the flows are typically different. 	
\end{remark}

For the proof of \Cref{thm:B} (and of its generalization to negatively curved spaces), we need some preparations. The topological pressure for the flow $\mathfrak g$ corresponding to the potential $\phi\in \Bow{\ms E}$  is the number
\[
	\Ptop(\phi):=\sup_{\mm\in \PTM{\mathfrak{g}}{\ms E}}\left(h_{\mm}(\mathfrak{g})+\mm(\phi)\right) 
\]
and we say $\nu$ is an equilibrium measure for the pair $(\mathfrak{g}, \phi)$ if the supremum is attained at $\nu$. As discussed previously, if there exists an equilibrium measure then there exist an ergodic equilibrium measure. Similar definitions hold for the suspension flow $\mathfrak t$.

As in the last part of \Cref{sec:LivsicSFT} (see \eqref{eq:potencialintegrado}) consider
\[
\lift{\phi}(\seq x)=\int_0^{f(\seq x)} \psi([\seq x,s])\dd s	
\]
where we are tacitly assuming that $\psi$ is integrable with respect to the variable $s$: this is no loss of generality, because $\psi$ is cohomologous a function satisfying this property, and equilibrium measures for cohomologous potentials are the same.

Recall that $f:\lift\Sigma\to\R$ is strictly positive. From this, and arguing identically as in Proposition $3.1$ of \cite{BowenRuelle} we obtain:

\begin{proposition}[Bowen and Ruelle]
It holds that $\Ptop(\phi)=c$, where $c$ is the unique solution to the equation $\Ptop(\lift{\phi}-c\cdot f)=0$. Moreover, $\nu$ is an equilibrium measure for $\phi$ if and only its corresponding measure $\mu$ on $\lift\Sigma$ is an equilibrium measure for $\tilde{\phi}-\Ptop(\phi)\cdot f$.	
\end{proposition}

\begin{remark}\label{rem:BRextended}
In \cite{BowenRuelle} the proof is given assuming some regularity of $\phi$. This property is only used for the uniqueness of equilibrium states for $\lift{\phi}$, which is true in our context.	
\end{remark}

We bring to the attention of the reader that \Cref{thm:A} (\Cref{thm:mainAB}) implies in particular that for every ergodic measure $\lift{\nu} \in \PTM{\mathfrak g}{\ms E}$ of full support, the spaces $\Bow[][]{\ms E}/\thicksim, \Bow[][\lift \nu]{\ms E}/\thicksim$ are Banach isomorphic; a similar statement holds for the suspension flow $\mathfrak t$.

In particular, due to \Cref{cor:BowengBowent} we get:

\begin{lemma}
There exists a Banach isomorphism $\Gamma_3:\Bow[][]{\ms E}/\thicksim \to \Bow[][]{\lift\Sigma_f}/\thicksim$.
\end{lemma}

From this we deduce:

\begin{corollary}
Let $\phi\in \Bow[][]{\ms E}, \phi'\in \Bow[][]{\lift\Sigma_f}$ so that $\Gamma_3([\phi])=[\phi']$.

Then $(\mathfrak g, \phi)$ has a unique equilibrium state if and only if $(\mathfrak t, \phi')$ has a unique equilibrium state.	
\end{corollary}

\begin{proof}[Proof of \Cref{thm:B}]
We fix some reference (ergodic) measure $\nu_{\lie g}\in \PTM{\lie g}{\ms E}$ of full support, and consider the Banach isomorphism $\Psi: \Bow[][\nu_{\lie g}]{\ms E}/\thicksim \approx \Bow[][\nu_{\lie t}]{\lift{\Sigma}_f}/\thicksim \to \Bow[][\mu]{\Sigma}/\thicksim$.

Given $\phi \in  \Bow[][\nu_{\lie t}]{\lift{\Sigma}_f}$ take $\varphi \in \Bow[][\mu]{\Sigma}$ so that $\Psi([\phi-\Ptop(\phi)f])=[\varphi]$. As discussed above, $\mu_{\varphi} \in \PTM{\tau}{\Sigma}$ is the equilibrium state of $\varphi$ if and only if its extension to $\lift\Sigma_f$ \Cref{eq:medidasuspension} is the unique equilibrium state of $\phi$. This proves uniqueness of the equilibrium state.

The Bernoulli part concerning equilibrium measures is consequence of \Cref{cor:Bernoulli}, since this implies that if $\nu_{\lie t}$ is a given equilibrium measure for a weak Bowen potential for the suspension flow, then $(\mathfrak t,\nu_{\lie t})$ is a Bernoulli flow. Since $h:(\mathfrak t,\nu_{})\to (\mathfrak g,\nu_{\lie g})$ is a metric isomorphism, we deduce the same for the later system. 
\end{proof}


\section{The Central Limit Theorem for Bowen quasimorphisms and statistics for the Patterson-Sullivan measure} 
\label{sec:the_central_limit_theorem_for_quasi_morphisms}

The last two sections are devoted to the statistical properties of quasimorphisms. In this part we prove \Cref{thm:C,thm:D}. 

We start introducing some notation.  Fix $\varphi:\Sigma\to\R$ a function with the Bowen property  and denote $\Rop[\varphi]$ its corresponding transfer operator. We assume that $\Rop[\varphi]\one=\one$,  and consider its associated stationary measure $\mu_{\varphi}$: $\mu_{\varphi}$ is invariant and is the unique equilibrium state corresponding to the system $(\varphi,\tau)$ (Cf.\@ \cite{EquSta} and compare with \Cref{cor:uniquestationary}).

We now specialize in measures associated to Markov potentials.

\begin{definition}\label{def:Markov}
	$\varphi:\Sigma\to\R$ is called a Markov potential of memory $s\geq 0$ if it is constant on cylinders of size $s+1$. The space of Markov potentials of memory $s$ is denoted by $\mrm{LC}_s$. The associated equilibrium state is called a Markov measure.
\end{definition}

\begin{example}
	The zero potential $\varphi_0\equiv 0$ is clearly a Markov potential of memory $0$. Its associated equilibrium state is $\mu=\MME$, the entropy maximizing measure of $\tau$  (also called the Parry measure of $\tau$). See \cite{Parry1964}. 

We remark however that $\varphi_0$ is not normalized, that is $\Rop[\varphi]\one\neq\one$. Instead, denoting $\lambda=e^{\htop(\tau)}$, there exists some function $u:\Sigma\to\R$ which only depends on the first coordinate, and so that 
\[
	\varphi(\seq{x})=u(\seq{x})-u(\tau \seq{x})-\htop(\tau)=u(x_0)-u(x_1)-\htop(\tau)
\]
is the (unique) normalized potential cohomologous to $\varphi_0$ (in particular, its equilibrium state is $\MME$).
\end{example}

In general, it is a consequence of the Perron-Frobenius-Ruelle theorem that any $\varphi\in\mrm{LC}_s$ is cohomologous to some normalized Markov potential $\varphi'\in \mrm{LC}_{s+1}$. With no loss of generality, we assume that $\varphi$ is a normalized Markov potential of memory $s$.

\begin{theorem}\label{thm:cohomologakernel}
	For every $\psi\in\Bow[][\mu_{\varphi}]{\Sigma}$ with $\int\psi \dd\mu_{\varphi}=0$, there is $h\in \Lp{\infty}{\mu_{\varphi}}$ such that $(\Id-\mu_{\varphi}h=\psi$.  Moreover $\norm[L^{\oo}]{h}\leq \Lambda(\varphi)\sqrt{s+1}\tnorm[B]{\psi}$, for some constant $\Lambda(\varphi)$.
\end{theorem}

\begin{corollary}\label{cor:cohomologakernel}
		For every $\psi\in\Bow[][\mu_{\varphi}]{\Sigma}$ with $\int\psi \dd\mu_{\varphi}=0$, there is $\cl{h}\in\Lp{\oo}{\mu_{\varphi}}$ with $\norm[L^\oo]{\cl h}\leq \Lambda(\varphi)\sqrt(s+1)\tnorm[B]{\psi}$ and such that $\cl{h}\circ\tau-\cl{h}-\psi\in\Ker\Rop[\varphi]$.
\end{corollary}
\begin{proof}[Proof of Corollary]
	Let $h$ be as in the theorem and $\cl{h}=\Rop[\varphi] h$, then 
	\[
   \Rop[\varphi]\paren{(\Rop[\varphi]h)\circ \tau-\Rop[\varphi]h-\psi}=(\Id-\Rop[\varphi]h)\Rop[\varphi]h-\Rop[\varphi]\psi=\Rop[\varphi]\paren*{(\Id-\Rop[\varphi])h-\psi}=0.
 	\]
\end{proof}

We now prove \Cref{thm:cohomologakernel}.

\begin{lemma}
	$\Rop[\varphi](\mrm{LC}_{s})\subset \mrm{LC}_{s}$ and $\Rop[\varphi](\mrm{LC}_{s+k})\subset \mrm{LC}_{s+k-1}$ for every $k\geq 1$. 

	In particular for $m\geq 0$, $\Rop[\varphi]^m(\mrm{LC}_{s+k})\subset \mrm{LC}_{s+k-m}$ for every $k\geq m$ and, if $k\leq m$ then $\Rop[\varphi]^m(\mrm{LC}_{s+k})\subset \mrm{LC}_{s}$. Thus, for  every $N\geq 1$ it holds $\Rop[\varphi]^{N-1}(\mrm{LC}_{N})\subset \mrm{LC}_{s}$.
\end{lemma}
\begin{proof}
Compute directly,
\begin{align*}
&\Rop[\varphi]\one_{[w_1w_2\dots w_{s}]}(\seq x)=\sum_j e^{\varphi(j\seq x)}\one_{[w_1w_2\dots w_{s}]}(j\seq x)=e^{\varphi(w_1x_0\dots x_{s-1})}\one_{[w_2\dots w_{s}]}(\seq x)
\intertext{and}
&\Rop[\varphi](\one_{[w_1w_2\dots w_{s+k}]})(\seq x)=\sum_je^{\varphi(j\seq x)}\one_{[w_1w_2\dots w_{s+k}]}(j\seq x)=e^{\varphi(w_1x_0\dots x_{s-1})}\one_{[w_2\dots w_{s+k}]}(\seq x).
\end{align*}
\end{proof}

The subspace $\mrm{LC}_{s} \subset \Lp{2}{\mu_{\varphi}}$ is closed. Let $P_s^{\perp}, P_s=\Id-P_s^{\perp}$ be the orthogonal projections of the splitting $\Lp{2}{\mu_{\varphi}}=\mrm{LC}_{s}^{\perp}\oplus \mrm{LC}_{s}$.

\begin{remark}
	We remind the reader that
	\[
	P_s(\rho)=\sum_{|\bm w|=s}\left(\frac{1}{\mu_\varphi([\bm w])}\int_{[\bm w]}\rho \dd\mu_\varphi\right)\one_{[\bm w]}= \Emu{\mu_\varphi}{\rho | \xi^s}.
	\]
\end{remark}

\begin{lemma}
	$P^{\perp}_s\circ \Rop[\varphi]\circ P^{\perp}_s =P^{\perp}_s\circ \Rop[\varphi]$.
\end{lemma}
\begin{proof}
It holds that $P_s\circ\Rop[\varphi]\circ P_s=\Rop[\varphi]\circ P_s$, since $\Rop[\varphi]$ preserves $\mrm{LC}_{s}$. Thus
 \begin{eqnarray*}
		(\Id-P_s)\circ\Rop[\varphi]\circ (\Id-P_s)&=&\Rop[\varphi]\circ (\Id-P_s)-P_s\circ\Rop[\varphi]\circ (\Id-P_s)\\
		&=&\Rop[\varphi]-\Rop[\varphi]\circ P_s-P_s\circ\Rop[\varphi]+P_s\circ\Rop[\varphi]\circ P_s\\
		&=&\Rop[\varphi]-\Rop[\varphi]\circ P_s-P_s\circ\Rop[\varphi]+\Rop[\varphi]\circ P_s\\
		&=&\Rop[\varphi]-P_s\circ\Rop[\varphi]=P^{\perp}_s\circ \Rop[\varphi].
	\end{eqnarray*}
\end{proof}

Letting $\Roph[\varphi]=P^{\perp}_s\circ \Rop[\varphi],$ we get for every $N\geq 1$,
\[
	\Roph[\varphi]^N=P^{\perp}_s\circ \Rop[\varphi]^N.
\]

\begin{lemma}
For every $N\geq 1$, $\mrm{LC}_{N}\subset\Ker\Roph[\varphi]^{N-1}$.
\end{lemma}

\begin{proof}
	Since $\Rop[\varphi]^{N-1}(\mrm{LC}_{N})\subset \mrm{LC}_{s}$, we get $\Roph[\varphi]^{N-1}(\mrm{LC}_{N})=P^{\perp}_s\Rop[\varphi]^{N-1}(\mrm{LC}_{N})=\{0\}$.
\end{proof}

It is not difficult to check that for every $1\leq p\leq \infty$, $\norm[\mrm{OP}]{\Rop[\varphi]:\Lp{p}{\mu_\varphi}\toit}=1$. On the other hand, the conditional expectation $P_s$ extends to a linear map $P_s:\Lp{p}{\mu_\varphi}\toit$ of norm one, hence we can also extend $P_s^{\perp}:=\Id-P_s: \Lp{p}{\mu_\varphi}\toit$. It follows in particular that $\norm[\mrm{OP}]{P_s^{\perp}:\Lp{\oo}{\mu_\varphi}}\leq 2$.

\begin{proposition}
	If $\psi\in \Bow[][\mu_{\varphi}]{\Sigma}$, then for every $N\geq 1$, 
	\[
	\norm[L^{\oo}]{\Roph[\varphi]^{N-1}\left(S_N\psi\right)}\leq 2\norm[B]{\psi}.
	\]
\end{proposition}

 \begin{proof}
 	Let us write 
 	\[
 	S_N\psi=\left(S_N\psi-P_N(S_N\psi)\right)+ P_N(S_N\psi)
 	\]
  and note that $\Roph[\varphi]^{N-1}(P_N(S_N\psi))=0$. On the other hand, by the Bowen property, we get that 
      \[
     \norm[L^{\oo}]{S_N\psi-P_N(S_N\psi)}{\oo}\leq \norm[B]{\psi}, 
     \]
    hence
     \begin{align*}
     \norm[L^{\oo}]{\Roph[\varphi]^{N-1}\left(S_N\psi\right)}=\norm[L^{\oo}]{P_s^{\perp}\Rop[\varphi]^{N-1}\left(S_N\psi-P_N(S_N\psi)\right)}\leq 2\norm[B]{\psi}.
 	\end{align*}
 \end{proof}

 Recall that we are denoting by $T$ the Koopman operator of $\tau$. The transference property of $\Roph[\varphi]$ now gives:

 \begin{lemma}
 It holds that for $k\geq 0$, $\Roph[\varphi]^n\circ T^k=\Roph[\varphi]^{n-k}$ for $N>k$ and $\Roph[\varphi]^N\circ T^N=P_s^{\perp}$. Hence for $N\geq 1$, 
   \[
    \Roph[\varphi]^{N-1}\left(S_N\psi\right)=P_s^{\perp}\psi+\sum_{k=1}^{N-1}\Roph[\varphi]^k\psi=P_s^{\perp}\left(\sum_{k=0}^{N-1}\Rop[\varphi]^k\psi\right).
    \]
 \end{lemma}
 \begin{proof}
 Compute
 \begin{align*}
\Roph[\varphi]^{N-1}\left(S_N\psi\right)&=\sum_{k=0}^{N-1}\Roph[\varphi]^{N-1}\left(T^k\psi\right)=P_s^{\perp}\psi+\sum_{k=0}^{N-2}\Roph[\varphi]^{N-1-k}\psi=P_s^{\perp}\psi+\sum_{k=1}^{N-1}\Roph[\varphi]^{k}\psi\\
&=P_s^{\perp}\left(\sum_{k=0}^{N-1}\Rop[\varphi]^{k}\psi\right).
 \end{align*}
\end{proof}

For $N\geq 1$ define $u_N=\Roph[\varphi]^{N-1}\left(S_N\psi\right)$: then $\Roph[\varphi] u_N=u_{N+1}-P_s^{\perp}\psi$,  and by the previous Proposition, $\norm[L^2]{u_{N}}\leq 2\norm[B]{\psi}$. Denote
\[
 	h_{M}=\frac{1}{M}\sum_{N=1}^M u_N.
\]
 By direct computation, 
 \[
 	\Roph[\varphi] h_{M}=h_{M}-P_s^{\perp}\psi-\frac{u_{N+1}-u_1}{M}.
 \]
Now take $h$ any weak-$\Lp{2}$ accumulation point of $(h_{M})_{M}$. Using that $\Roph[\varphi]$ is weakly continuous (in fact, strongly continuous), we deduce
\[
 	\Roph[\varphi]h=h-P_s^{\perp}\psi.
 \]
This tells us in particular that $h=P_s^{\perp}\left(\Rop[\varphi]h+\psi\right)$, hence $P_s^{\perp}h=h$ and so 
\[
 	P_s^{\perp}\left(\Rop[\varphi]h-h+\psi\right)=0 \qquad \Rop[\varphi]h-h+\psi\in \mrm{LC}_s.
 \]

\begin{remark}\label{rem:integralh}
By construction, $\int h\dd\mu=0$.
\end{remark}

\begin{lemma}
$\norm[L^{\oo}]{h}\leq 2\norm[B]{\psi}$.
\end{lemma}

\begin{proof}
  Indeed, for every $M\geq 1$ we have that $\norm[L^\oo]{h_M}\leq 2\norm[B]{\psi}$.  Since the weak limit of non-negative functions is non-negative, we get that the the $\Lp{\oo}$ norm of $h$ is bounded by $2\norm[B]{\psi}$.
\end{proof}

Putting everything together, we have proved:

\begin{corollary}\label{cor:cohlocconstant}
	 If $\psi\in \Bow[][\mu_{\varphi}]{\Sigma}$ then there exists $h \in \mrm{LC}_s^{\perp}$ verifying:
	 \begin{enumerate}
	 	\item $\Rop[\varphi]h-h+\psi\in \mrm{LC}_s$.
	 	\item $\norm{L^{\oo}}{h}\leq 2\norm[B]{\psi}$.
	 \end{enumerate}
\end{corollary}

\paragraph{The projection on the locally constant part}  Dealing with functions in $\mrm{LC}_s$ is essentially linear algebra. Denote $\mrm{LC}_s^0=\{\phi\in\mrm{LC}_s:\int \phi\dd\mu_\varphi=0\}$.

\begin{lemma}
	$\Rop[\varphi]:\mrm{LC}_s\toit$ is a Perron-Frobenius operator with eigenvalue $1$ and eigenfunction $1$, and so $\restr{\Id-\Rop[\varphi]}{\mrm{LC}^0_s}$ is invertible. If $\rho\in \mrm{LC}_s^0$ then
	\[
	\norml{(\Id-\Rop[\varphi])^{-1}\rho}\leq \Lambda(\varphi)\sqrt{s+1} \norml{\rho}
	\]
\end{lemma}

 \begin{proof}
The fact that $\Rop[\varphi]$ is a Perron-Frobenius operator (equivalently, an stochastic matrix on the $d\cdot (s+1)-$dimensional vector space $\mrm{LC}_s$) is direct, and thus by its spectral resolution we get the invertibility of $\Id-\Rop[\varphi]$ on the complementary of the eigenspace associated to the Perron eigenvalue. For a square matrix $A$ denote $\norm[\mrm{OP}]{A}$ its operator norm with respect to the $\ell^{\oo}$ norm on the vector space, and  $\norm[\mrm{sp}]{A}$ its spectral norm. Let $\lambda_2$ be the second largest eigenvalue of $\Rop[\varphi]$: $\lambda_2$ is strictly smaller than $1$, hence for $\rho\in \mrm{LC}_s^0$ 
\begin{align*}
 \sum_{n=0}^\oo \norml{\Rop[\varphi]^n\rho}&\leq \sum_{n=0}^\oo \norm[\mrm{OP}]{\restr{\Rop[\varphi]^n}{\mrm{LC}^0_s}}\cdot\norml{\rho}\leq \sqrt{d\cdot (s+1)} \sum_{n=0}^{\oo}\norm[\mrm{sp}]{\restr{\Rop[\varphi]^n}{\mrm{LC}^0_s}}\cdot\norml{\rho}\\
&\leq \sqrt{s+1} \frac{\sqrt{d}}{1-\lambda_2}\norml{\rho}.
 \end{align*}
Taking $\Lambda(\varphi)=\frac{\sqrt{d}}{1-\lambda_2}$ it follows that $\norml{(\Id-\Rop[\varphi])^{-1}\rho}\leq \sqrt{s+1}\Lambda(\varphi)\norml{\rho}$.
 \end{proof}

We are to prove that any $\psi\in\mrm{LC}_s^0$ is Livšic cohomologous to an element of $\Ker \Rop[\varphi]$.

\begin{proof}[Proof of \Cref{thm:cohomologakernel}]
For $\psi\in\mrm{LC}_s^0$ let $\rho=\Rop[\varphi]h-h+\psi$. Then $\rho\in \mrm{LC}_s^0$, so there is $h_1\in \mrm{LC}_s^0$ such that $h_1-\Rop[\varphi]h_1=\rho$. Letting $h_0=-h+h_1$, it follows that 
\begin{align*}
h_0-\Rop[\varphi]h_0&=-h+\Rop[\varphi]h+h_1-\Rop[\varphi]h_1=\psi-\rho+h_1-\Rop[\varphi]h_1=\psi,
\end{align*}
and moreover
\begin{align*}
\norm[L^\oo]{h_0}&\leq \norm[L^\oo]{h}+\norml{h_1}{\oo}\leq 2\norm[B]{\psi}+\Lambda(\varphi)\sqrt{s+1}\norml{\rho}\\
&\leq 2\norm[B]{\psi}+\Lambda(\varphi)\sqrt{s+1}(2\norm[L^\oo]{h}+\norm[L^\oo]{\psi})\leq 2\norm[B]{\psi}+\Lambda(\varphi)\sqrt{s+1}\left(4\norm[B]{\psi}+\norm[L^\oo]{\psi}\right)\\
&\leq \left(2+4\Lambda(\varphi)\sqrt{s+1}\right)\norm[B]{\psi}+\Lambda(\varphi)\sqrt{s+1}\norm[L^\oo]{\psi}\\
&\leq 5\Lambda(\varphi)\sqrt{s+1}\tnorm[B]{\psi}.
\end{align*}
\end{proof}

\paragraph{Properties of the variance}  At this stage we can use some classical machinery initially developed by Gordin to establish the CLT for dynamical system, based mostly on the CLT for martingales differences of Billingsley. We rely on Brown's version \cite{BrownCLT}; a thorough exposition is given in the book by Hall and Heyde \cite{MartingaleLimit}. To simplify the notation now $\mu$ denotes a fixed Markov measure, and 
$\Rop=\Rop[\varphi]$ the corresponding transfer operator.

Fix a locally constant quasicocycle $L=(L^{(n)})_n$, and consider its associated weak Bowen function $\varphi_L=\Bow[][\mu]{\Sigma}$ as given by \Cref{pro:potentialarbitrary}. Denote $\psi_L=\varphi_L-\mu(L)$, and we remind the reader that changing $L$ by a cohomologous cocycle has the effect of changing $\psi_L$ by a Livšic cohomologous potential, which also has the weak Bowen property. Define
\begin{align}\label{eq:variance}
 \sigma^2(L)=\limsup_n\frac{1}{n}\norm[L^2]{S_n\psi_{L}}^2
 \end{align}

 From the previous line one gets that $\sigma^2$ is constant on cohomology classes. Let $\lift{\psi}_L\in \Ker \Rop$  be as in \Cref{cor:cohomologakernel}, hence $\psi_L=\lift{\psi}_L+v-v\circ \tau$ for some $v\in\Lp{\oo}{\mu}$. It follows that there exists some constant $E>0$ independent of $N$ so that $\aee{\mu}$,
\begin{align}\label{eq:Lcohkernel}
|L^{(N)}(\seq x)-S_N\lift{\psi}_L-N\mu(L)|\leq E.
 \end{align}

On the other hand, for $\psi\in \Lp{2}{\mu}$
 \begin{align}\label{eq:sumaNvar}
 \frac{\norm[L^2]{S_N\psi}^2}{N}=\norm[L^2]{\psi}^2+2\sum_{n=1}^{N-1}\left(1-\frac{n}{N}\right)\inner{\psi\circ \tau^n}{\psi}
\end{align}
hence for $\psi=\lift\psi_L$ we obtain
\begin{align*}
 \frac{\norm[L^2]{S_N\psi}}{N}=\norm[L^2]{\lift\psi_L}^2 
 \end{align*}
and
 \[
 	\sigma^2(L)=\norm[L^2]{\lift \psi_L}^2. 
 \]

\begin{corollary}\label{eq:sigmaL}
If $L$ verifies $\mu(L)=0$, then $\sigma(L)=0$ if and only if $L$ is cohomologically trivial.
\end{corollary}

 \begin{proof}
 Indeed, $\sigma(L)=0$ if and only if $\varphi_L=\psi_L=v-v\circ \tau$ in $\Lp{2}(\mu)$. Since $\psi_L$ has the weak Bowen property, this holds if and only if $\psi_L$ is a $\Lp{\oo}-$coboundary with $\Lp{\oo}$ transfer function (\Cref{thm:livsicWeakBowen}). This implies the claim.
 \end{proof}

Before moving further, we establish some properties of the variance, which are of interest. We have that $\psi_L=h-\Rop h$, where $h\in \Lp{\oo}(\mu)$. Recall (\Cref{rem:integralh}) that $\int h\dd\mu=0$.

\begin{lemma}
For every $u\in \Lp{\oo}{\mu_\varphi}$ with $\int u\dd\mu_\varphi=0$ it holds $\norm[L^\oo]{\Rop^n u} \xrightarrow[n\mapsto\oo]{}0$.
\end{lemma}

 	\begin{proof}
 		Indeed, the uniformly bounded sequence of operators $\{\Rop^n:\Lp{\oo}_0(\mu)\to \Lp{\oo}_0(\mu)\}$ converges pointwise to zero for every continuous function $u\in \Lp{\oo}_0(\mu)=\{v\in \Lp{\oo}{\mu}, \mu(u)=0\}$, therefore everywhere.  
 	\end{proof}

 The above permits us to identify the function $h$ in $\Lp{p}{\mu}$.

 \begin{lemma}
 For every $1\leq p\leq \oo$, $h \eqin{\Lp{p}}\sum_{n\geq 0}\Rop^n \psi_L$, i.e. $\norm[L^p]{h-\sum_{n=0}^{N-1}\Rop^n\psi_L} \xrightarrow[N\mapsto\oo]{}0$.
 \end{lemma}

 \begin{proof}
 \begin{align*}
 h-\sum_{n=0}^{N-1}\Rop^n\psi_L &=h-\sum_{n=0}^{N-1}\Rop^n(h-\Rop^n h)=h-\left(\sum_{n=0}^{N-1}\Rop^n h-\sum_{n=1}^{N}\Rop^n h\right)\\
 &= h-\left(h-\Rop^N h\right)=\Rop^N h \xrightarrow[N\mapsto\oo]{\Lp{p}{\mu}} 0.
 \end{align*}
 \end{proof}

 In particular we get:
 \begin{corollary}
	For every $\rho\in \Lp{1}{\mu}$, it holds that $(\int \psi_L\cdot\rho\circ \tau^n \dd\mu)_n$ is summable and the sum is 
\[
 	\sum_{n\geq 0}\int\psi_L\cdot \rho\circ \tau^n \dd\mu=\int h\rho \dd\mu.
 \]
 \end{corollary}

Next we get a formula for the variance.

 \begin{corollary}\label{cor:sigmaformula}
 It holds that $\sigma^2_{\mu}(\psi_L)=\norm[L^{2}]{\psi_L}^2+2\sum_{n\geq 1}\int\psi\cdot\psi\circ \tau^n \dd\mu$.
 \end{corollary}

 \begin{proof}
 Compute
 \begin{align*}
 \sigma^2_{\mu}(\psi_L)&= \norm[L^2]{\cl\psi_L}^2=\norm[L^2]{\psi+\Rop h-\left(\Rop h\right)\circ \tau}^2=\norm[L^2]{h-\left(\Rop h\right)\circ \tau}^2\\
&=\norm[L^2]{h}^2+\norm[L^2]{\left(\Rop h\right)\circ \tau}^2-2\inner{h}{\left(\Rop h\right)\circ \tau}\\
&=\norm[L^2]{h}^2+\norm[L^2]{\Rop h}^2-2\inner{\Rop h}{\Rop h}=\norm[L^2]{h}^2-\norm[L^2]{\Rop h}^2\\
&=\norm[L^2]{\psi_L+\Rop h}^2-\norm[L^2]{\Rop h}^2\\
&=\norm[L^2]{\psi_L}^2+2\inner{\psi_L}{\Rop h}=\norm[L^2]{\psi_L}^2+2\sum_{n\geq 1} \inner{\psi_L}{\Rop^n\psi_L}\\
&=\norm[L^2]{\psi_L}^2+2\sum_{n\geq 1} \inner{\psi_L}{\psi_L\circ \tau^n}.
 \end{align*}
 \end{proof}

We also deduce the usual formula:

 \begin{lemma}
$\frac{1}{n}\norm[L^2]{S_n\psi_L}^2 \xrightarrow[n\mapsto\oo]{}\sigma^2_{\mu}(\psi_L)=\norm[L^2]{\cl \psi_L}^2$.
 \end{lemma}

\begin{proof}
 For every $n$,
 \[
 	S_n\psi_L=S_n\cl{\psi}_L+\left(\Rop h\right)\circ\tau^n-\Rop h,
 \]
so
 \begin{align*}
 \norm[L^2]{S_n\psi_L}^2&=\norm[L^2]{S_n\cl{\psi}_L}^2+\norm[L^2]{\left(\Rop h\right)\circ\tau^n-\Rop h}^2+2\inner{S_n\cl{\psi}_L}{\left(\Rop h\right)\circ\tau^n-\Rop h}\\
&=n\norm[L^2]{\cl{\psi}_L}^2+\norm[L^2]{\left(\Rop h\right)\circ\tau^n-\Rop h}^2+2\inner{S_n\cl{\psi}_L}{\left(\Rop h\right)\circ\tau^n}-2\inner{S_n\cl{\psi}_L}{\Rop h}\\
&=n\norm[L^2]{\cl{\psi}_L}^2+\norm[L^2]{\left(\Rop h\right)\circ\tau^n-\Rop h}^2-2\inner{S_n\cl{\psi}_L}{\Rop h}
 \end{align*}
since $\cl{\psi}_L\in\Ker \Rop$ implies that $\Rop^n\left(S_n\cl{\psi}_L\right)=0$. Now, on the one hand, we have that $\norm[L^2]{\left(\Rop h\right)\circ\tau^n-\Rop h}^2\leq 2\norm[L^2]{\Rop h}^2$.  Since $\int\cl{\psi}_L \dd\mu=0$, by Von Neumann ergodic theorem the sequence $\frac{1}{n}S_n\cl{\psi}_L$ goes to zero in $\Lp{2}{\mu}$ and hence we get that $\frac{1}{n}\inner{S_n\cl{\psi}_L}{\Rop h} \xrightarrow[n\mapsto\oo]{}0$, so the lemma follows.
 \end{proof}

\begin{corollary}
It holds 
\[
	\frac{1}{N}\sum_{n=0}^{N-1}k\int\psi_L\cdot \psi_L\circ\tau^k\dd\mu\xrightarrow[N\mapsto\oo]{} 0
\]
and hence $(n\int\psi_L\cdot\psi_L\circ\tau^n\dd\mu)_{n\geq 0}$ tends to $0$ in the Ces\`aro sense. 
\end{corollary}

\begin{proof}
By \Cref{cor:sigmaformula} and \eqref{eq:sumaNvar} we get
\begin{align*}
\sigma^2(\psi_L)=\norm[L^2]{\psi_L}^2+2\sum_{n\geq 1}\inner{\psi_L}{\psi_L\circ \tau^n}=\lim_{N\to \oo}\frac{\norm[L^2]{S_N\psi}^2}{N}\\
=\lim_{N\to\oo}\left(\norm[L^2]{\psi_L}^2+2\sum_{n=1}^{N-1}\left(1-\frac{n}{N}\right)\inner{\psi_L\circ \tau^n}{\psi_L}\right),
\end{align*}
which proves the Lemma.
\end{proof}

\begin{remark}
Convergence of the series $\sum_{n\geq 1} n\int\psi_L\cdot\psi_L\circ\tau^n\dd\mu$ remains to be proven. If this were the case one would get that $(S_n\psi_L-n\sigma^2(\psi_L))_{n\geq 1}$ is uniformly bounded in $\Lp{2}{\mu}$.
\end{remark}

\paragraph{Cohomology to a martingale difference} After this preparations we move to the proof of the Central Limit Theorem. Assume, with no loss of generality, that $\mu(L)=0$ and $\sigma:=\sigma(L)\neq 0$. Since $0=\Rop\psi_L(\tau \seq x)=\Emu{\mu}{\psi_L|\tau^{-1}\BorelM[\Sigma]}(\seq x)$, it follows that $\left(S_n\psi_L,\tau^{-n}\BorelM[\Sigma]\right)_{n\geq0}$ is a reverse Martingale, therefore by the Martingale Central Limit Theorem \cite{BrownCLT} we get

\[
	\mu\left(\frac{S_n\psi_L}{\sigma\sqrt{n}}\leq c\right)\xrightarrow[n\mapsto\oo]{} \frac{1}{\sigma\sqrt{2\pi}}\int_{-\oo}^c e^{-\frac{u^2}{2\sigma}} \dd u. 
\]

\begin{lemma}\label{lem:CLTqm}
It holds
\[
	\mu\left(\frac{L^{(n)}}{\sigma\sqrt{n}}\leq c\right)\xrightarrow[n\mapsto\oo]{} \frac1{\sigma\sqrt{2\pi}}\int_{-\oo}^c e^{-\frac{u^2}{2\sigma}} \dd u 
\]
\end{lemma}

\begin{proof}
Fix $\eps>0$ and write $r_n=\frac{L^{(n)}(\seq x)-S_n\psi_L}{\sigma\sqrt n}$, so that $\norm[L^2]{r_n}\leq \frac{E}{\sqrt{n}}$, and therefore
\begin{align*}
\mu\left(r_n>\epsilon\right)\leq \frac{E^2}{\epsilon n}. 
\end{align*}

Then for $c\in\R$ fixed,
\begin{align*}
 \mu\left(\frac{L^{(n)}}{\sigma\sqrt n}\leq c\right)&\leq  \mu\left(\frac{S_n\psi_L}{\sigma\sqrt n}\leq c-\eps\right)+\mu(r_n>\eps)\\
& \xrightarrow[n\mapsto\oo]{} \frac1{\sigma\sqrt{2\pi}}\int_{-\oo}^{c-\epsilon} e^{-\frac{u^2}{2\sigma}} \dd u\\
&\Rightarrow \limsup_n\mu\left(\frac{L^{(n)}}{\sigma\sqrt n}\leq c\right)\leq  \frac1{\sigma\sqrt{2\pi}}\int_{-\oo}^c e^{-\frac{u^2}{2\sigma}} \dd u
 \intertext{and likewise}
 &\liminf_n\mu\left(\frac{L^{(n)}}{\sigma\sqrt n}\leq c\right)\leq  \frac1{\sigma\sqrt{2\pi}}\int_{-\oo}^c e^{-\frac{u^2}{2\sigma}} \dd u
 \end{align*}
which implies the claim.
\end{proof}

By the same argument one obtains:

\begin{lemma}\label{lem:cohomologyCLT}
If $\psi_L\sim \psi_L'$, then for all $c\in\R$
\[
	\mu\left(\frac{S_n\psi_L'}{\sigma\sqrt{n}}\leq c\right)\xrightarrow[n\mapsto\oo]{} \frac{1}{\sigma\sqrt{2\pi}}\int_{-\oo}^c e^{-\frac{u^2}{2\sigma}} \dd u. 
\]
\end{lemma}

Next consider the natural extension $\pi:\lift\Sigma\to\Sigma$. The unique lift of $\mu$ to $\Sigma$ is the equilibrium state of $\varphi\circ \pi$, and will be denoted as $\lift \mu$

\begin{lemma}\label{lem:CLTtwosidedshift}
Let $\psi \in \Bow[][\lift\mu]{\lift\Sigma}$ be a function with zero integral. Then for all $c\in\R$,
\[
	\lift{\mu}\left(\frac{S_n\psi}{\sigma\sqrt{n}}\leq c\right)\xrightarrow[n\mapsto\oo]{} \frac{1}{\sigma\sqrt{2\pi}}\int_{-\oo}^c e^{-\frac{u^2}{2\sigma}} \dd u. 
\]
\end{lemma}

\begin{proof}
Using \Cref{pro:Bowenonesidedcohtwosided} and the previous Lemma one gets that it is enough to prove the same for $\psi_L$ only depending on positive coordinates. This is checked directly, remembering that convergence in distribution for a sequence $(\mm_n)_n \in \ProbM[\R]$ to $\mm$ is equivalent to convergence of $(\mm_n(k))_n$ to $\mm(k)$, for every bounded continuous function $k:\R\to\R$. 	
\end{proof}

\begin{proof}[Proof \Cref{thm:D} and \Cref{cor:C}]

We use the \Cref{thm:compactification}, where $X_{\Gamma}=\Sigma$ and consider $\mu=\MME$. \Cref{cor:C} was proven in \Cref{lem:CLTqm}.  As for \Cref{thm:D}, we write	
\begin{align*}
 &S_n=S_n\psi_L\\
 &\lie{L}_t(n,x)=\frac{S_{([nt])}(x)+(nt-[nt])\psi_L(\tau^{[nt]+1})}{\sigma \sqrt{n}}
 \end{align*}
Then $(S_n)_{n\geq 1}$ is an stationary ergodic sequence cohomologous to $(L^{(n)})_{n\geq 1}$. As noted above $(S_n, \tau^{-n}\BorelM[\Sigma])$ is a reverse martingale with uniformly bounded differences, hence Donsker invariance principle and the Law of iterated logarithm holds for this sequence. See \cite{BrownCLT} for the first part and Corollary $4.1$ in \cite{MartingaleLimit} for the second. This establishes the first two parts of \Cref{thm:D}. 

The last part is direct consequence of the concentration inequalities for uniformly bounded Martingale differences. See for example Corollary $2.4.7$ in 
 \cite{Dembo_2010}.
 \end{proof}

\paragraph{CLT for flows} We now address flows.

\begin{notation}
We denote by $\nu_{\mathfrak t}, \nu_{\mathfrak g}$ the corresponding measure on $\lift{\Sigma}_f, \ms{E}$ for the flows $\mathfrak t, \mathfrak g$.
\end{notation}

\begin{definition}\label{def:Markovmeasuregeodesic}
An invariant measure $\nu_{\lie g}$ for $\lie g$ is said to be Markov if the corresponding measure on the symbolic model is a Markov measure.	
\end{definition}

 Since we working with $\mu=\MME$, $\nu_{\mathfrak g}$ is Markov. If $\psi\in \Bow[][\lie t]{\lift{\Sigma}_f}, T>0$  then 
 \[
 	\int_0^T \psi([\seq x,t]) \dd t = S_n\lift\psi(\seq x)+\int_n^T \psi([\seq x,t]) \dd t
 \]
 where $n$ is the unique natural number so that $S_nf(\seq x)\leq T <S_{n+1}f(\seq x)$, and $\lift{\psi}(\seq x)=\int_0^{f(\seq x)} \psi([\seq x,t]) \dd t$ (cf.\@ \eqref{eq:potencialintegrado}).

Due to \Cref{lem:CLTtwosidedshift} we deduce the CLT for $\nu_{\lie t}-$Bowen functions.

\begin{corollary}\label{cor:CLTsuspension}
If $\psi\in  \Bow[][\lie t]{\lift{\Sigma}_f}$ is not a coboundary, then there exists some $\sigma^2>0$ so that
	\[
    \forall c\in\R,\ \nu_{\mathfrak t}\left([\seq x,s]:\frac{\int_0^T\psi([\seq x,s+t])\dd t-T\nu_{\lie t}(\psi)}{\sigma\sqrt{T}}\leq c\right)\xrightarrow[T\to\oo]{}\frac{1}{\sigma\sqrt{2\pi}}\int_{-\oo}^c e^{-\frac{(u-\gamma)^2}{2\sigma}}\dd u
\]
\end{corollary}

\begin{proof}
With no loss of generality assume $\nu_{\mathfrak t}(\psi)=0$: let $\sigma^2=\sigma^2(\lift\psi)>0$. Note that the hypothesis and \Cref{lem:Livsicreturnmap} imply that $\tilde{\psi}$ is not a coboundary. 

For $\eps>0$, and arguing analogously as in \Cref{lem:cohomologyCLT}, we get that
 \begin{align*}
 \limsup_{T\to\oo} \nu_{\mathfrak t}\left([\seq x,s]:\frac{\int_0^T\psi([\seq x,s+t])\dd t}{\sigma\sqrt{T}}\leq c\right)\leq \limsup_{n\to\oo} \lift{\mu}\left([\seq x,s]:\frac{S_n\psi}{\sigma\sqrt{n}}\leq c-\epsilon\right), 
 \end{align*}
 hence by \Cref{lem:CLTtwosidedshift}
 \begin{align*}
 	\limsup_{T\to\oo} \nu_{\mathfrak t}\left([\seq x,s]:\frac{\int_0^T\psi([\seq x,s+t])\dd t}{\sigma\sqrt{T}}\leq c\right)\leq\frac{1}{\sigma\sqrt{2\pi}}\int_{-\oo}^c e^{-\frac{(u-\gamma)^2}{2\sigma}}\dd u.
 \end{align*}

 Similarly, $\liminf_{T\to\oo} \nu_{\mathfrak t}\left([\seq x,s]:\frac{\int_0^T\psi([\seq x,s+t])\dd t}{\sigma\sqrt{T}}\leq c\right)\geq\frac{1}{\sigma\sqrt{2\pi}}\int_{-\oo}^c e^{-\frac{(u-\gamma)^2}{2\sigma}}\dd u$.
\end{proof}

\begin{proof}[Proof of \Cref{thm:C}]
 	Recall that the semi-conjugacy $h:\lift{\Sigma}_f\to \ms E$ induces an isomorphism $h^{*}:\Bow[][\nu_{\lie g}]{\ms E}/\thicksim\to \Bow[][\nu_{\lie t}]{\lift{\Sigma}_f}/\thicksim$.  Fix $\psi\in \Bow[][\nu_{\lie g}]{\ms E}$ that is not a coboundary and verifies $\nu_{\mathfrak g}(\psi)=0$. Let $\sigma^2=\sigma^2(h^*[\psi])$ and denote
	\begin{align*}
 	 &X_T(x)=\frac{S_T^{\mathfrak g}\psi(x)}{\sigma\sqrt{T}}\\
 	 &Y_T([\seq x,t])=\frac{S_T^{\mathfrak t}\psi\circ h([\seq x,t])}{\sigma\sqrt{T}}
	 \end{align*}

  Since $h(g_t)=\tau_t(h)$, $h\nu_{\mathfrak g}=\nu_{\mathfrak t}$, we get $X_T\nu_{\mathfrak g}=Y_T\nu_{\mathfrak t}$, while by the Corollary above we know that $(Y_T\nu_{\mathfrak t})_{T>0}$ converges in distribution to $\mc N(0,\sigma)$. This finishes the proof.
 \end{proof}

\begin{remark}
It would be interesting to establish the CLT for Bowen functions, when the equilibrium state is associated to non-Markov potentials. Particularly, the case of the potential $\varphi:\lift{\Sigma}_f\to\R, \varphi=-\htop(\mathfrak g)f$ corresponds to the entropy maximizing measure of $\mathfrak g$. When $M$ is a closed hyperbolic manifold this is the same as the Liouville measure on $T^1M$.
\end{remark}


\subsection{CLT for quasimorphisms with respect to the Patterson-Sullivan measure and spherical means}

We now reap some corollaries of the previous work, and obtain asymptotic information for unbounded quasimorphisms $[L]\in \QMt{\Gamma}$ with respect to the Patterson-Sullivan measure $\nu\in \ProbM[\partial\Gamma]$. To do this, we combine the general limit theorems of the previous part with techniques originally developed by Calegari-Fujiwara \cite{CALEGARI2009}, and further developed by Cantrell \cite{Cantrell2020,Cantrell_2020} (see also  Cantrell-Sert \cite{Cantrell_2023}) to study regular (e.g. combable) quasimorphisms. The content of this part was suggested to us by S. Cantrell.

The latter three articles even give Barry-Esseen type inequalities. Here we extend some of their results to arbitrary unbounded quasimorphisms, but we lose information on the speed of convergence. We will be following mainly \cite{Cantrell_2020}, where the computations are written in terms of uniform asymptotic estimates which are not present in our context. For convenience of the reader, we sketch the relevant parts in our context.

Take $\Gamma$ a non-elementary hyperbolic group (say, $\Gamma=\pi_1(M,*)$ where $M$ is a closed locally $\CAT$ space), and fix $S$ a finite symmetric generating set. We will use a different type of symbolic model for $\Gamma$ and $\partial\Gamma$.

\begin{theorem}[Cannon coding, \cite{GhysHarpe}]
	There exists a finite graph $(V,E)$ (vertices, edges) and $l:E\to S$ (the labeling map) so that
	\begin{enumerate}
		\item there exists an initial vertex $*\in V$;
		\item Finite paths $*\mapsto x_1\mapsto x_2\mapsto \cdots \mapsto x_n$ in the graph are in bijection with finite words in $S$ of the form
		\begin{align*}
		&l(*,x_1)l(x_1, x_2)\cdots l(x_{n-1}x_n)\\
		&|l(*,x_1)l(x_1, x_2)\cdots l(x_{n-1}x_n)|_S=n.
		\end{align*}
	\end{enumerate}
\end{theorem}

In order to be able to represent elements of $\Gamma$ by finite words in $\Sigma(A)$, it is convenient to add and additional vertex $0\not\in V_S$ together with vertices $(v,0)$ for every $v\in V_S\setminus\{*\}$: the function $l$ is extended as $l(v,0)=1_{\Gamma}$. Consider the resulting graph $\mc G=(\mc V,\mc E)$, and let $A\in\Mat_{|\mc V|}(\{0,1\})$ be its adjacency matrix. We consider $(\Sigma(A),\tau)$ the corresponding dynamical system: note however that this is not irreducible in our definition. On the other hand, if $\lambda=e^{\htop(\tau)}$ is the spectral radius of $A$, then there exist square sub-matrices $B$ (of $A$) with spectral radius equal to $\lambda$, and each such sub-matrix verifies that $\tau|:\Sigma(B) \subset \Sigma(A)\to \Sigma(B)$ is transitive. Denote by $\{B_1, \cdots, B_k\}$ all these transitive components. We remark that $\tau|:\Sigma(B_i)\to \Sigma(B_i)$ has a unique measure of maximal entropy $\mu_i$, which is a Markov measure.

There is a natural embedding $\iota: \Gamma\to\Sigma(A)$, $\iota(g)=(*,x_1,\cdots ,x_n, 0,0,\cdots)$ where
\begin{align*}
&g=l(*,x_1)l(x_1, x_2)\cdots l(x_{n-1}x_n)\\
&|l(*,x_1)l(x_1, x_2)\cdots l(x_{n-1}x_n)|_S=n.
\end{align*}
Also, if $A'$ is the matrix obtained from $A$ by removing the row and column corresponding to $v=0$, and $Y=\{\seq x\in\Sigma_{A'}:x_0=*\}$, then there is a bijection $h:Y\to \partial\Gamma$. If $n\in\N, \seq x\in Y$, we write $h(\seq x)_n$ for the initial segment of size $n$ of the geodesic ray $h(\seq x)$.

\begin{proposition}[\cite{CALEGARI2009}]
	There exists a unique measure $\hat \nu$ on $\Sigma(A')$ so that $h\hat \nu=\nu$. In particular, $\hat\nu$ is supported on $Y$.
\end{proposition}

Let 
\[
	\mu_N =\frac{1}{N}\sum_{k=0}^{N-1} \tau^k\hat\nu. 
\]

\begin{proposition}[Lemma $4.2$ and Proposition $4.2$ in \cite{Cantrell_2020}]
	The sequence $(\mu_N)$ converges to a measure $\mu$, which is a convex combination of the measures $\mu_1,\cdots, \mu_k$.
\end{proposition}

This measure $\mu$ is supported on $X=\bigcup_i^k\Sigma(B_i)$. The idea is then comparing the measure $\hat\nu$ with the limit measure $\mu$: one of the main difficulties to overcome is that these measure are supported in different sets of $\Sigma(A)$.

\smallskip 
 
Fix $0\neq [L]\in \QMt{\Gamma}$, which with no-loss of generality we assume to be homogeneous. The Patterson--Sullivan can be obtained as
\begin{align*}
\nu=\lim_n \nu_n,\quad \nu_n=\frac{1}{\# S_n}\sum_{g\in B_n}\delta_g
\end{align*}
where $B_n=\{g\in\Gamma:|g|_S\leq n\}$ \cite{PsmeasureCoor}. We use $\iota$ to lift the measures $\nu_n$ to measures $\eta_n$ on $Y$: it is direct to check that $(\eta_n)_n$ converges weakly to $\hat \nu$. Similarly, by identifying $Y=\Sigma(A')$ one uses $L$ to define $L\in \QM{Y}$, and considers  $\bm{L}=\{L^{(n)}:Y\to\R\}_n$ the corresponding locally constant quasicocycle. 

Since $B_n$ is symmetric and $L$ is homogeneous, it follows that for every $n, m\in \N$, $\eta_{n+m}(L^{(n)})=0$, and thus 
$\hat\nu(L^{(n)})=0$. Similarly, for every $j\geq 0$, $\hat\nu(L^{(n)}\circ \tau^j)=0$. From this we deduce that $\mu(L^{(n)})=0$, and since $\mu$ is invariant, $\mu(\bm L)=0$, which in turn implies that for every $1\leq i\leq k$, $\mu_i(\bm L)=0$ as well.

For $\sigma\neq 0$ denote $\mc N(0,\sigma)$ the centered normal distribution function with variance $\sigma^2$. As consequence of the results in the previous part we get:

\begin{proposition}\label{pro:CLTmui}
  There exists $\sigma^2=\sigma^2(L)>0$ so that for every $c\in\R$, 
	\[
  \frac{L^{(n)}(\seq x)}{\sigma\sqrt{n}} \xrightarrow[n\mapsto\oo]{\ms D} \mc N(0,\sigma)   
	\]
\end{proposition}

\begin{proof}
Applying \Cref{lem:CLTqm} to each component $\Sigma_i$ one deduces one of the two possibilities: either $L^{(n)}$ is bounded on $\Sigma(B_i)$, or $\left(\frac{L^{(n)}(\seq x)}{\sigma\sqrt{n}}\right)_n$ converges in distribution to $\mc N(0,\sigma_i)$, for some $\sigma_i$. If $\sigma_i$ is zero, then one can show that the set $\{[r]\in\partial \Gamma: \sup_{n\geq \N}|L(r_n)|<\oo\}$ has positive $\nu-$measure, hence by ergodicity of $\nu$ it implies that this set has full measure. A posteriori, $L$ is bounded, which contradicts what we assumed. See for example Proposition $5.6$ in  \cite{Cantrell2020}. This method is due to Calegari and Fujiwara \cite{CALEGARI2009}, where the ergodicity is $\nu$ is used further to guarantee that all $\sigma_i$ coincide (in page $17$ of \cite{Cantrell2020} substitute $f^n$ by $L^{(n)}$), hence $\sigma_i=\sigma>0$.   
\end{proof}

To continue, Proposition $4.6$ in \cite{Cantrell_2020} shows that $\norm[\mrm{TV}]{\mu_N-\mu}=O(N^{-1})$, therefore defining
\[
	E_n(c,X)=\left(\seq x\in X :\frac{L^{(n)}(\seq x)}{\sigma\sqrt{n}}\leq t \right),
\]
one gets:

\begin{corollary}\label{cor:cltX}
    For $n\in\N, c\in\R$, 
	\[
	\left|\mu_N(E_n(c,X))-\frac{1}{\sigma\sqrt{2\pi}}\int_{-\oo}^c e^{-\frac{u^2}{2\sigma}} \dd u\right|\leq \xi_n(c)+O(N^{-1}).
	\]
	where $\lim_{n\to\oo} \xi_n(t)=0$.
\end{corollary}

The following measures are introduced to compare $\mu_N(E_n(c,X))$ with $\hat\nu(E_n(c,X))$. Consider the sets
\begin{align*}
&Y_i=\{y\in Y: \exists m\in\N\text{ such that }\forall n\geq m, \tau^ny\in \Sigma(B_i)\}\\
&\widetilde{Y}=\bigcup_i Y_i\\
&E_n(c,\widetilde{Y})=\left\{\seq x\in \widetilde{Y}: \frac{L^{(n)}(\seq x)}{\sigma\sqrt{n}}\leq c\right\}.
\end{align*}

\begin{definition}
	For $j\in\N$ let $A_j=\{\seq x\in\Sigma(A'):\tau^ky\notin X \forall k=0,\ldots, j-1\text{ and } \tau^j\seq{x}\in X\}$. For $n\geq 1$ define the measures
	\[
	\hat\nu_n(E)=\hat\nu(E\cap \bigcap_{j=0}^n A_j).
	\]
\end{definition}

Observe that if $E \subset X$ then $\tau^j\hat\nu(E)=\tau^j\hat\nu_j(E)$.

\begin{lemma}[Lemma $5.8$ in \cite{Cantrell_2020}]
	For every sequence $(N_n)$, 
	\[
	\frac{1}{N_n}\sum_{j=0}^{N_n-1}\hat\nu_j\left(E_n(t,\widetilde{Y})\right)=\hat{\nu}\left(E_n(t,\widetilde{Y})\right)+O(N_n^{-1})
	\]
\end{lemma}

Fix $N_n=[\sqrt[4]{n}]$ and let $C>0$ so that $\norml{L^{(n)}}\leq Cn$. Define
\begin{align*}
C_n^{\pm}(c,X)=\left\{\seq x\in X: \frac{L^{(n)}(y)}{\sigma\sqrt{n}}\leq c\pm \frac{C}{\sqrt[4]{n}}\right\}
\end{align*}
Then for $0\leq j\leq N_n$, 
\[
	\tau^j\hat\nu(C_n^-(c,X))= \tau^j\hat\nu_j(C_n^-(c,X))\leq \hat\nu_j(E_n(c,\widetilde{Y}))\leq \tau^j\hat\nu_j(C_n^+(c,X))=\tau^j\hat\nu(C_n^+(c,X))
\]
and
\[
	\mu_{N_n}(C_n^-(c,X))\leq \hat{\nu}\left(E_n(c,\widetilde{Y})\right)+O(n^{-\frac{1}{4}})\leq \mu_{N_n}(C_n^+(c,X)).
\]

\begin{corollary}
For every $c\in\R$, $\lim_n \hat{\nu}\left(E_n(c,\widetilde{Y})\right)=\frac{1}{\sigma\sqrt{2\pi}}\int_{-\oo}^c e^{-\frac{u^2}{2\sigma}} \dd u$.
\end{corollary}

\begin{proof}
Direct consequence of \Cref{cor:cltX} and the previous inequality.	
\end{proof}

We are ready to prove \Cref{cor:D}, starting with its first part.

\begin{theorem*}
It holds
\[
	\nu\left([\tilde r]:\exists r\in [\tilde \ga]\text{ with }r_0=1_{\Gamma}\text{ and } \frac{L(r_n)}{\sigma^2\sqrt{n}}\leq c\right)\xrightarrow[n\mapsto\oo]{}\frac{1}{\sigma\sqrt{2\pi}}\int_{-\oo}^c e^{-\frac{u^2}{2\sigma}} \dd u.
\]
\end{theorem*}

\begin{proof}
	Indeed, denote $\hat{R}_n(c)$ the measure on the left hand side, and 
\[
	H_n(c)=\hat\nu\left(\left\{\seq x\in \widetilde{Y}: \frac{L(h(\seq x)_n)}{\sigma\sqrt{n}}\leq c\right\}\right).
\]
By the previous Corollary, $\lim_n H_n(c)=\frac{1}{\sigma\sqrt{2\pi}}\int_{-\oo}^c e^{-\frac{u^2}{2\sigma}} \dd u$. On the other hand, Lemma $5.7$ in  \cite{Cantrell_2020} shows the existence of some $K>0$ so that for every $n\geq 1, c\in\R$: $H_n(c)\leq \hat{R}_n(c)\leq H_n(c+\frac{K}{\sigma\sqrt n})$. As a consequence, $\lim_n \hat{R}_n(c)=\frac{1}{\sigma\sqrt{2\pi}}\int_{-\oo}^c e^{-\frac{u^2}{2\sigma}} \dd u$, which is what we wanted to show.
\end{proof}

The second  part of \Cref{cor:D} follows from the same skeleton of reasoning as above, this time following section $8$ of \cite{Cantrell_2023}, which is based on \cite{Gekhtman_2022}. The only difference is that one substitutes Theorem $3.8$ of that article by \Cref{pro:CLTmui}. We refer the reader to the above cited works for more details.


\printbibliography[title={References}]
\end{document}